\documentclass[a4paper,11pt]{article}

\usepackage[latin1]{inputenc}
\usepackage[T1]{fontenc}
\usepackage{lmodern}
\usepackage{graphicx,tikz} 

\usepackage{natbib}
\usepackage{setspace}
\renewcommand{\theenumi}{\roman{enumi}}
\usepackage{enumitem}\setlist[enumerate]{label={\rm(\theenumi)}}

\usepackage{amsmath,amssymb,amsthm}
\usepackage{mathtools} 
\usepackage{mathrsfs} 

\numberwithin{equation}{section}

\theoremstyle{plain}
\newtheorem{theorem}{Theorem}[section]
\newtheorem{proposition}[theorem]{Proposition}

\newtheorem{lemma}[theorem]{Lemma}
\newtheorem{assumption}[theorem]{Assumption}

\theoremstyle{definition}
\newtheorem{definition}[theorem]{Definition}
\newtheorem{remark}[theorem]{Remark}
\newtheorem{example}[theorem]{Example}

\newtheoremstyle{key}{3pt}{3pt}{}{}{\bfseries}{:}{.5em}{}
\theoremstyle{key}
\newtheorem*{JEL}{JEL subject classification}

\newtheorem*{keywords}{Keywords}

\DeclareMathOperator*{\esssup}{ess\,sup}

\renewcommand{\mid}{\,\vert\,}
\newcommand{\bigv}{\!\bigm\vert\!}
\newcommand{\Bigv}{\!\Bigm\vert\!}
\newcommand{\biggv}{\!\biggm\vert\!}
\providecommand{\abs}[1]{\left\lvert#1\right\rvert}
\providecommand{\norm}[1]{\lVert#1\rVert}
\providecommand{\ceil}[1]{\left\lceil#1\right\rceil}

\newcommand{\nbd}[1]{$#1$\nobreakdash-\hspace{0pt}}
\newcommand{\indi}[1]{\mathbf{1}_{\{#1\}}}
\newcommand{\indinb}[1]{\mathbf{1}_{#1}}

\newcommand{\N}{\mathbb{N}}
\newcommand{\Q}{\mathbb{Q}}
\newcommand{\R}{\mathbb{R}}

\newcommand{\B}{{\mathscr B}}

\newcommand{\F}{{\mathscr F}}

\newcommand{\T}{{\mathscr T}}

\usepackage[pdfpagelayout=OneColumn]{hyperref}

\title{Symmetric Equilibria in Stochastic Timing Games}

\author{Jan-Henrik Steg\thanks{%
Center for Mathematical Economics, Bielefeld University, Germany. E-mail: \texttt{jsteg@uni-bielefeld.de}
\protect\\
This version: May 20, 2018. Financial support by the German Research Foundation (DFG) via CRC 1283 is gratefully acknowledged.
}}
\date{ }

\begin{document}
\maketitle

\begin{abstract}
We construct subgame-perfect equilibria with mixed strategies for symmetric stochastic timing games with arbitrary strategic incentives. The strategies are qualitatively different for local first- or second-mover advantages, which we analyse in turn. When there is a local second-mover advantage, the players may conduct a war of attrition with stopping rates that we characterize in terms of the Snell envelope from the theory of optimal stopping. This is a very general result, but it provides a clear interpretation. When there is a local first-mover advantage, stopping typically results from preemption and is abrupt. Equilibria may differ in the degree of preemption, precisely when it is triggered or not. We develop an algorithm to characterize when preemption is inevitable and to construct corresponding payoff-maximal symmetric equilibria.

\begin{keywords}
Stochastic timing games, mixed strategies, subgame-perfect equilibrium, optimal stopping, {S}nell envelope.
\end{keywords}
\begin{JEL}
C61, C73, D21, L12
\end{JEL}
\end{abstract}

\section{Introduction}\label{sec:intro}

The aim of this paper is to construct and understand subgame-perfect equilibria for symmetric stochastic timing games, which have important applications, for instance, in strategic real option models. It is well known that in many timing games in continuous time, there exist no equilibria in pure strategies. If they do exist, however, they typically involve asymmetric payoffs that only depend on the respective roles of the players, which must be determined before the game starts. Then there is an unresolved strategic conflict. Here we strive for a rather general existence result and possibly symmetric payoffs, so we consider mixed strategies. In particular, no assumption is made concerning the local incentives, which can move randomly between first- and second-mover advantages. 

Restricting attention to games with a second-mover advantage is known to be helpful for equilibrium existence. We begin by analyzing that case, too, demonstrating the general payoff asymmetry in pure-strategy equilibria. Our main contribution for this case is the construction of mixed strategy equilibria with symmetric payoffs in a possibly general model, making no specific assumptions concerning the underlyling uncertainty. Nevertheless, the equilibrium strategies have a clear characterization and interpretation, using the concept of the Snell envelope from optimal stopping theory. Specifically, we can describe the stopping rates which make a player indifferent to stay in the game when forgoing a profitable local payoff. Due to the possible uncertainty, the tradeoffs can be all in expected terms. As we are not assuming any kind of smoothness or monotonicity of the underlying payoff processes, we also generalize existing results for purely deterministic models.

With a first-mover advantage, there is often a preemption incentive that leads to equilibrium existence problems, even with mixed strategies and in very simple and well-behaved deterministic models.\footnote{%
See, e.g., \cite{FudenbergTirole85} and \cite{HendricksWilson92}.
}
Strategy spaces and outcome distributions need to be extended to model preemption appropriately in continuous time, which requires some coordination device (not to be confused with familiar \emph{correlated} strategies).\footnote{%
Specifically, as there is no respective ``next period'' to charge with a threat of stopping if nobody stops in the current one, players have to be able to prevent that everybody waits for a positive amount of time, regardlessly of the opponents' strategies (resp.\ deviations), but without the implication that all stop simultaneously with probability one.
} 
We use the generalization of \citeauthor{FudenbergTirole85}'s \citeyearpar{FudenbergTirole85} concept developed in \cite{RiedelSteg17} for stochastic models, which affords us local equilibria of preemption reminiscent of symmetric mixed equilibria in discrete time. Such continuation equilibria can be combined with the previous, continuously mixed equilibria for local second-mover advantages, to establish a general existence and characterization of subgame-perfect equilibria with symmetric payoffs~-- without any restriction on the order of the underlying payoff processes.

These results facilitate the analysis of richer models, e.g., of strategic real options. So far, the focus has been either on attrition or on preemption, but feedback effects between strategically relevant first- and second-mover advantages in different phases of one game have not been studied much.\footnote{%
See \cite{StegThijssen15} for a model where every firm's equilibrium option exercise can happen while there is a first- or a second-mover advantage (each with positive probability) and for an explicit analysis of the strategic tradeoffs. Many real options models with preemption effects also have phases with second-mover advantages, but it is then argued that waiting is optimal (see \cite{RiedelSteg17} and \cite{Steg18} for general conditions when this is the case). \cite{Kwonetal16} either observe preemption or attrition depending on their initial choice of parameters. \cite{DecampsMariotti04} discuss introducing a preemption incentive to their war of attrition, which essentially leads to a new boundary condition for their equilibrium differential equation. 
}

Depending on the profitability of future continuation equilibria, preemption does not have to occur just because there is a current first-mover advantage. As the term says, preemption eliminates future continuation payoffs; so the less often preemption occurs, the higher are potential equilibrium payoffs. To determine at which times preemption is indeed inevitable, we provide an algorithm working under the additional assumption that simultaneous stopping is generally the least desirable outcome. Then we find that in any equilibrium with symmetric payoffs in every subgame, the equilibrium payoff can never exceed the expected value of optimally stopping the \emph{minimum} of the local leader and follower payoffs. This means, no matter how the players mix, possibly even with arbitrary public correlation and independently of infinite remaining time, they can never benefit from a high value of the underlying payoff processes if that is not attained by both the leader and follower payoffs simultaneously.

Then we know that the game has to end by preemption whenever the leader payoff exceeds the equilibrium payoff bound. Iterating this procedure cumulates in the identification of times when preemption cannot be avoided. Confining preemption to such times, we are able to construct an equilibrium with least sustainable preemption and highest possible payoffs.

\subsection{Main theorems}

This paper develops three main Theorems \ref{thm:SPE}, \ref{thm:symeql} and \ref{thm:maxeql}, which build on each other. 

Theorem \ref{thm:SPE} constructs subgame-perfect equilibria with mixed strategies for games with a systematic (weak) second-mover advantage, for instance. The players have to coordinate on a suitable payoff process therefor, which consists of the leader payoff up to some feasible time that either admits a simultaneous stopping equilibrium or is sufficiently late such that both players will have stopped for sure by then. The equilibrium proceeds by optimally stopping this fixed process. As long as there are expected gains, no player stops. When a point is reached, however, at which it would be strictly optimal to stop, i.e., if any further delay would imply a loss, then there needs to be a compensation in terms of some probability to obtain the superior follower payoff. We characterize the exact stopping rate that the respective opponent has to use to make each player indifferent to continue at such points. 

Owing to its generality, this result is technically not as clear as its interpretation. With typical Brownian models, for instance, one cannot apply local arguments as there is no path monotonicity at all. Consequently, it is then also impossible to distinguish proper time intervals on which mixing occurs~-- imagine a Brownian motion fluctuating around the boundary of the region where mixing indeed takes place. Nevertheless, using martingale arguments, we obtain a clear representation of strategies involving the concept of Snell envelope from the theory of optimal stopping, which allows us to speak meaningfully of a (local) expected loss, for instance. These strategies will typically be continuous up to some terminal jumps. Another important question then is time-consistency. If we define mixed strategies for all subgames, i.e., stopping times as starting dates, we have to ensure that they imply consistent conditional stopping probabilities throughout the game, which is generally not trivial.

Theorem \ref{thm:symeql} then incorporates further strategy extensions, which provide symmetric preemption equilibria for regimes with a first-mover advantage. The theorem establishes that they form feasible continuation equilibria when leaving regimes with second-mover advantages. In aggregate, we thus obtain payoff-symmetric equilibria for games without any restriction on the local incentives. There may be arbitrary, random alternations of first- or second-mover advantages in the game.

Theorem \ref{thm:maxeql} determines \emph{optimal} symmetric equilibria. Whereas the previous ones involve extreme preemption~-- whenever there is a strict first-mover advantage~--, we now identify equilibria with least sustainable preemption, resulting in the highest feasible payoffs. For that purpose we focus on \emph{payoff-symmetric} equilibria, with symmetric payoffs in every subgame, as this property has important implications for equilibrium strategies. Roughly, conditional stopping probabilities can only differ when players are currently indifferent between becoming leader or follower. Under the additional assumption that simultaneous stopping is not strictly better than leading or following, the players can coordinate at most on optimally stopping the minimum of the leader and follower payoff processes in any equilibrium. Whenever the leader payoff exceed that value, preemption must occur. Knowing this restricts the relevant stopping times in the previous problem, which further reduces the attainable value. Iterating the procedure formally as an algorithm identifies inevitable preemption points.

Theorem \ref{thm:maxeql} establishes that we even obtain a well-defined subgame-perfect equilibrium in the end, not only a limit value. It is based on the previous equilibria, but suppressing preemption whenever possible, and verifies that indeed well measurable, time-consistent strategies result when applying the proposed algorithm to all subgames.

\subsection{Related literature}

Strategic timing problems appear in an abundance of contexts, in particular in economics but also in biology, e.g., and consequently the related literature is vast.

On the one hand, there is a branch on deterministic timing problems in continuous time addressing a wide range of applications, where typically a distinction is made between preemption models and wars of attrition. Correspondingly, \cite{HendricksWilson92} and \cite{Hendricksetal88} study stylized models with systematic first- and second-mover advantages, respectively. A war of attrition appears in \cite{GhemawatNalebuff85}, who consider exit from a declining industry.\footnote{%
\cite{FudenbergTirole86} analyse a market exit problem with incomplete information. \cite{BulowKlemperer99} consider a similar problem with more than two firms. 
} 
In a seminal contribution, \cite{FudenbergTirole85} emphasize subgame-perfection in a symmetric preemption game. \cite{HoppeLehmann-Grube05} model a similar technology adoption game, allowing the leader payoff function to be multi-peaked, but restricting the follower payoff to be nonincreasing.\footnote{%
We discuss some implication of $F$ not to increase \emph{in expectation}, i.e., to be a supermartingale, in Section \ref{sec:eqlpure}. \cite{Duttaetal95} obtain again a similar structure as \cite{FudenbergTirole85} (including the single-peakedness property) from a model of product differentiation.
}
Without uncertainty, these games proceed quite linearly due to perfect foresight. More complications arise when the incentives may vary more freely. \cite{Larakietal05} consider general deterministic \nbd{N}player games with payoffs that are just continuous functions of time (for given identities of first-movers). They prove that there do always exist \nbd{\varepsilon}equilibria, but not necessarily exact equilibria.

On the other hand, there is also a wide branch of the literature considering timing games with uncertainty in continuous time. \cite{DuttaRustichini93}, e.g., formulate a symmetric Markovian setting. However, restricting themselves to pure strategies, their Markov perfect equilibrium payoffs are generally asymmetric and need to assume away preemption issues.

Important and numerous applications with uncertainty are strategic real options. For instance, an early contribution is \cite{Smets91}. A typical symmetric model of preemptive investment is that of \cite{MasonWeeds10}.\footnote{%
\cite{PawlinaKort06} consider a similar model with asymmetric investment costs and  \cite{Thijssen10} one with firm-specific uncertainty. \cite{LambrechtPerraudin03} model preemption with incomplete information.
}
\cite{Grenadier96} considers strategic real-estate development with construction delay and \cite{Weeds02} irreversible R\&D investment. A war of attrition results from \citeauthor{Murto04}'s \citeyearpar{Murto04} model of exit from duopoly, as well as from investment with learning externalities in \cite{DecampsMariotti04}.\footnote{%
The latter features Poisson uncertainty, in contrast to Brownian uncertainty in the other mentioned models, which allows an analysis based on monotonically evolving beliefs and a resulting differential equation.
}

Finally, as we emphasize uncertainty, the literature on Dynkin games and its large tradition needs to be mentioned. As these are two-person, \emph{zero-sum} timing games, the classical question is the existence of an equilibrium saddle point, or value, under varying conditions. Here we only refer to the more recent work by \cite{TouziVieille02}, as their payoff processes are very general and~-- more importantly~-- as they also use mixed strategies. They prove that many more Dynkin games then have a value. Their concept of mixed strategies is different, does not consider subgames, but expected payoffs at time zero can be related to those from our strategies.

Recently, also some more abstract work considering stochastic timing games with nonzero-sum payoffs has been conducted. \cite{HamadeneZhang10}, e.g., prove existence of a Nash equilibrium for 2-player games with second-mover advantage throughout.\footnote{%
See also \cite{HamadeneHassani14} for an extension to $N$ players using a similar approach. \cite{LarakiSolan13} make less assumptions concerning the incentives in a 2-player game. Consequently, even allowing for mixed strategies, they can only prove existence of \nbd{\varepsilon}equilibria.
}

\subsection{Outline}

This paper is organized as follows. In Section \ref{sec:game} we define our timing games, making only minimal regularity assumptions, and summarize the concept of subgame-perfect equilibria in mixed strategies as developed in \cite{RiedelSteg17}.

Although we are generally working with mixed strategies, equilibrium verification is related to solving optimal stopping problems by linearity. We establish a convenient representation of this connection in Section \ref{sec:BR} and present the needed facts from the general theory of optimal stopping. On the one hand, strategies will be represented in terms of the Snell envelope, which we motivate. On the other hand, in our games we have to be quite careful about existence of optimal stopping times, which depends strongly on path properties of the involved processes, so we will address some important details.

By a first application of this theory in Section \ref{sec:eqlpure}, we establish equilibria in pure strategies and argue that they typically generate coordination problems. These are resolved in Section \ref{sec:eqlmix} by the construction of subgame-perfect equilibria in mixed strategies, in a first step for games with systematic second-mover advantage. Although our representation of the equilibrium strategies can be well interpreted, we derive a completely explicit equilibrium for a market exit example in Section \ref{sec:expduo}.

In Section \ref{sec:eqlsym} we use the aforementioned strategy extensions to deal with first-mover advantages, which then enables us to construct and characterize subgame-perfect equilibria for arbitrary symmetric timing games. Finally, we identify equilibria with maximal payoffs and least possible preemption in Section \ref{sec:symeql}. Section \ref{sec:conc} concludes. Appendices contain some technical results and all proofs.

\section{The timing game}\label{sec:game}

We use the framework for subgame-perfect equilibria with mixed strategies developed in \cite{RiedelSteg17}, where the concepts summarized in this section are explained in more detail. Here we only consider symmetric games, which allows some simplifications incorporated in the following. 

The timing game consists of two players $i=1,2$, who each decide when to stop in continuous time $t\in[0,\infty]$ (with stopping at $t=\infty$ interpreted as ``never stopping''). Potential uncertainty about the state of the world is modeled by a fixed probability space $(\Omega,\F,P)$, and partial information about the true state that can evolve exogenously over time by a filtration $(\F_t)_{t\geq 0}$. The player's stopping decisions may of course use this information, so when a player stops may depend on the state (see Section \ref{subsec:strateql} for the formal definition of strategies).

As usual in timing games, we focus on situations (resp.\ histories) in which no player has stopped, yet. Therefore, the game ends as soon as some player stops. A player who is the single one to stop first is called the \emph{leader}; in this case, the other player becomes the \emph{follower}. Their respective payoffs are determined by two given stochastic processes, $L=(L_t)_{t\geq 0}$ and $F=(F_t)_{t\geq 0}$. Both processes incorporate the possible effect of an (optimal, contingent) stopping decision that the follower might have in a more primitive model, given that the opponent has already stopped, like in Example \ref{exm:attrition}. If the game ends by both players stopping simultaneously, then their payoffs are determined by a third given process, $M=(M_t)_{t\geq 0}$, respectively the random variable $M_\infty$ if no player stops in finite time. All payoffs are measured in the same numeraire, say, discounted to time zero, and the players are risk neutral.\footnote{%
Alternatively, one can interpret the payoff processes as measured in discounted ``utils''.
} 

Equilibria will obviously be based on solving optimal stopping problems involving the three underlying payoff processes. We need to make some weak regularity assumptions in order to have well defined problems in the following.

\begin{assumption}\label{asm:payoffs}
\par\noindent
\begin{enumerate}
\item
$(\Omega,\F,P)$ is a fixed probability space equipped with a filtration $(\F_t)_{t\geq 0}$ satisfying the usual conditions (i.e., right-continuity and completeness).

\item
The processes $L$, $F$ and $M$ are adapted, right-continuous (a.s.) and of class {\rm (D)}, and $E[\abs{M_\infty}]<\infty$.

\item\label{LFusc}
$\min(L,F)$ is upper-semi-continuous from the left in expectation, in fact on $[0,\infty]$ if we put $L_\infty:=F_\infty:=M_\infty$.
\end{enumerate}
\end{assumption}

\begin{remark}\label{rem:payoffs}
\par\noindent
\begin{enumerate}
\item\label{stoppingtime}
The payoff processes $L$, $F$ and $M$ do not have to be random; deterministic ones are just a special case. Even then the probability space and filtration might be nontrivial and represent possible public randomization devices. The current payoffs at any time $t$ just should be known by the public information $\F_t$. The state-dependent dates identifiable by the dynamic information $(\F_t)$ are the \emph{stopping times} $\tau\colon\Omega\to[0,\infty]$, i.e., those satisfying $\{\tau\leq t\}\in\F_t$ for every $t\in\R_+$. They are crucial for strategies and outcomes. Let $\T$ denote the set of all stopping times.

\item\label{rem:M_infty}
It depends on the model whether there is a natural payoff if both players ``never stop'', which may be some limit of $M$ or of $L$. In the latter case, we simply set $M_\infty\equiv L_\infty$ and work with $M_\infty$ for a unified payoff notation. For convenience, we also formally define
\begin{equation*}
F_\infty:=M_\infty.
\end{equation*}

\item
Two important general technical issues are measurability, in particular concerning strategies, and integrability. We need to ensure that expectations are always well defined and that pointwise converging random variables converge in expectation, too. Class {\rm (D)} is the possibly weakest integrability condition we can work with.\footnote{\label{fn:classD}%
A measurable process $(X_t)_{t\geq 0}$ is of class {\rm (D)} if the family $\{X_\tau\mid\tau\in\T,\tau<\infty\text{ a.s.}\}$ is uniformly integrable. Then the family is bounded in expectation and pointwise convergence of $X_t$ at a stopping time $\tau<\infty$ implies convergence in $L^1(P)$ as well. This is a mild regularity condition implied by, e.g., $E[\sup_t\abs{X_t}]<\infty$ or $\sup_\tau E[\abs{X_\tau}^p]<\infty$ for some $p>1$. We may equivalently define any $X_\infty\in L^1(P)$ and consider \emph{all} (also nonfinite) stopping times $\tau$ in the previous set; see, e.g., Lemma B.1 in \cite{RiedelSteg17}.
}
Boundedness would be much too strong for many applications (e.g., involving Brownian motion).

\item
In order to have any general existence results for equilibria, some path regularity of the payoff processes is necessary, as can be clearly seen even in the deterministic, single agent  case. Nevertheless, it suffices for us to have upper-semi-continuity from the left only in expectation.\footnote{%
Upper-semi-continuity from the left in expectation means $E[L_\tau\wedge F_\tau]\geq\limsup_{n}E[L_{\tau_n}\wedge F_{\tau_n}]$ for any sequence of stopping times $(\tau_n)_{n\in\N}$ that is a.s.\ increasing to a stopping time $\tau$.
}
This is also a known necessary condition for optimal stopping problems. We use it for equilibria in mixed strategies when there is a (local) second-mover advantage. It is of course only required for $L$ if that never exceeds $F$. Indeed, one could restrict attention to intervals $[\tau,\,\inf\{t\geq\tau\mid L_t>F_t\}]$, where $\tau$ is a stopping time; the area $\{L>F\}\subseteq\Omega\times\R_+$ is only relevant at transitions. The assumption is satisfied if, e.g., the paths of $L$ and $F$ are a.s.\ (upper-semi-)continuous from the left.\footnote{%
Then $\limsup_{s\nearrow t}(L_s\wedge F_s)\leq(\limsup_{s\nearrow t}L_s)\wedge(\limsup_{s\nearrow t}F_s)\leq L_t\wedge F_t$ for all $t\in[0,\infty]$ a.s., and we note that $L$ and $F$ are of class {\rm (D)}.
}
\end{enumerate}
\end{remark}

\begin{example}\label{exm:attrition}
Let us consider a market exit problem as a simple example for a stochastic timing game with second-mover advantage, i.e., $F\geq L$, like in the classical war of attrition.\footnote{%
For related examples of preemption type, see \cite{Steg18}.
}

Suppose that two firms are operating in one market such that duopoly returns $\pi^D$ might not be sustainable in the long run, depending on uncertain exogenous conditions. Whereas each firm would in general like the opponent to leave the market in order to earn the monopoly profit $\pi^M\geq\pi^D$, it might be too costly to wait for that possibly random event. Each firm thus decides on times when waiting becomes no longer promising, and at which to leave the market if the other is still present.

The payoff processes are then given by
\begin{align}\label{Ft}
L_t={}&M_t:=\int_0^t\pi^D_s\,ds,\nonumber\displaybreak[0]\\
F_t:={}&L_t+\esssup_{\tau^F\in\T\colon\tau^F\geq t}E\biggl[\int_t^{\tau^F}\pi^M_s\,ds\biggv\F_t\biggr]
\end{align}
for every $t\in[0,\infty]$, reflecting that also a monopolist may quit at stopping times $\tau^F\in\T$ (cf.\ Remark \ref{rem:payoffs} \ref{stoppingtime}) when even the monopoly return is not profitable, so that immediate exit is dominant and the second-mover advantage not strict. However, we will not explicitly model the \emph{strategy} of a single remaining firm, but incorporate the corresponding optimal decision in the payoff processes; subgame-perfection requires that the latter be chosen optimally.

Assumption \ref{asm:payoffs} is satisfied by this example if $\pi^D$ and $\pi^M$ are adapted and \nbd{P\otimes dt}integrable, as all processes are then bounded by an integrable random variable.  The convention $F_\infty=M_\infty=L_\infty$ here holds naturally. It will follow from our discussion in Section \ref{subsec:optstop} that there exists a right-continuous process $F$ such that relation \eqref{Ft} holds even when replacing $t$ by a general stopping time $\tau$, which is one of the most important results in continuous-time stopping theory.\footnote{%
This issue is more delicate for $L$ if the leader's payoff also depends on the stopping time eventually chosen by the follower, like in an entry model; see \cite{Steg18}. Optimality of the follower's decision implies right-continuity (in expectation) of the resulting payoff process, but not for the leader's, because the follower's decision is in general unrelated to (optimally stopping) the leader's payoff \emph{stream}. Such problems will typically not arise in diffusion models, however.
}
Moreover, $F_\tau$ is then indeed consistent with an optimal follower stopping decision $\tau^{F*}\in\T$, $\tau^{F*}\geq\tau$; see also the discussion of a more specific instance in Section \ref{sec:expduo}.
\end{example}

\subsection{Mixed strategies and equilibrium concept}\label{subsec:strateql}

We use the following concept of subgame-perfect equilibrium for stochastic timing games developed in \cite{RiedelSteg17}. As usual in timing games, the basic objects are plans for each player when to move, resp.\ \emph{stop}, which are followed as long as no other player stops before.\footnote{%
Cf.\ \cite{Larakietal05} or \cite{FudenbergTirole85}, who call these plans ``simple strategies.''
}
Consequently, the players stopping first are those whose planned times are minimal, and then payoffs as prescribed by the given processes accrue, depending on which players' plans are minimal. A feasible state-contingent plan is a \emph{stopping time} $\tau$, the set of which is denoted by $\T$ (see Remark \ref{rem:payoffs} \ref{stoppingtime}). These times are identifiable by the dynamic information $(\F_t)$ and therefore also constitute feasible continuous-time decision nodes when to revise plans. To describe behavior also off path, \emph{every} stopping time is hence additionally considered as the beginning of a \emph{subgame} with the connotation that no player has stopped, yet. In this role, they will typically be denoted by $\vartheta\in\T$.\footnote{\label{fn:stopinfo}%
Considering only deterministic times $t\in\R_+$ and their information structures $\F_t$ for starts of subgames would in general be incomplete, because stopping times $\vartheta\in\T$ and their information structures $\F_\vartheta=\{A\in\F\mid\forall t\in\R_+\colon A\cap\{\vartheta\leq t\}\in\F_t\}$ are (dynamically) identifiable by $(\F_t)$, but form a much richer system in continuous time.
}

In any subgame, the players can randomize over remaining plans by specifying distribution functions $G_i^\vartheta$ on $[\vartheta,\infty]$ that may still condition on the state by $(\F_t)$. A ``pure'' plan $\tau\geq\vartheta$ then corresponds to $G_i^\vartheta(t)=\indi{t\geq\tau}$. Plans for different subgames must be time-consistent; in particular, randomized plans for different starting dates have to induce the same \emph{conditional} stopping probabilities whenever possible.

Additional strategy extensions are needed for subgames with first-mover advantages, to model preemption appropriately in continuous time.\footnote{%
See also \cite{HendricksWilson92} on equilibrium existence issues for deterministic preemption games.
}
As in \cite{FudenbergTirole85}, the players can also place an ``atom'' $\alpha_i^\vartheta(t)$~-- a conditional stopping probability~-- on \emph{every} point $t\in[0,\infty]$, which is conceived to capture limit outcomes from discrete time.\footnote{%
For an actual limit analysis in the deterministic case, see \cite{Steg17}.
} 
These extensions are included in the following formal definition, although we will ignore them in the discussion of games with a second-mover advantage and only take them up later for general games.

\begin{definition}\label{def:alpha}
An \emph{extended mixed strategy} for player $i\in\{1,2\}$ in the subgame starting at $\vartheta\in\T$ is a pair of processes $(G_i^\vartheta,\alpha_i^\vartheta)$ that both take values in $[0,1]$ and satisfy the following properties. 
\begin{enumerate}
\item
$G_i^\vartheta$ is adapted. It is a.s.\ nondecreasing, right-continuous, and satisfying $G_i^\vartheta(t)=0$ for all $t<\vartheta$. 

\item
$\alpha_i^\vartheta$ is progressively measurable.\footnote{%
Formally, the restricted mappings $\alpha_i^\vartheta\colon\Omega\times[0,T]\to\R$ must be $\F_T\otimes\B([0,T])$-measurable for any $T\in\R_+$. This is a stronger condition than adaptedness, but weaker than optionality, which we automatically have for $G_i^\vartheta$ by right-continuity. Progressive measurability implies that $\alpha_i^\vartheta(\tau)$ will be \nbd{\F_\tau}measurable for any $\tau\in\T$.
} 
It is a.s.\ right-continuous in all $t\in\R_+$ for which $\alpha_i^\vartheta(t)<1$ and satisfying $\alpha_i^\vartheta(t)=0$ for all $t<\vartheta$.\footnote{%
As we are here only interested in \emph{symmetric} games, we may demand $\alpha_i^\vartheta(\cdot)$ to be right-continuous also where it takes the value zero, which simplifies the definition of outcomes. See \cite{RiedelSteg17} for issues with asymmetric games and corresponding weaker regularity restrictions.
}
  
\item
\begin{equation*}
\alpha_i^\vartheta(t)>0\Rightarrow G_i^\vartheta(t)=1\qquad\text{for all }t\geq 0\text{, a.s.}
\end{equation*}
\end{enumerate}
For every extended mixed strategy, define also $G_i^\vartheta(0-)\equiv 0$, $G_i^\vartheta(\infty)\equiv 1$ and $\alpha_i^\vartheta(\infty)\equiv 1$.
\end{definition}

We speak of a ``standard'' mixed strategy if $\alpha_i^\vartheta(t)=0$ for all $t\in\R_+$, i.e., $\alpha_i^\vartheta\equiv\alpha_i^\infty$. Restricting to such strategies is in the following equivalent to defining only mixed strategies $G_i^\vartheta$ with the given properties and no extensions at all. If furthermore $G_i^\vartheta(t)=\indi{t\geq\tau}$ for some stopping time $\tau\in\T$, then the respective strategy is referred to as ``pure.'' For a pair of pure strategies, corresponding to stopping times $\tau_i,\tau_j\geq\vartheta$, player $i$'s expected payoff at $\vartheta$ will be
\begin{equation*}\label{payoffpure}
E\Bigl[\indi{\tau_i<\tau_j}L_{\tau_i}+\indi{\tau_i>\tau_j}F_{\tau_j}+\indi{\tau_i=\tau_j}M_{\tau_i}\Bigv\F_\vartheta\Bigr].
\end{equation*}
This is extended linearly to mixed strategies.

\begin{definition}\label{def:payoffs_extended}
Given two extended mixed strategies $(G_i^\vartheta,\alpha_i^\vartheta)$, $(G_j^\vartheta,\alpha_j^\vartheta)$, $i,j\in\{1,2\}$, $i\neq j$,  the \emph{payoff} of player $i$ in the subgame starting at $\vartheta\in\T$ is
\begin{align*}
V_i^\vartheta\bigl(G_i^\vartheta,\alpha_i^\vartheta,G_j^\vartheta,\alpha_j^\vartheta\bigr):=E&\biggl[\int_{[0,\hat\tau^\vartheta)}\bigl(1-G_j^\vartheta(s)\bigr)L_s\,dG_i^\vartheta(s)+\int_{[0,\hat\tau^\vartheta)}\bigl(1-G_i^\vartheta(s)\bigr)F_s\,dG_j^\vartheta(s)\nonumber\\
&+\sum_{s\in[0,\hat\tau^\vartheta)}\Delta G_i^\vartheta(s)\Delta G_j^\vartheta(s)M_s+\lambda^\vartheta_{L,i}L_{\hat\tau^\vartheta}+\lambda^\vartheta_{L,j}F_{\hat\tau^\vartheta}+\lambda^\vartheta_{M}M_{\hat\tau^\vartheta}\biggv\F_\vartheta\biggr],
\end{align*}
where $\hat\tau^\vartheta:=\inf\{t\geq\vartheta\mid\alpha_1^\vartheta(t)+\alpha_2^\vartheta(t)>0\}$ and $\lambda^\vartheta_{L,i}$, $\lambda^\vartheta_{L,j}$ and $\lambda^\vartheta_{M}$ are the terminal outcome probabilities (of players $i$ or $j$ becoming leader, or simultaneous stopping, resp.) induced by $\alpha_i^\vartheta,\alpha_j^\vartheta$ at $\hat\tau^\vartheta$ and defined in Appendix \ref{app:outcome}.
\end{definition}

The outcome probabilities from the extensions sum up to $(1-G_i^\vartheta(\hat\tau^\vartheta-))(1-G_j^\vartheta(\hat\tau^\vartheta-))$, the probability of nobody stopping before they are used. Their definition in Appendix \ref{app:outcome} is a simplification of that in \cite{RiedelSteg17} due to slightly stronger regularity. If both players use standard mixed strategies, in particular if they reserve some mass for $t=\infty$, then $\hat\tau^\vartheta=\infty$, $\lambda^\vartheta_{L,i}=\lambda^\vartheta_{L,j}=0$, and the associated payoff part is $(1-G_i^\vartheta(\infty-))(1-G_j^\vartheta(\infty-))M_\infty$.

The pathwise integrals include possible jumps of the right-continuous integrators at zero, as player $i$ can indeed become leader (follower) from an initial jump of $G_i^\vartheta$ ($G_j^\vartheta$). Assumption \ref{asm:payoffs} ensures that the payoffs are well defined and bounded in expectation; cf.\ Lemma \ref{lem:LdG}. 

For a consistent dynamic view of the whole game, the conditional stopping probabilities at any fixed time from strategies for different subgames have to be the same whenever possible.

\begin{definition}\label{def:TC_extended}
A \emph{time-consistent extended mixed strategy} for player $i\in\{1,2\}$ in the timing game is a family $((G_i^\vartheta,\alpha_i^\vartheta);\vartheta\in\T)$ of extended mixed strategies for all subgames such that for all $\vartheta,\vartheta',\tau\in\T$ with $\vartheta\leq\vartheta'\leq\tau$ it holds that (a.s.)
\begin{equation*}
G_i^\vartheta(t)=G_i^\vartheta(\vartheta'-)+\bigl(1-G_i^\vartheta(\vartheta'-)\bigr)G_i^{\vartheta'}(t)\text{ for all }t\geq\vartheta'\quad\text{and}\quad\alpha_i^\vartheta(\tau)=\alpha^{\vartheta'}_i(\tau).
\end{equation*}
\end{definition}

Time-consistency implies that for any two subgames resp.\ starting at $\vartheta,\vartheta'\in\T$ indeed $G_i^\vartheta\equiv G_i^{\vartheta'}$ (a.s.) on the event $\{\vartheta=\vartheta'\}$, as should be expected. 

\begin{definition}\label{def:SPE_extended}
A \emph{subgame-perfect equilibrium} is a pair of time-consistent extended mixed strategies $((G_1^\vartheta,\alpha_1^\vartheta);\vartheta\in\T)$, $((G_2^\vartheta,\alpha_2^\vartheta);\vartheta\in\T)$ such that for all $\vartheta\in\T$, $i,j\in\{1,2\}$, $i\neq j$, and extended mixed strategies $(G_a^\vartheta,\alpha_a^\vartheta)$ for the subgame at $\vartheta$ it holds that
\begin{equation*}
V_i^\vartheta\bigl(G_i^\vartheta,\alpha_i^\vartheta,G_j^\vartheta,\alpha_j^\vartheta\bigr)\geq V_i^\vartheta\bigl(G_a^\vartheta,\alpha_a^\vartheta,G_j^\vartheta,\alpha_j^\vartheta\bigr)\quad\text{a.s.},
\end{equation*}
i.e., such that every pair $(G_1^\vartheta,\alpha_1^\vartheta)$, $(G_2^\vartheta,\alpha_2^\vartheta)$ is an \emph{equilibrium} in the subgame at $\vartheta\in\T$.
\end{definition}

\section{Best replies and optimal stopping}\label{sec:BR}

The payoffs in Definition \ref{def:payoffs_extended} are apparently linear in strategies. In this section, we derive a more explicit representation of that linearity, which will be very helpful for rigorous equilibrium verification and also for necessity arguments. To construct or to verify any best replies, it is in general necessary to maximize over (extended) mixed strategies against such objects. Here we make related statements such as ``any stopping time in the support of the mixed strategy needs to be optimal'' precise. We furthermore introduce central concepts from the theory of optimal stopping, notably the \emph{Snell envelope}, which will play a crucial role in the following representation and interpretation of mixed strategies in equilibrium.

The following arguments concern the distributions $G_i^\vartheta$, so we focus on ``standard'' mixed strategies and ignore the extensions $\alpha_i^\vartheta$ for notational simplicity until Section \ref{sec:eqlsym} (which does not weaken our equilibrium notion as Lemma \ref{lem:BRpure} shows). 

For an alternative representation of player $i$'s payoff in the subgame starting at $\vartheta\in\T$, we introduce the process $S_i^\vartheta$ given by
\begin{equation}\label{Si}
S_i^\vartheta(t):=\int_{[0,t)}F_s\,dG_j^\vartheta(s)+\Delta G_j^\vartheta(t)M_t+\bigl(1-G_j^\vartheta(t)\bigr)L_t
\end{equation}
for all $t\in[0,\infty)$, where $G_j^\vartheta$ is from a given mixed strategy for the opponent. Lemma \ref{lem:SclassD} shows that $S_i^\vartheta$ is well behaved: like $L$, $F$, and $M$, it is optional\footnote{\label{fn:optional}%
This means that $S_i^\vartheta$ is measurable w.r.t.\ the optional \nbd{\sigma}field on the product space $\Omega\times\R_+$, which is generated by all right-continuous adapted processes or equivalently by the random intervals $[0,\tau)$, $\tau\in\T$.
}
(so in particular adapted) and of class {\rm (D)}, but not necessarily right-continuous. Given $M_\infty\in L^1(P)$, we can extend the definition of $S_i^\vartheta$ in \eqref{Si} to $t=\infty$, implying also $S_i^\vartheta(\infty)\in L^1(P)$. 
$S_i^\vartheta$ may now be integrated by $dG_i^\vartheta$ thanks to Lemma \ref{lem:LdG}, such that player $i$'s expected payoff at $\vartheta\in\T$ from a pair of standard mixed strategies can be written as\footnote{%
An application of Fubini's theorem given integrability thanks to Lemma \ref{lem:LdG} yields in particular
\begin{align*}
&\int_{[0,\infty)}\bigl(1-G_i^\vartheta(s)\bigr)F_s\,dG_j^\vartheta(s)=\int_{[0,\infty)}\int_{[0,\infty]}\indi{t>s}\,dG_i^\vartheta(t)F_s\,dG_j^\vartheta(s)\\
={}&\int_{[0,\infty]}\int_{[0,\infty)}\indi{s<t}F_s\,dG_j^\vartheta(s)\,dG_i^\vartheta(t)=\int_{[0,\infty]}\int_{[0,t)}F_s\,dG_j^\vartheta(s)\,dG_i^\vartheta(t)\in L^1(P).
\end{align*}
}
\begin{equation}\label{V=SdG}
V_i^\vartheta\bigl(G_i^\vartheta,G_j^\vartheta\bigr)=E\biggl[\int_{[0,\infty]}S_i^\vartheta(t)\,dG_i^\vartheta(t)\biggv\F_\vartheta\biggr].
\end{equation}
The linearity in \eqref{V=SdG} suggests that a best reply exists if and only if there is a pure one. 

\begin{lemma}\label{lem:BRpure}
For any $\vartheta\in\T$ and standard mixed strategies $G_i^\vartheta$, $G_j^\vartheta$ in the subgame at $\vartheta$,
\begin{equation}\label{stopSi}
V_i^\vartheta\bigl(G_i^\vartheta,G_j^\vartheta\bigr)\leq\esssup_{\tau\in\T\colon\tau\geq\vartheta}E\bigl[S_i^\vartheta(\tau)\bigv\F_\vartheta\bigr]\qquad\text{a.s.},
\end{equation}
with equality if and only if for a.e.\ $x\in[0,1)$, the right-hand side is attained by the stopping time $\tau_i^{G,\vartheta}(x):=\inf\{t\geq\vartheta\mid G_i^\vartheta(t)>x\}$. Moreover, the inequality still holds if $G_i^\vartheta$ is extended by any feasible (nontrivial) $\alpha_i^\vartheta$.
\end{lemma}

\noindent
{\it Proof:} In Appendix \ref{app:miscproofs}.
\medskip

Lemma \ref{lem:BRpure} implies that we can speak of equilibria in pure or standard mixed strategies in the sense of Definition \ref{def:SPE_extended}, i.e., such that deviations to arbitrary extended mixed strategies are allowed, but that it suffices to consider only other pure or standard mixed strategies.

Specifically, Lemma \ref{lem:BRpure} indicates that $G_i^\vartheta$ will be a best reply to $G_j^\vartheta$ if and only if for any stopping time $\tau^*$ with $dG_i^\vartheta(\tau^*)>0$\footnote{%
This means that $G_i^\vartheta(t)>G_i^\vartheta(\tau-)$ for all $t>\tau$ a.s.
}
it holds that
\begin{equation*}
E\bigl[S_i^\vartheta(\tau^*)\bigv\F_\vartheta\bigr]\geq E\bigl[S_i^\vartheta(\tau)\bigv\F_\vartheta\bigr]\qquad\text{for all }\tau\in\T\text{ with }\tau\geq\vartheta
\end{equation*}
(see the proof of Theorem \ref{thm:mixedeql} in Appendix \ref{app:miscproofs} for details). Therefore, we generally need to solve the stopping problem on the right-hand side in Lemma \ref{lem:BRpure}.

A central aspect of continuous-time games of timing is their inherent discontinuity, even if the underlying data (here $L$, $F$, and $M$) is continuous. For instance, from the definition of $S_i^\vartheta$ in \eqref{Si} it is clear that a best reply cannot have any joint mass points when $F>M$, because $S_i^\vartheta(t+)-S_i^\vartheta(t)=\Delta G_j^\vartheta(t)\bigl(F_t-M_t\bigr)$ by right-continuity of $L$ and $G_j^\vartheta$; this will be a frequent argument. Depending on $G_j^\vartheta$, there need not exist any stopping time that actually attains the value of the problem, as $S_i^\vartheta$ may have various kinds of discontinuities. Dealing with such discontinuities will be one of the major issues.

In the following subsection, we present some crucial facts from the general theory of optimal stopping in continuous time, providing in particular sufficient (and basically necessary) conditions for the existence of optimal stopping times and their characterization in terms of the Snell envelope. The latter is in fact our main tool to derive and represent mixed equilibrium strategies.

\subsection{Optimal stopping in continuous time}\label{subsec:optstop}

As a motivating stopping problem to present the theory, consider the unilateral problem of when to become the leader optimally, i.e., supposing the opponent will never move. This problem will play an important role in the following.\footnote{%
To stay in the framework of the game and to find a (pure) best reply to $G^0_j$ given by $\indi{t\geq\infty}$, we have to use the payoff $M_\infty$ for not stopping in finite time. Recall our convention $L_\infty\equiv M_\infty$, however. 
}

It is well established how to characterize the solution of the optimal stopping problem
\begin{equation*}
V_L(0):=\esssup_{\tau\in\T}E\bigl[L_\tau\bigr]
\end{equation*}
given Assumption \ref{asm:payoffs}. In fact, our payoff process $L$ is right-continuous (hence optional, cf.\ fn.\ \ref{fn:optional}) and of class {\rm (D)}, so we can apply the general theory of optimal stopping as in, e.g., \cite{Mertens72} and \cite{BismutSkalli77}:
There exists a smallest supermartingale $U_L$ dominating the payoff process $L$, called the \emph{Snell envelope} of $L$, which satisfies
\begin{equation}\label{UL}
U_L(\vartheta)=\esssup_{\tau\in\T\colon\tau\geq\vartheta}E\bigl[L_\tau\bigv\F_\vartheta\bigr]\quad\text{a.s.}
\end{equation}
for all stopping times $\vartheta\in\T$. In particular, $U_L(0)=V_L(0)$. We remark that one can well define the right-hand side of \eqref{UL} for any $\vartheta\in\T$, but the key insight is that there exists a well behaved \emph{process} $U_L=(U_L(t))_{t\geq 0}$ that can be evaluated at any stopping time $\vartheta$ to know the continuation value then. In view of the dynamic programming principle, we also need to consider continuation problems at stopping times; the latter are feasible quantities, but much richer than deterministic times. 

Now $U_L$ is optional and of class {\rm (D)} as well,\footnote{%
See \cite{Mertens72}, Th\'eor\`eme T4 for the existence and Th\'eor\`eme T5 and proof for $U_L$ being of class {\rm (D)}.
}
and such supermartingales have very convenient regularity properties: There exists a \emph{Doob-Meyer decomposition}\footnote{%
See \cite{Mertens72}, Th\'eor\`eme T3.
}
\begin{equation*}
U_L=M_L-D_L
\end{equation*}
that we extensively use, with a uniformly integrable, right-continuous martingale\footnote{%
Therefore, the crucial optional sampling holds: $M_L(\sigma)=E[M_L(\tau)\mid\F_\sigma]$ for all $\sigma\leq\tau\in\T$. Moreover, $M_L$ has a last element $M_L(\infty)$ to which it converges in $L^1(P)$.
}
$M_L$ and a \emph{nondecreasing}, predictable and integrable process $D_L$. The latter can be interpreted as measuring the \emph{expected loss from stopping too late}: If we ignore to stop before any $\tau\in\T$, then we cannot achieve more than $E[U_L(\tau)]=U_L(0)-E[D_L(\tau)]$, even if we stop optimally from $\tau$ onwards. 

Reflecting the dynamic programming principle, the value process $U_L$ is a martingale as long as there still exists a future time $\tau\in\T$ giving at least the same value in expectation as stopping immediately. Whether there exists any \emph{optimal} stopping time depends on the continuity properties of $D_L$. If $L$ is upper-semi-continuous in expectation (as by Assumption \ref{asm:payoffs}\,\ref{LFusc} if $L\leq F$, e.g.), then $D_L$ has left-continuous paths a.s.\footnote{\label{fn:Lusc}%
See \cite{BismutSkalli77}, Th\'eor\`eme II.2 and proof. (Semi-) Continuity in expectation is in general weaker than the corresponding path property from the left.

Our payoff processes are not necessarily positive. However, if $L$ is optional and of class {\rm (D)}, the same will be true for its negative part $L^-:=\max(-L,0)$, which thus has a Snell envelope $U_{L^-}=M_{L^-}-D_{L^-}$ decomposing into a uniformly integrable right-continuous martingale $M_{L^-}$ and an integrable increasing process $D_{L^-}$. Then $M_{L^-}-L^-\geq 0$, implying $L+M_{L^-}\geq 0$. Adding the martingale $M_{L^-}$ neither affects $L$ being optional, of class {\rm (D)}, or (semi-) continuous in expectation, nor any optimal stopping times for $L$.
}
By right-continuity of $L$, $D_L$ will be even continuous.\footnote{\label{fn:Lrc}%
See \cite{BismutSkalli77}, (2.15), where right-continuity of the payoff process implies in fact $Z^+=X$.
}
With left-continuous $D_L$, there exist the optimal stopping times\footnote{%
\emph{Example}: $L$ not upper-semi-continuous $\Rightarrow$ $\inf\{D_L>0\}$ not optimal.

\begin{minipage}[b]{0.4\linewidth}
\centering
   \begin{tikzpicture}[inner sep=0pt,minimum size=0pt,label distance=3pt]
    \draw[->] (-0.1,0) -- (4,0) {}; 
    \draw[->] (0,-0.3) -- (0,2) {}; 
    \draw[-] (0,0.8) -- (1.8,0.8) [] {};
    \draw[-] (1.8,1.5) -- (3.5,1) [] {};
    \draw[dotted] (0,1.5) -- (3.5,1.5) []{};
    \draw[dashed] (1.8,0) -- (3.5,0.5)[]{};
    \fill[black] (1.8,0.8) circle (.04);
    \filldraw[fill=white,draw=black] (1.8,1.5) circle (.04);

    \node at (3.5,1.5) [label=right:$M_L$] {};
    \node at (3.5,0.9) [label=right:$L$] {};
    \node at (3.5,0.35) [label=right:$D_L$] {};
  \end{tikzpicture}
\end{minipage}
\begin{minipage}[b]{0.4\linewidth}
\centering
   \begin{tikzpicture}[inner sep=0pt,minimum size=0pt,label distance=3pt]
    \draw[->] (-0.1,0) -- (4,0) {}; 
    \draw[->] (0,-0.3) -- (0,2) {}; 
    \draw[-] (0,0.8) -- (1.8,1.5) [] {};
    \draw[-] (1.8,1) -- (3.5,1) [] {};
    \draw[dotted] (0,1.5) -- (3.5,1.5) []{};
    \draw[dashed] (1.8,0.5) -- (3.5,0.5)[]{};
    \fill[black] (1.8,1) circle (.04);
    \fill[black] (1.8,0.5) circle (.04);
    \filldraw[fill=white,draw=black] (1.8,1.5) circle (.04);
    \filldraw[fill=white,draw=black] (1.8,0) circle (.04);

    \node at (3.5,1.5) [label=right:$M_L$] {};
    \node at (3.5,0.9) [label=right:$L$] {};
    \node at (3.5,0.35) [label=right:$D_L$] {};
  \end{tikzpicture}
\end{minipage}
}
\begin{align}\label{tauopt}
\tau_L^{*}(\vartheta):=\inf\bigl\{t\geq\vartheta\bigv U_L(t)=L_t\bigr\}\quad\text{and}\quad\tau_L^{**}(\vartheta):=\inf\bigl\{t\geq\vartheta\bigv D_L(t)>D_L(\vartheta-)\bigr\}.\hphantom{,}
\end{align}
They are the respectively smallest and largest stopping times after $\vartheta\in\T$ attaining\footnote{%
See \cite{BismutSkalli77}, Th\'eor\`eme II.3.
}
\begin{equation}\label{U_L=L}
U_L(\vartheta)=E\Bigl[L_{\tau_L^{*}(\vartheta)}\Bigv\F_\vartheta\Bigr]=E\Bigl[L_{\tau_L^{**}(\vartheta)}\Bigv\F_\vartheta\Bigr]\quad\text{a.s.}
\end{equation}
Hence, by optimality it must hold that $U_L=L$ a.s.\ at any point of increase of $D_L$, which implies in fact\footnote{\label{fn:(U_L-L)dD_L=0}%
For $L$ right-continuous and upper-semi-continuous in expectation, $D_L$ is continuous, so $U_L$ inherits right-continuity from $M_L$. Then, by \eqref{tauopt}, \eqref{U_L=L} and right-continuity of $U_L-L$, $\inf\{t\in\R_+\mid\int_0^t\indi{U_L-L\geq\varepsilon}\,dD_L>0\}=\infty$ a.s.\ for any $\varepsilon>0$, i.e., $U_L-L<\varepsilon$ \nbd{dD_L}a.e.\ with probability one, implying the claim. \eqref{(U_L-L)dD_L=0} still holds without right-continuity of $U_L-L$, as long as $L$ is upper-semi-continuous in expectation; see Remark \ref{rem:eqlLusc} in the appendix.
}
\begin{equation}\label{(U_L-L)dD_L=0}
\int_{[0,\infty]}(U_L(t)-L_t)\,dD_L(t)=0\quad\text{a.s.}
\end{equation}

\section{Equilibria in pure strategies}\label{sec:eqlpure}

In symmetric games with systematic second-mover advantage $F\geq L$, it is straightforward to identify certain subgame-perfect equilibria in pure strategies. Player $j$, say, just has to stop sufficiently late, such that $i$ will solve the problem of optimally stopping $L$ presented in Section \ref{subsec:optstop}. We show in this section that such pure strategy equilibria typically entail asymmetric payoffs, however. The respective roles of the players have to be determined before the game starts, and correspondingly who obtains the higher payoff. With mixed strategies that we will consider thereafter,  equilibria with symmetric payoffs are obtained that do not create another strategic conflict outside the model.

Stopping sufficiently late to support a pure strategy equilibrium need not be ``never'': Whenever it is optimal to stop $L$, it simply must not be worthwhile for player $i$ to wait until $j$ stops, in order to become follower then. This will be the case, e.g., if $j$ only stops when $F=L$~-- like at $\tau_j\equiv\infty$.

The easiest example is thus $G_j^\vartheta$ given by $\indi{t=\infty}$ and $G_i^\vartheta$ by $\indi{t\geq\tau_L^{*}(\vartheta)}$ for every $\vartheta\in\T$, or analogously with $\tau_L^{**}(\vartheta)$ defined in \eqref{tauopt}. In either case, waiting is indeed optimal for player $j$ on $[0,\infty)$, because it is never strictly better to realize $L$ before an optimal stopping time, and $F$ dominates $L$ at both $\tau_L^{*}(\vartheta)$ and $\tau_L^{**}(\vartheta)$.

There can also be quite complex patterns based on the same logic, but with players switching roles across subgames. This can be illustrated best with a little more structure like in Example \ref{exm:attrition}, where the follower's optimal stopping times are ``sufficiently late'' for an equilibrium. However, the arguments are more general: The exploited properties are $F\geq L\geq M$ and that $F$ is a supermartingale, i.e., becoming follower sooner is better than later.\footnote{%
These also hold in \emph{phases} with second-mover advantage of typical market entry games; see \cite{Steg18}.
} 
Then finding a stopping time that is optimal from $\tau_L^{*}(\vartheta)$ on is enough (for any $G_j^\vartheta$).

\begin{lemma}\label{lem:tauLdom}
Suppose $F\geq L\geq M$ and that $F$ is a supermartingale. For any $\vartheta\in\T$, standard mixed strategy $G_j^\vartheta$ in the corresponding subgame, and stopping time $\tau_i\geq\vartheta$ it holds that
\begin{flalign*}
&& E&\Bigl[S_i^\vartheta\bigl(\bigl(\tau_i\vee\tau_L^{*}(\vartheta)\bigr)+\bigr)\Bigv\F_\vartheta\Bigr]\geq E\Bigl[S_i^\vartheta\bigl(\tau_i\bigr)\Bigv\F_\vartheta\Bigr] &&\\
\text{and}\displaybreak[0]\\
&& E&\Bigl[S_i^\vartheta\bigl(\tau_L^{*}(\vartheta)+\bigr)\Bigv\F_\vartheta\Bigr]\geq E\Bigl[L_{\tau_L^{*}(\vartheta)}\Bigv\F_\vartheta\Bigr]\qquad\qquad\text{a.s.}&&
\end{flalign*}
So, if a stopping time $\tau_i^*\geq\tau_L^{*}(\vartheta)$ attains $\esssup_{\tau\geq\tau_L^{*}(\vartheta)}E[S_i^\vartheta(\tau)\mid\F_{\tau_L^{*}(\vartheta)}]$, then it also attains $E[S_i^\vartheta(\tau_i^*)\mid\F_\vartheta]=\esssup_{\tau\geq\vartheta}E[S_i^\vartheta(\tau)\mid\F_\vartheta]\geq E[L_{\tau_L^{*}(\vartheta)}\mid\F_\vartheta]$. All claims also hold with $\tau_L^{**}(\vartheta)$ instead.
\end{lemma}

\noindent
{\it Proof:} In Appendix \ref{app:miscproofs}.
\medskip

Note that the supermartingale property of $F$ is important for the result, to ensure relatively high payoffs for the case of becoming follower before the optimum of $L$ is reached. It is not enough that there are even strictly better future stopping times for $L$ and that $F\geq L$: If $G_j^\vartheta$ puts mass between $\vartheta$ and $\tau_L^{*}(\vartheta)$, when $F$ still dominates $L$ but when both are very low, then it may be worthwhile to secure the current payoff $L_\vartheta$ due to the risk of becoming follower while waiting for the optimum of $L$. An alternative condition would be that $L$ is a \emph{submartingale} on $[\vartheta,\tau_L^{*}(\vartheta)]$.

In Example \ref{exm:attrition}, $L=\int_0^\cdot\pi^D\,ds$ is the duopolists' payoff process. Then the optimal stopping times in the follower's problem are sufficiently late to support an equilibrium: the perspective to become follower (monopolist) at a time when immediate exit is optimal has no value, and it leads to ceding when $\pi^D$ seems an unsustainable loss~-- at $\tau_L^{*}(\vartheta)$. Indeed, as a monopolist stops $\int_0^\cdot\pi^M\,ds$ with $\pi^M\geq\pi^D$, that optimal stopping time satisfies $\tau_F(\vartheta)\geq\tau_L^{*}(\vartheta)$. Furthermore, it holds that $F=L=M$ a.s.\ at $\tau_F(\vartheta)$, so in particular simultaneous stopping is feasible on $\{\tau_L^{*}(\vartheta)=\tau_F(\vartheta)\}$ by $F=M$. These properties generate a whole class of equilibria with varying roles of the players, decided by events $C$ at $\tau_L^{*}(\vartheta)$.

\begin{proposition}\label{prop:pureeql}
Suppose $F\geq L\geq M$ and that $F$ is a supermartingale. Let $\vartheta\in\T$ and consider a stopping time $\tau_F(\vartheta)\geq\tau_L^{*}(\vartheta)$ a.s., such that at $\tau_F(\vartheta)$ we have $F=L$, and more specifically $F=M$ on the set $\{\tau_F(\vartheta)=\tau_L^{*}(\vartheta)\}$ (a.s.)~-- e.g., $\tau_F(\vartheta):=\inf\{t\geq\vartheta\mid F_t=M_t\}$. Then, for any given event $C\in\F_{\tau_L^{*}(\vartheta)}$, the pure strategies corresponding to
\begin{equation*}
\tau_1^*=\tau_L^{*}(\vartheta)\indi{C}+\tau_F(\vartheta)\indi{C^c}\quad\text{and}\quad\tau_2^*=\tau_L^{*}(\vartheta)\indi{C^c}+\tau_F(\vartheta)\indi{C}
\end{equation*}
form an equilibrium in the subgame starting at $\vartheta$.
\end{proposition}

\noindent
{\it Proof:} In Appendix \ref{app:miscproofs}.
\medskip

Equilibria in pure strategies typically involve asymmetric payoffs, for instance if $F>L$ at $\tau_L^{*}(\vartheta)$ in those that we have specified. Consequently, there arises a coordination problem before the start of the game, each player wanting to become follower eventually. This problem is even aggravated in the equilibria of Proposition \ref{prop:pureeql}, where the roles may switch across subgames. For this reason, such equilibria are also difficult to aggregate for a subgame-perfect equilibrium: for each subgame starting at some $\vartheta\in\T$, an event $C\in\F_{\tau_L^{*}(\vartheta)}$ has to be agreed on that determines the respective roles.

Maybe even more importantly, no player can obtain the preferred follower payoff by taking or threatening to take a certain action, but only by the threat of taking \emph{no} action for a longer time, which has to induce the opponent to stop. Effectively, players compete in the credibility of taking no action. Such problems can be avoided by allowing for mixed strategies, making the players indifferent about the roles when stopping occurs. This is our topic in the following.

\section{Equilibria in mixed strategies}\label{sec:eqlmix}

The very universal principle of the Snell envelope allows us to construct equilibria in mixed strategies in our general setting. But we do not only obtain existence: the equilibrium strategies can be clearly interpreted like the Snell envelope itself. Recall that the compensator relates to the expected loss from stopping too late. 

The logic of the following equilibria in symmetric games is this: If $F\geq L$, but without the other conditions of Lemma \ref{lem:tauLdom}, then waiting for future optimal times to stop $L$ is not necessarily \emph{always} optimal in the game. Nevertheless, the players have an incentive and the possibility to coordinate on waiting, which can be extended until the latest optimal time to stop $L$, $\tau_L^{**}(\vartheta)$. To cross that point, however, any player has to be compensated by some chance to become follower when $F>L$, because otherwise any delay would definitely be costly. Of course, the opponent has to be willing to provide that chance, so we identify the suitable rate to compensate \emph{exactly} the impending loss $dD_L>0$ and make both players indifferent.

This principle does not work when $L>F$, however, when the players would want to stop much more intensely due to preemption incentives (see Section \ref{sec:eqlsym}). On the other hand, even if we were considering games without first-mover advantage, then there might be equilibria with even higher symmetric payoffs~-- if simultaneous stopping is feasible and sufficiently profitable at some future time, precisely when $M\geq F>L$. For these reasons, we need to generalize the appropriate payoff process that the players coordinate on.

\begin{theorem}\label{thm:mixedeql}
Let $\vartheta,\tau^\vartheta\in\T$ be stopping times with $\tau^\vartheta\in[\vartheta,\,\inf\{t\geq\vartheta\mid L_t>F_t\}]$ a.s. Define the auxiliary process $\tilde L^{\tau^\vartheta}$ by $\tilde L_t^{\tau^\vartheta}:=\indi{t<\tau^\vartheta}L_t+\indi{t\geq\tau^\vartheta}\max(F_{\tau^\vartheta},M_{\tau^\vartheta})$, let $D_{\tilde L}^{\tau^\vartheta}$ denote the compensator of its Snell envelope and $\tau_i^\vartheta:=\inf\{t\geq\vartheta\mid\int_{[\vartheta,t]}\indi{F\leq L}\,dD_{\tilde L}^{\tau^\vartheta}>0\}\wedge\tau^\vartheta$. 

Then there exists a payoff-symmetric equilibrium in the subgame starting at $\vartheta$ with standard mixed strategies given by 
\begin{flalign}
&& G_i^\vartheta(t)&=1-\indi{t<\tau_i^\vartheta}\exp\biggl(-\int_\vartheta^t\frac{\indi{F_s>L_s}\,dD_{\tilde L}^{\tau^\vartheta}(s)}{F_s-L_s}\biggr) & \label{Geql}\\
&\text{and} \nonumber\displaybreak[0]\\
&& G_j^\vartheta(t)&=1-\indi{t<\tau^\vartheta}\exp\biggl(-\int_\vartheta^t\frac{\indi{F_s>L_s}\,dD_{\tilde L}^{\tau^\vartheta}(s)}{F_s-L_s}\biggr) & \label{Gjeql}
\end{flalign}
for $i,j\in\{1,2\}$, $i\neq j$, if and only if a.s.\ $\Delta G_i^\vartheta(M-F)\leq 0$ at $\tau_i^\vartheta$ on $\{\tau_i^\vartheta<\tau^\vartheta\}$ and $\Delta G_i^\vartheta(M-F)\geq 0$ at $\tau^\vartheta$.

Moreover, there exists a symmetric equilibrium with both players using the strategy given by \eqref{Geql} if and only if a.s.\ $\Delta G_i^\vartheta(M-F)\geq 0$ at $\tau_i^\vartheta$, with equality on $\{\tau_i^\vartheta<\tau^\vartheta\}$ whenever $\Delta G_i^\vartheta(\tau_i^\vartheta)<1$.
\end{theorem}

\noindent
{\it Proof:} In Appendix \ref{app:miscproofs}.
\medskip

On $\{L_\vartheta>F_\vartheta\}$, $\tau^\vartheta=\vartheta$, and $M_\vartheta\geq F_\vartheta$ is required to support simultaneous stopping. On $\{F_\vartheta\geq L_\vartheta\}$, the ``terminal condition'' at $\tau_i^\vartheta$ seems contradictory, but it plays the following role. First consider $\Delta G_i^\vartheta(\tau^\vartheta)>0$, so $\tau_i^\vartheta=\tau^\vartheta$, and there is a joint terminal jump. This means that the players coordinate on the terminal payoff $M_{\tau^\vartheta}$, which again requires $M_{\tau^\vartheta}\geq F_{\tau^\vartheta}$ (where $M_\infty=F_\infty$ by convention). $G_i^\vartheta$ can also jump to one before $\tau^\vartheta$: when $F=L$ and $dD_{\tilde L}^{\tau^\vartheta}>0$, such that no compensation is possible. Then player $i$ stops, but we keep $G_j^\vartheta$ continuous to address the case $F=L>M$, and payoffs are hence symmetric. This choice can only be an equilibrium, however, if indeed $F\geq M$. Otherwise, player $j$ could obtain a higher payoff by stopping at $\tau_i^\vartheta$ and not supporting the equilibrium earlier on, but we could then adjust $\tau^\vartheta$ to ensure suitable continuation values. Finally, if $G_i^\vartheta$ reaches the value one continuously, then $G_j^\vartheta(t)=G_i^\vartheta(t)$ for all $t\in\R_+$. This case is one reason for defining $\tilde L^{\tau^\vartheta}$ with $\max(F_{\tau^\vartheta},M_{\tau^\vartheta})$: Then $\Delta G_i^\vartheta(M-F)\geq 0$ holds at $\tau^\vartheta$ by $\Delta G_i^\vartheta=0$, which allows $\tilde L^{\tau^\vartheta}$ to have the terminal value $F_{\tau^\vartheta}>M_{\tau^\vartheta}$ and affects the continuation values before. Another reason is that we will obtain continuation equilibria with payoff $\max(F_{\tau^\vartheta},M_{\tau^\vartheta})$ when considering \emph{extended} mixed strategies in Section \ref{sec:eqlsym}.

Except for possible terminal jumps, the strategies in Theorem \ref{thm:mixedeql} are continuous. As motivated before and given the appropriate process $\tilde L^{\tau^\vartheta}$, the opponent's stopping rate $dD_{\tilde L}^{\tau^\vartheta}/(F-L)$ makes each player indifferent when it would seem optimal to secure the current value of $L$. The expected loss from forgoing it is exactly compensated by the probability of obtaining $F>L$. The resulting equilibrium payoffs are given by
\begin{equation*}
V_i^\vartheta(G_i^\vartheta,G_j^\vartheta)=V_j^\vartheta(G_j^\vartheta,G_i^\vartheta)=\esssup_{\tau\in\T\colon\tau\geq\vartheta}E\Bigl[\indi{\tau<\tau^\vartheta}L_\tau+\indi{\tau\geq\tau^\vartheta}\max(F_{\tau^\vartheta},M_{\tau^\vartheta})\Bigv\F_\vartheta\Bigr]=:U_{\tilde L}^{\tau^\vartheta}(\vartheta),
\end{equation*} 
resp., if both players use $G_i^\vartheta$ and this is an equilibrium, by
\begin{align*}
V_i^\vartheta(G_i^\vartheta,G_i^\vartheta)=U_{\tilde L}^{\tau^\vartheta}(\vartheta)+E\Bigl[\indi{\tau_i<\tau^\vartheta}\indi{\Delta G_i(\tau_i)>0}\bigl(M_{\tau_i}-F_{\tau_i}\bigr)\Bigv\F_\vartheta\Bigr].
\end{align*}
These can be rewritten using the first time that a delay becomes costly for $\tilde L^{\tau^\vartheta}$, denoted by $\tau_{\tilde L}^{**}(\vartheta):=\inf\{t\geq\vartheta\mid D_{\tilde L}^{\tau^\vartheta}(t)>D_{\tilde L}^{\tau^\vartheta}(\vartheta)\}$; cf.\ \eqref{U_L=L}. Then $V_i^\vartheta(G_i^\vartheta,G_j^\vartheta)=E[\tilde L_{\tau_{\tilde L}^{**}(\vartheta)}^{\tau^\vartheta}\mid\F_\vartheta]$ and
\begin{equation*}
V_i^\vartheta(G_i^\vartheta,G_i^\vartheta)=E\Bigl[\tilde L_{\tau_{\tilde L}^{**}(\vartheta)}^{\tau^\vartheta}+\indinb{\{\tau_{\tilde L}^{**}(\vartheta)<\tau^\vartheta\}\cap\{L_{\tau_{\tilde L}^{**}(\vartheta)}=F_{\tau_{\tilde L}^{**}(\vartheta)}\}}\Bigl(M_{\tau_{\tilde L}^{**}(\vartheta)}-\tilde L_{\tau_{\tilde L}^{**}(\vartheta)}^{\tau^\vartheta}\Bigr)\Bigv\F_\vartheta\Bigr].
\end{equation*}

The proof of Theorem \ref{thm:mixedeql} is based on martingale arguments. An important aspect is to take care of the different kinds of jumps in the strategies and to ensure that the underlying payoff process $\tilde L^{\tau^\vartheta}$ has the necessary properties (e.g., that $D_{\tilde L}^{\tau^\vartheta}$ is continuous). Moreover, when stopping happens continuously, this may be at a rate with respect to time ($dt$) like in the explicit Brownian example in Section \ref{sec:expduo}, but not necessarily, and continuous strategies can also charge a set of time points of measure zero.\footnote{%
In this case, $dG_i^\vartheta$ would be a \emph{singular} measure, which often appear in optimal control of Brownian models.
}

The present equilibrium strategies are trivial~-- given that the endpoint is feasible~-- if $L$ or $F$ is a (sub-)martingale on $[\vartheta,\tau^\vartheta]$; then there is no loss from waiting and $dD_{\tilde L}^{\tau^\vartheta}\equiv 0$.\footnote{%
Cf.\ Theorem 3.3 in \cite{RiedelSteg17} for this case, but also their Section 4.3 on issues in asymmetric games.
}

\begin{remark}
It may happen that $D_{\tilde L}$~-- and hence $G_i^\vartheta$~-- has jumps if $L$ is only upper-semi-continuous from the right (and the left). In Theorem \ref{thm:mixedeql}, waiting is always at least as good as obtaining $L$ and there must be indifference at increases of $D_{\tilde L}$. Joint mass points are therefore impossible when $L>M$.\footnote{%
\emph{Example}: No symmetric payoff equilibrium if $L$ ($>M$) not right-continuous.

\medskip
\begin{minipage}[c]{0.25\linewidth}
\centering
   \begin{tikzpicture}[inner sep=0pt,minimum size=0pt,label distance=3pt]
    \draw[->] (-0.1,0) -- (3,0) {}; 
    \draw[->] (0,-0.3) -- (0,2) {}; 
    \draw[-] (0.8,0.2) -- (2.8,1) [] {};
    \draw[-] (0.8,1.8) -- (2.8,1) [] {};
    \draw[-] (0.8,1) -- (1.6,1.2) []{};
    \draw[-] (1.6,1) -- (2.8,1)[]{};
    \fill[black] (1.6,1.2) circle (.04);
    \filldraw[fill=white,draw=black] (1.6,1) circle (.04);

    \node at (0.8,1.8) [label=left:$F$] {};
    \node at (0.8,1) [label=left:$L$] {};
    \node at (0.8,0.2) [label=left:$M$] {};
    \node at (1.6,0) [label=below:$T_1$] {};
    \node at (2.8,0) [label=below:$T_2$] {};
  \end{tikzpicture}
\end{minipage}
\hfill
\begin{minipage}[c]{0.7\linewidth}
Waiting is strictly optimal for player $i$ at $t\in(T_1,T_2)$ if $G_j(T_2-)>G_j(t)$, so $G_i(T_2-)=G_i(T_1)$ and $G_j(T_2-)=G_j(T_1)$ by payoff symmetry, giving a continuation payoff $L(T_2)$ on $(T_1,T_2]$. The only symmetric continuation payoff at $T_1$ is then in $(L(T_2),L(T_1))$ from $\Delta G_i(T_1)=\Delta G_j(T_1)\in(0,1)$. Waiting is also strictly dominant on $[0,T_1)$, but stopping short of $T_1$ now yields a higher payoff than stopping at $T_1$.
\end{minipage}
}
However, Theorem \ref{thm:mixedeql} remains true if $L\equiv M$ (e.g., in an attrition model); see Remark \ref{rem:eqlLusc} in the appendix.
\end{remark}

The equilibria of Theorem \ref{thm:mixedeql} can so far only deal with subgames satisfying $M_\vartheta\geq F_\vartheta$ on $\{L_\vartheta>F_\vartheta\}$. Therefore, if we want to aggregate them for a subgame-perfect equilibrium, we have to assume this for all subgames for the moment. Then setting $\tau^\vartheta=\inf\{t\geq\vartheta\mid L_t>F_t\text{ or }M_t>F_t\}$ will satisfy the existence conditions. In the special case that $F\geq\max(L,M)$ throughout, this implies simply $\tau^\vartheta\equiv\infty$ and then $\tilde L^{\tau^\vartheta}=L$ and $D_{\tilde L}^{\tau^\vartheta}=D_L$. In this case, the stopping rates do not depend on $\vartheta$, which ensures time-consistency. 

In general, however, we may have $D_{\tilde L}^{\tau^\vartheta}\neq D_L$ on $[\vartheta,\tau^\vartheta]$, specifically when $\max(F,M)>L$ or $\max(F,M)<U_L$ at $\tau^\vartheta<\infty$. Then time-consistency requires that $\tau^\vartheta$ will not be changed if it has not been passed, yet: For any two $\vartheta,\vartheta'\in\T$, we should have $\tau^\vartheta=\tau^{\vartheta'}$ on $\{\vartheta\leq\vartheta'\leq\tau^\vartheta\}$ and vice versa, or in summary $\tau^\vartheta=\tau^{\vartheta'}$ on $\{(\vartheta\vee\vartheta')\leq(\tau^\vartheta\wedge\tau^{\vartheta'})\}$.

Then we indeed obtain payoff-symmetric subgame-perfect equilibria with standard mixed strategies for games, e.g., with systematic second-mover advantage.

\begin{theorem}\label{thm:SPE}
Fix $i,j\in\{1,2\}$, $i\neq j$. If $(G_i^\vartheta;\vartheta\in\T)$, $(G_j^\vartheta;\vartheta\in\T)$ are such that for every $\vartheta\in\T$, $G_i^\vartheta$, $G_j^\vartheta$ satisfy the conditions for an equilibrium in the subgame starting at $\vartheta$ as in Theorem \ref{thm:mixedeql}, and the associated $(\tau^\vartheta;\vartheta\in\T)$ satisfy $\tau^\vartheta=\tau^{\vartheta'}$ a.s.\ on $\{(\vartheta\vee\vartheta')\leq(\tau^\vartheta\wedge\tau^{\vartheta'})\}$ for all $\vartheta,\vartheta'\in\T$, then $(G_i^\vartheta;\vartheta\in\T)$ and $(G_j^\vartheta;\vartheta\in\T)$ represent a subgame-perfect equilibrium in standard mixed strategies. This is the case, e.g., if $\max(F_t-L_t,M_t-F_t)\geq 0$ for all $t\in\R_+$ a.s.\ and $\tau^\vartheta=\inf\{t\geq\vartheta\mid L_t>F_t\text{ or }M_t>F_t\}$ for every $\vartheta\in\T$.
\end{theorem}

\noindent
{\it Proof:} In Appendix \ref{app:miscproofs}.
\medskip

Even if the required time-consistency condition for the family $(\tau^\vartheta;\vartheta\in\T)$ holds, the family $(D_{\tilde L}^{\tau^\vartheta};\vartheta\in\T)$ needs to induce time-consistent stopping rates $(dG_i^\vartheta;\vartheta\in\T)$. This is the main point of (the proof of) Theorem \ref{thm:SPE}, given optimality by Theorem \ref{thm:mixedeql}.

\section{Example: Exit from duopoly}\label{sec:expduo}

In this section, we illustrate the simplification resulting from a systematic second-mover advantage pointed out in the context of Theorem \ref{thm:mixedeql}. Specifically, we determine subgame-perfect equilibrium strategies by explicitly deriving the Snell envelope $U_L$ and its compensator $D_L$ for a version of the market exit game in Example \ref{exm:attrition}. The stopping rate during attrition is then represented in terms of a sustained flow of losses from unprofitable operations. 

To specify the model, assume that at each time $t$, discounted duopoly profits are given by
\begin{equation*}
\pi^D_t = e^{-rt}(Y_t-c),
\end{equation*}
where $c>0$ is a constant operating cost and revenues $(Y_t)_{t\geq 0}$ follow a geometric Brownian motion solving $dY_t=\mu Y_t\,dt+\sigma Y_t\,dB_t$. The profit stream of a remaining monopolist is
\begin{equation*}
\pi^M_t = e^{-rt}(mY_t-c),
\end{equation*}
where $m>1$. Each firm can decide to leave the market with accumulated payoff $L_t=M_t=\int_0^te^{-rs}(Y_s-c)\,ds$, for example when $Y$ is so low that revenue does not cover production costs. In such a phase, the game is a war of attrition if monopoly still seems profitable. However, it may also be optimal to stop immediately in the follower's problem with payoff 
\begin{equation*}
F_t=L_t+\esssup_{\tau^F\in\T\colon\tau^F\geq t}E\biggl[\int_t^{\tau^F}e^{-rs}(mY_s-c)\,ds\biggv\F_t\biggl].
\end{equation*}
The latter problem is a standard exercise under the condition $r>\max(\mu,0)$,\footnote{%
This is also necessary and sufficient for the processes to be of class {\rm (D)} in accordance with Assumption \ref{asm:payoffs}. Then $-c/r\leq L_t\leq\int_0^\infty e^{-rs}\abs{Y_s-c}\,ds\in L^1(P)$ and similarly for $F$, inserting $m$.
}
and its unique solution is to stop as soon as $Y$, starting from $Y_t$, falls below the threshold
\begin{equation*}
y_m = \frac{\beta_2}{\beta_2-1}\frac{r-\mu}{r}\frac{c}{m}<\frac{c}{m},
\end{equation*}
where $\beta_2$ is the negative root of the quadratic equation $\frac{1}{2}\sigma^2\beta(\beta-1)+\mu\beta-r=0$. The value of the stopping problem can be explicitly expressed as
\begin{equation}\label{Fexpl}
F_t-L_t=e^{-rt}\indi{Y_t>y_m}\biggl[\frac{mY_t}{r-\mu}-\frac{c}{r}-\left(\frac{Y_t}{y_m}\right)^{\beta_2}\left(\frac{my_m}{r-\mu}-\frac{c}{r}\right)\biggr],
\end{equation}
which shows that $F$ is a continuous process, $F_\tau$ corresponds to an optimal follower decision from $\tau\in\T$, and $F_t=L_t\,(=M_t)\Leftrightarrow Y_t\leq y_m$. Hence, for any equilibrium as in Theorem \ref{thm:mixedeql}, we need $\tau^\vartheta\geq\inf\{t\geq\vartheta\mid Y_t\leq y_m\}$ for the endpoint condition. On the other hand, stopping is strictly dominant for a monopolist as soon as $Y_t\leq y_m$, and so it is in duopoly, where revenues can never exceed those in monopoly. Therefore, we can choose $\tau^\vartheta=\inf\{t\geq\vartheta\mid Y_t\leq y_m\}$ without loss of generality and it will lead to a symmetric equilibrium at any $\vartheta\in\T$ as follows. 

As $F$ is a supermartingale by $m\geq 1$ and dominates $L$, it also dominates the Snell envelope $U_L$ of the latter, such that we have $F=L\Rightarrow F=U_L=L$. Consequently, $\tilde L^{\tau^\vartheta}$ from Theorem \ref{thm:mixedeql} is here just $L$ stopped at $\tau^\vartheta$, and the Snell envelope $U_{\tilde L}^{\tau^\vartheta}$ coincides with $U_L$ until $\tau^\vartheta$. Applying the right-hand side of \eqref{Fexpl} with $m=1$ yields the solution to optimally stopping the leader (duopoly) payoff:
\begin{align*}
U_L(t)={}&\esssup_{\tau\in\T\colon\tau\geq t}E[L_\tau\mid\F_t]=L_t+e^{-rt}\indi{Y_t>y_1}\biggl[\frac{Y_t}{r-\mu}-\frac{c}{r}-\left(\frac{Y_t}{y_1}\right)^{\beta_2}\left(\frac{y_1}{r-\mu}-\frac{c}{r}\right)\biggr].
\end{align*}
Applying It\=o's lemma shows that the monotone part of the supermartingale $U_L$ is just the drift
\[dD_L(t)=-\indi{Y_t<y_1}\,dL_t=\indi{Y_t<y_1}e^{-rt}(c-Y_t)\,dt
\]
when stopping $L$ immediately is optimal. With $\tau^\vartheta=\inf\{t\geq\vartheta\mid Y_t\leq y_m\}$ for every $\vartheta\in\T$, $dD^{\tau^\vartheta}_{\tilde L}=dD_L$, and \eqref{Fexpl}, we now have a fully explicit symmetric subgame-perfect equilibrium, with payoffs $V_i^\vartheta(G_i^\vartheta,G_j^\vartheta)=U_L(\vartheta)$, respectively.

As $y_m<y_1<c$, we see that $dD_L$ is simply the stream of losses resulting from unprofitable operations. If a duopolist never hoped to become monopolist, these losses would be too large to keep operating. Here, whenever $Y\in(y_m,y_1)$, both firms are leaving duopoly at a rate that depends directly on those running losses; it is decreasing in $Y$. The state may rise next to the region $(y_1,c)$. Then there are still running losses, but the firms suspend mixing because the option to wait for a market recovery is sufficiently valuable. There is thus no need for a compensation. Typically, there will be alternating periods of continuous and no mixing. When the state drops to $[0,y_m]$, however, the option to wait for market recovery would be worthless in the face of running losses even if a firm was (sure to become) monopolist, and both firms quit immediately.

\section{Equilibria for general symmetric games}\label{sec:eqlsym}

\subsection{Preemption with extended mixed strategies}\label{subsec:extended}

In a preemption situation, i.e., when there is a first-mover advantage $L>F$, there typically exist no equilibria in pure strategies in continuous time. \cite{FudenbergTirole85} and \cite{HendricksWilson92} show that this issue arises when there is an incentive to wait ($L$ is increasing). If the model is sufficiently regular and the first-mover advantage is strict, then one may have equilibria in standard mixed strategies, with one player stopping immediately and the other stopping at a sufficient rate, such that the first would not be able to realize the increase in $L$. The payoffs are then asymmetric, $L$ and $F$. However, these equilibria cannot be extended to the boundary of the preemption region; if $L=F$, then the stopping rate needed to support any equilibrium explodes. This observation does not depend on any regularity conditions~-- the payoff processes can be deterministic and arbitrarily smooth.

Therefore, if we want to allow for any equilibria when preemption is about to start (or also symmetric payoffs when $L>F$), then we need to enrich the strategy and outcome spaces. The key is to facilitate some partial coordination when players try to stop at the same time, but when simultaneous stopping would be the worst outcome. Therefore, we are now going to use the strategy extensions $\alpha_i^\vartheta$ from Definition \ref{def:alpha}. With these extended strategies, it is possible to capture continuous-time limits of symmetric, mixed discrete-time equilibria, which do not have the previous issues.\footnote{%
In discrete time, there can be equilibria with a positive probability of simultaneous stopping even if that is the worst outcome, because the players can only assign positive probabilities to the single periods; one cannot circumvent coordination failure by stopping an arbitrarily small time $\varepsilon>0$ after a mass point of the opponent. See \cite{Steg17} for a discretized preemption model and a formal limit analysis.
}

We then obtain the following equilibria of immediate stopping for subgames with a first-mover advantage~-- here for a symmetric game:\footnote{%
The present extension of Proposition 3.1 in \cite{RiedelSteg17} to a nonempty set $\{M>F\}$ is straightforward in the symmetric case.
}

\begin{proposition}[{Cf.\ Proposition 3.1 in \cite{RiedelSteg17}}]\label{prop:eqlL>F}
Fix $\vartheta\in\T$ and suppose $\vartheta=\inf\{t\geq\vartheta\mid L_t>F_t\}$ a.s. Then $(G_1^\vartheta,\alpha_1^\vartheta)$, $(G_2^\vartheta,\alpha_2^\vartheta)$ defined by
\begin{equation*}
\alpha_i^\vartheta(t)=\begin{cases}
1 & \text{if}\quad M_t\geq F_t\text{ and }t=\inf\{s\geq t\mid L_s>F_s\},\\[6pt]
\displaystyle\indi{L_t>F_t}\frac{L_t-F_t}{L_t-M_t} & \text{else}
\end{cases}
\end{equation*}
for any $t\in[\vartheta,\infty)$ and $G_i^\vartheta=\indi{t\geq\vartheta}$, $i=1,2$, are an equilibrium in the subgame at $\vartheta$.

The resulting payoffs are $V_i^\vartheta(G_i^\vartheta,\alpha_i^\vartheta,G_j^\vartheta,\alpha_j^\vartheta)=\max(F_\vartheta,M_\vartheta)$.
\end{proposition}

\begin{remark}\label{rem:eqlL>F}
When $L>F\geq M$, then the choice of $\alpha_i^\vartheta$ makes the respective other player indifferent between stopping and waiting, and $\alpha_i^\vartheta(\cdot)$ is right-continuous, allowing a limit outcome argument. When $M>F$, stopping is of course the unique best reply. In the polar case $L_t=F_t=M_t$, there might not be a right-hand limit of $\indi{L_t>F_t}\frac{L_t-F_t}{L_t-M_t}$, so we set $\alpha_i^\vartheta(t)=1$. If the limit does exist, then it can be used to make $\alpha_i^\vartheta(\cdot)$ right-continuous even here, as the players will be indifferent in this case.
\end{remark}

If $L_\vartheta=F_\vartheta>M_\vartheta$, then each player becomes leader or follower with probability $\frac{1}{2}$.\footnote{%
Then the $\liminf$ and $\limsup$ in Definition \ref{def:outcome} are both $\frac{1}{2}$ with the strategies of Proposition \ref{prop:eqlL>F}.
}
This is the same outcome as in \cite{FudenbergTirole85} for their smooth, deterministic model. If $L_\vartheta>F_\vartheta>M_\vartheta$, however, then there is a positive probability of simultaneous stopping, which is the price of preemption, driving the payoffs down to $F_\vartheta$.

\subsection{General symmetric equilibria}\label{subsec:eqlsym}

We can now combine the equilibria obtained for $F\geq L$ or $L>F$. With standard mixed strategies, the equilibria for a (temporary) second-mover advantage from Theorem \ref{thm:mixedeql} depend on the ``terminal condition'' $\Delta G_i^\vartheta(M-F)\geq 0$, e.g., when a preemption regime may start with both players trying to stop immediately. Proposition \ref{prop:eqlL>F}, however, gives us ``continuation'' equilibria of immediate stopping at such transitions with payoffs $\max(F,M)$. Indeed, if player $j$ uses an \emph{extended} mixed strategy, then the payoff difference for player $i$ between stopping and waiting when $G_j^\vartheta$ jumps to one at $\hat\tau_j^\vartheta=\inf\{t\geq\vartheta\mid\alpha_j^\vartheta(t)>0\}$ changes from $\Delta G_j^\vartheta(M-F)$ to 
\begin{equation*}
\Delta G_j^\vartheta\bigl(\alpha_j^\vartheta M+\bigl(1-\alpha_j^\vartheta\bigr)L-F\bigr),
\end{equation*}
which is nonnegative if $\alpha_j^\vartheta$ is as in Proposition \ref{prop:eqlL>F}. Therefore, this possibility to coordinate partially in preemption also generates suitable endpoints for attrition regimes when we cannot have $M\geq F$ before reaching $\{L>F\}$.

We thus obtain a payoff-symmetric subgame-perfect equilibrium for any symmetric timing game~-- where the payoff processes $L$, $F$ and $M$ do not depend on the individual players.

\begin{theorem}\label{thm:symeql}
Under Assumption \ref{asm:payoffs}, there exists a payoff-symmetric subgame-perfect equilibrium in extended mixed strategies $((G_1^\vartheta,\alpha_1^\vartheta);\vartheta\in\T)$, $((G_2^\vartheta,\alpha_2^\vartheta);\vartheta\in\T)$ given as follows:

\medskip
Pick $i,j\in\{1,2\}$, $i\neq j$. For any $\vartheta\in\T$, set $\tau^\vartheta:=\inf\{t\geq\vartheta\mid L_t>F_t\text{ or }M_t>F_t\}$. Define $G_i^\vartheta$, $G_j^\vartheta$ as in Theorem \ref{thm:mixedeql} and $\alpha_i^\vartheta=\alpha_j^\vartheta$ as in Proposition \ref{prop:eqlL>F}.

\medskip
Moreover, if $L$, $F$, and $M$ are such that for every stopping time $\tau\in\T$ it holds that $L_\tau=F_\tau\Rightarrow F_\tau=M_\tau\text{ or }\tau=\inf\{t>\tau\mid L_t>F_t\}$ a.s., then there is a symmetric subgame-perfect equilibrium where each player's strategy is given by $(G_i^\vartheta,\alpha_i^\vartheta)$ for every $\vartheta\in\T$.
\end{theorem}

\noindent
{\it Proof:} In Appendix \ref{app:miscproofs}.
\medskip

The idea of these equilibria is basically pasting the war of attrition on $\{F\geq L\}$ using the continuous strategies from Theorem \ref{thm:mixedeql} with the preemption equilibria of immediate stopping on $\{L>F\}$ by extended mixed strategies from Proposition \ref{prop:eqlL>F}. However, we had to prepare well for doing so, because the precise attrition behavior depends strongly on the continuation payoffs, resp.\ when preemption starts.

By the upper-semi-continuity of Assumption \ref{asm:payoffs}\,\ref{LFusc}, it feasible during attrition that the players coordinate on a future continuation equilibrium with payoffs $\max(F,M)$. Then there will be no predictable drop in payoffs from starting preemption. The corresponding symmetric equilibrium payoffs are given by
\begin{equation*}
\esssup_{\tau\in\T\colon\tau\geq\vartheta}E\Bigl[\indi{\tau<\tau^\vartheta}L_\tau+\indi{\tau\geq\tau^\vartheta} \max(F_{\tau^\vartheta},M_{\tau^\vartheta})\Bigv\F_\vartheta\Bigr]\end{equation*}
with $\tau^\vartheta=\inf\{t\geq\vartheta\mid L_t>F_t\text{ or }M_t>F_t\}$ for any $\vartheta\in\T$.

Whereas the endpoint condition $\Delta G_i^\vartheta(M-F)\geq 0$ at $\tau^\vartheta$ is now replaced by the preemption continuation equilibria, we still need to ensure the second one, $\Delta G_i^\vartheta(M-F)\leq 0$ at $\tau_i^\vartheta$; imposing the cap $\tau^\vartheta\wedge\inf\{t\geq\vartheta\mid M_t>F_t\}$ works in general, but there may also be alternative choices in more specific cases. The proof of Theorem \ref{thm:symeql} relies of course on those of Theorem \ref{thm:mixedeql} and Proposition \ref{prop:eqlL>F}. The main issue is that the former was formulated in a reduced setting with ``standard'' mixed strategies, so we establish a formal relation to the present setting with extended mixed strategies.

\section{Optimal symmetric equilibrium}\label{sec:symeql}

The equilibria of Theorem \ref{thm:symeql} involve the most aggressive preemption that is conceivable~-- it happens whenever $L>F$ and just for this reason. Their structure is thus relatively simple: the game ends as soon as there is a strict first-mover advantage. Preemption need not be that severe if there are future continuation equilibria with sufficiently high (expected) payoffs. In this section, we identify equilibria with least possible preemption and thus entailing the highest attainable equilibrium payoffs. We focus on the class of \emph{payoff-symmetric equilibria}, which are the subgame-perfect equilibria with $V^\vartheta_1=V^\vartheta_2$ a.s.\ at any stopping time $\vartheta\in\T$. These have clear implications for equilibrium strategies. In competitive games, where $M$ is throughout the lowest payoff, equilibrium payoffs are then at most the value of optimally stopping $\min(L,F)$~-- no matter how players mix, possibly using public correlation (Proposition \ref{prop:U_min(L,F)}). This bound on equilibrium payoffs enables us to identify inevitable preemption points: when the leader payoff $L$ exceeds any continuation equilibrium payoff. Theorem \ref{thm:maxeql} formulates a corresponding algorithm and establishes the existence of an ``optimal'' subgame-perfect equilibrium.

It is quite clear that any stopping on $\{F>L\}$ must induce the lower payoff $L$ if $M$ is not better. The basis of our argument is the more subtle result that players also cannot exploit $L>F$ by mixing in any payoff-symmetric equilibrium, even if they have no time constraint.

\begin{proposition}\label{prop:U_min(L,F)}
Suppose $M\leq\min(L,F)$. Then, in any payoff-symmetric subgame-perfect equilibrium and for any $\vartheta\in\T$ and $i,j\in\{1,2\}$ with $i\neq j$,
\begin{equation*}
V_i^\vartheta\bigl(G_i^\vartheta,\alpha_i^\vartheta,G_j^\vartheta,\alpha_j^\vartheta\bigr)\leq\esssup_{\tau\in\T\colon\tau\geq\vartheta}E\bigl[L_\tau\wedge F_\tau\bigv\F_\vartheta\bigr]=:U_{L\wedge F}(\vartheta),
\end{equation*}
where it is in fact enough to consider to stopping times $\tau\leq\inf\{t\geq\vartheta\mid G_i^\vartheta(t)\vee G_j^\vartheta(t)\geq 1\}$.
\end{proposition}

The proof of Proposition \ref{prop:U_min(L,F)} in Appendix \ref{app:U_min(L,F)} is based on the following important facts for any payoff-symmetric equilibrium (which do not depend on the assumption $M\leq\min(L,F)$, yet): First, the conditional stopping probabilities of the players must be the same on $\{F\neq L\}$ (Lemma \ref{lem:dGi=dGj}), because a player who stops with a higher conditional probability also becomes leader with a higher conditional probability, whereas the other becomes follower on that event. As one consequence, $G_1^\vartheta$ and $G_2^\vartheta$ must then even be identical before they put any mass on $\{F=L\}$ (Lemma \ref{lem:Gi=Gj}). Moreover, on $\{F\neq L\}$, there can only be simultaneous jumps, and these are only possible when $M\geq F$ or when preemption occurs with $L\geq F>M$. Most importantly, there cannot be any jumps when $F>\max(L,M)$. Finally, the local payoff from any terminal jump is bounded by $\max(F,M)$ (Lemma \ref{lem:DG>0}).

The intuition for Proposition \ref{prop:U_min(L,F)} is now the following. In equilibrium, player $i$ must be willing to wait until any time at which still $G_i^\vartheta<1$ and to stop only from there on with the corresponding conditional probabilities. Consider as such a time the first one at which any player puts some mass on $\{F\geq L\}$; call it $\tilde\tau$. By waiting until $\tilde\tau$, player $i$ might become follower before if $G_j^\vartheta$ increases, and then on $\{F<L\}$. At $\tilde\tau$, at least one player is willing to stop by definition. The corresponding (symmetric) local payoff is clearly $L_{\tilde\tau}\leq F_{\tilde\tau}$ when $G_1^\vartheta$, $G_2^\vartheta$ are continuous. A jump can only occur if indeed $F_{\tilde\tau}=L_{\tilde\tau}$, which is also the maximal local payoff (with the hypothesis $M\leq \min(L,F)$, we cannot have any jump when $F_{\tilde\tau}>L_{\tilde\tau}$ as we have seen). Finally, it may happen that $G_i^\vartheta$ is exhausted on $\{F<L\}$, before ever reaching $\tilde\tau$. Then, however, we must have $G_1^\vartheta=G_2^\vartheta$. If they jump to one, the terminal payoff is at most $F<L$; if they approach one continuously, this means eventually becoming follower on $\{F<L\}$ for sure. In summary, player $i$ never receives more than $\min(L,F)$ when stopping occurs.

Proposition \ref{prop:U_min(L,F)} implies that whenever $L_\vartheta>U_{L\wedge F}(\vartheta)$, we must have  $G_1^\vartheta(\vartheta)\vee G_2^\vartheta(\vartheta)=1$ by preemption.\footnote{%
This argument is not impaired by any jump $\Delta G_j^\vartheta(\vartheta)\in(0,1)$ due to which player $i$ could not realize $L_\vartheta$. $L$ is right-continuous, so player $i$ could try to stop right after $\vartheta$. The formal argument is given in the proof of Theorem \ref{thm:maxeql}.
}
If there are any such preemption points in the future, they also restrict the feasible stopping times $\tau$ to maximize the expected value of $\min(L,F)$ in Proposition \ref{prop:U_min(L,F)}, which even further reduces the maximally attainable equilibrium payoff. By iteration, we can identify when preemption is inevitable.

\begin{theorem}\label{thm:maxeql}
Suppose $M\leq\min(L,F)$. Then there exists a payoff-symmetric subgame-perfect equilibrium with maximal payoffs within this class. For any $\vartheta\in\T$, these are
\begin{equation*}
V_1^\vartheta\bigl(G_1^\vartheta,\alpha_1^\vartheta,G_2^\vartheta,\alpha_2^\vartheta\bigr)=V_2^\vartheta\bigl(G_2^\vartheta,\alpha_2^\vartheta,G_1^\vartheta,\alpha_1^\vartheta\bigr)=\esssup_{\tau\in\T\colon\tau\in[\vartheta,\tilde\tau(\vartheta)]}E\bigl[L_\tau\wedge F_\tau\bigv\F_\vartheta\bigr]=:U_{(L\wedge F)^{\tilde\tau(\vartheta)}}(\vartheta),
\end{equation*}
where $\tilde\tau(\vartheta)$ is the latest sustainable preemption point after $\vartheta$ determined by the following algorithm:
\begin{enumerate}
\item
Set $\tau_0(\vartheta):=\inf\{t\geq\vartheta\mid L_t>U_{L\wedge F}(t)\}$ and
\begin{equation*}
(L\wedge F)^{\tau_0(\vartheta)}:=\bigl(L_{t\wedge\tau_0(\vartheta)}\wedge F_{t\wedge\tau_0(\vartheta)}\bigr)_{t\geq 0}
\end{equation*}
with Snell envelope $U_{(L\wedge F)^{\tau_0(\vartheta)}}:=(\esssup_{t\leq\tau\in\T}E[(L\wedge F)^{\tau_0(\vartheta)}_\tau\mid\F_t])_{t\geq 0}$.

\item
For every $n\in\N$, set $\tau_n(\vartheta):=\inf\{t\geq\vartheta\mid L_t>U_{(L\wedge F)^{\tau_{n-1}(\vartheta)}}(t)\}\wedge\tau_{n-1}(\vartheta)$ and
\begin{equation*}
(L\wedge F)^{\tau_n(\vartheta)}:=\bigl(L_{t\wedge\tau_n(\vartheta)}\wedge F_{t\wedge\tau_n(\vartheta)}\bigr)_{t\geq 0}
\end{equation*}
with Snell envelope $U_{(L\wedge F)^{\tau_n(\vartheta)}}=(\esssup_{t\leq\tau\in\T}E[(L\wedge F)^{\tau_n(\vartheta)}_\tau\mid\F_t])_{t\geq 0}$.

\item
Take the monotone limit $\tilde\tau(\vartheta):=\lim_{n\to\infty}\tau_n(\vartheta)$.
\end{enumerate}
\end{theorem}

\noindent
{\it Proof:} In Appendix \ref{app:maxeql}.
\medskip

The payoff-maximal equilibrium is implemented using the strategies of Theorem \ref{thm:symeql}, but setting $\alpha_i^\vartheta=0$ before $\tilde\tau(\vartheta)$. Constructing $\tilde\tau(\vartheta)$ by the algorithm is technically not difficult. The main problem is rather to verify the claimed equilibrium properties: to make sure that there is no preemption incentive when $L>F$ on $[\vartheta,\tilde\tau(\vartheta))$, that $\tilde\tau(\vartheta)$ is indeed maximal, and that there is a continuation equilibrium of preemption at $\tilde\tau(\vartheta)$. Furthermore, measurability is a major technical issue, since we want to have a time-consistent version of the strategies where we set $\alpha_i^\vartheta=0$ on $[\vartheta,\tilde\tau(\vartheta))$ for all $\vartheta\in\T$ to achieve the maximal payoff in all subgames.

In order to suppress preemption when $L_\vartheta>F_\vartheta$, it is obviously not sufficient that there exist $\tau\geq\vartheta$ such that $E[L_\tau\wedge F_\tau\mid\F_\vartheta]\geq L_\vartheta$; this relation then rather has to hold on all of $[\vartheta,\tau]\cap\{L>F\}$. For instance, the algorithm of Theorem \ref{thm:maxeql} can be applied to the model of \cite{FudenbergTirole85} by visual inspection, shown in Figure \ref{fig:FTpreem}.

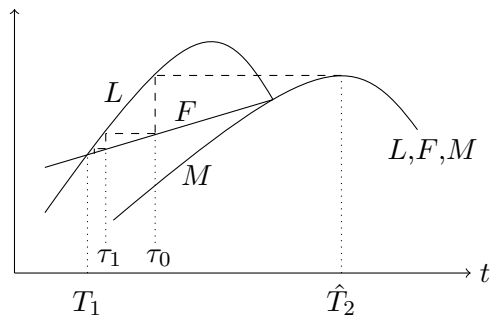
\begin{figure}[ht]
\centering
  \begin{tikzpicture}[inner sep=0pt,minimum size=0pt,label distance=3pt]
    \draw[->] (0,0) -- (6,0) node[label=right:$t$] {}; 
    \draw[->] (0,0) -- (0,3.5) node[] {}; 
    
    \draw[-] (.4,.8) .. controls (2.3,3.5) and (2.7,3.5) .. (3.4,2.3) [] {};
    \draw[-] (1.3,.7) .. controls (4,3) and (4.5,3) .. (5.3,1.9) [] {};
    \draw[-] (.4,1.4) -- (3.4,2.3) [] {};
    
    \node at (1.3,2.4) [] {$L$};
    \node at (2.25,2.15) [] {$F$};
    \node at (2.4,1.3) [] {$M$};
    \node at (5.5,1.6) [] {$L$,$F$,$M$};
    
    \draw[dotted] (.96,0) -- (.96,1.55) [] {};
    \draw[dotted] (1.2,0.4) -- (1.2,1.65) [] {};
    \draw[dotted] (1.85,0.4) -- (1.85,1.85) [] {};    
    \draw[dotted] (4.3,0) -- (4.3,2.6) [] {};
    \draw[dashed] (4.3,2.62) -- (1.85,2.62) [] {}; 
    \draw[dashed] (1.85,2.6) -- (1.85,1.85) [] {};   
    \draw[dashed] (1.85,1.85) -- (1.2,1.85) [] {}; 
    \draw[dashed] (1.2,1.85) -- (1.2,1.65) [] {}; 
    \draw[dashed] (1.2,1.65) -- (1.05,1.65) [] {};   
    \draw[dashed] (1.05,1.65) -- (1.05,1.58) [] {};                        
    \node at (.96,-.4) {$T_1$};       
    \node at (1.25,.2) {$\tau_1$}; 
    \node at (1.9,.2) {$\tau_0$};      
    \node at (4.3,-.35) {$\hat T_2$}; 
  \end{tikzpicture}
  \caption{Preemption, Fudenberg and Tirole (1985)}
  \label{fig:FTpreem}
\end{figure}

$L$ exceeds the future maximum of $\min(L,F)$ for the first time at $\tau_0$, whence there will be preemption. Taking that into account, at most the maximum of $\min(L,F)$ \emph{up to} $\tau_0$ might be achieved. However, $L$ will also exceed this reduced value, at  $\tau_1$. In the limit, $\tilde\tau(0)=T_1$ is the first inevitable preemption point. Fudenberg and Tirole also consider another, Case B, in which the peak at $\hat T_2$ is higher than first one. Then $\tilde\tau(0)=\infty$, because $L=F$ at and from their global peak onwards, and the players can coordinate on joint late adoption. In general, we may have much more complex stochastic patterns, of course, with arbitrary regions of first- and second-mover advantages, that may trigger preemption or not.

\section{Conclusion}\label{sec:conc}

In many timing games, mixed strategies play an important role as we have argued, either for equilibrium existence or for resolving any strategic conflicts (about roles with differing amenities) within the game. 
%
Having analysed the two different kinds of local strategic incentives, we have been able to prove existence of subgame-perfect equilibria for general symmetric stochastic timing games and to characterize them quite explicitly, providing symmetric equilibrium payoffs. Our approach is based on the general theory of optimal stopping and demonstrates which kinds of stopping problems need to be solved to verify equilibria; not only, but in particular for mixed strategies. 

There are possibly different equilibria for a given timing game, with varying degrees of preemption. We have considered the two extreme cases: If one initiates preemption whenever there is a first-mover advantage, then payoffs may be severely restricted. However, we have shown how to reduce preemption to a minimum and proved existence of corresponding equilibria with maximal attainable payoffs. If preemption can indeed be prevented in a certain regime with first-mover advantage (by sufficiently profitable future continuation equilibria), then there may also exist further equilibria with continuous mixing, which we have only employed for second-mover advantages. Nevertheless, any such additional mixing will be inefficient and induce lower payoffs (which one can also show directly).

A more specific strategic investment model with random first- and second-mover advantages is analysed in \cite{StegThijssen15}, where the strategies corresponding to the ones derived here have Markovian representations.

\appendix {

\section{Technical results}\label{app:tecres}

\begin{lemma}\label{lem:LdG}
If $L$ is a (measurable) process of class {\rm (D)}, then there exists a constant $K\in\R_+$ such that for any process $G$ which is a.s.\ right-continuous, nondecreasing, nonnegative and bounded by some $G_\infty\in L^\infty(P)$ and for any $[0,\infty]$-valued random variables $a,b$ it holds that
\begin{flalign*}
&\quad\text{\rm(i)} & &E\biggl[\int_{[a,b)}\abs{L_t}\,dG_t\biggr]\leq K\norm{G_\infty}_\infty<\infty\qquad\text{and} && 
\end{flalign*}
\vspace{-20pt}
\begin{flalign*}
&\quad\text{\rm(ii)} & &
\int_{[a,b)}\abs{L_t}\,dG_t=\int_0^\infty\abs{L_{\tau^G(x)}}\indi{\tau^G(x)\in[a,b)}\,dx<\infty\quad\text{a.s.}, && 
\end{flalign*}
where $\tau^G(x):=\inf\{t\geq 0\mid G_t\geq x\}$ for any $x\in\R_+$, and $\Delta G_0:=G_0$. The same holds if $\tau^G(x)$ is instead defined with ``\,$G_t>x$''.

If the bound on $G$ is only integrable, i.e., $G_\infty\in L^1(P)$, but $\{\abs{L_\tau}\indi{\tau<\infty}\mid\tau\in\T\}$ bounded in $L^\infty(P)$ by $K\in\R_+$, then {\rm(i)} holds with $KE[G_\infty]$ instead and {\rm(ii)} as stated.
\end{lemma}

\begin{proof}
The a.s.\ nondecreasing family of stopping times $(\tau^G(x))_{x\in\R_+}$ is the left-continuous inverse of $G$, which satisfies
\begin{equation*}
\tau^G(x)\leq t\ \Leftrightarrow\ G_t\geq x.
\end{equation*}
Thus, with the convention $\int_{[0,c]}\,dG=G_c$, we have a.s.\ $\int_{[0,\infty)}\indinb{A}\,dG=\int_0^\infty\indi{\tau^G(x)\in A}\,dx$ for all $A\in\{[0,c]\mid c\in\R_+\}$ and also for $A=\R_+$ by monotone convergence. The relation extends to all Borel sets $A$ from $\R_+$ by a monotone class argument.\footnote{%
See, e.g., \cite{Kallenberg02}, Theorem 1.1. 
}
As $L_\cdot(\omega)\colon\R_+\to\R$, $t\mapsto L_t(\omega)$, is Borel measurable\footnote{%
See, e.g., \cite{Kallenberg02}, Lemma 1.26 (i). 
} 
like the function $\indi{t\in[a(\omega),b(\omega))}$, we thus obtain the change-of-variable formula\footnote{%
See, e.g., \cite{Kallenberg02}, Lemma 1.22. It is necessary to restrict $dx$ to $\{\tau^G(x)<\infty\}$, which is redundant when integrating over $[a,b)$.
}
\begin{equation*}
\int_{[a,b)}\abs{L_t}\,dG_t=\int_{\{\tau^G(x)<\infty\}}\abs{L_{\tau^G(x)}}\indi{\tau^G(x)\in[a,b)}\,dx=\int_0^\infty\abs{L_{\tau^G(x)}}\indi{\tau^G(x)\in[a,b)}\,dx \qquad\text{a.s.} 
\end{equation*}
The formula also holds with $\inf\{t\geq 0\mid G_t>x\}=\tau^G(x+)$ instead of $\tau^G(x)$ because the two can only differ on a set of Lebesgue measure ($dx$) zero. By Fubini's Theorem,
\begin{align}
E\biggl[\int_0^\infty\abs{L_{\tau^G(x)}}\indi{\tau^G(x)\in[a,b)}\,dx\biggr]&\leq\int_0^\infty E\Bigl[\abs{L_{\tau^G(x)}}\indi{\tau^G(x)<\infty}\Bigr]\,dx \nonumber\\
&=\int_0^{\norm{G_\infty}_\infty}E\Bigl[\abs{L_{\tau^G(x)}}\indi{\tau^G(x)<\infty}\Bigr]\,dx. \label{ELdG}
\end{align}
As $L$ is of class {\rm (D)}, $\{\abs{L_\tau}\indi{\tau<\infty}\mid\tau\in\T\}$ is bounded in $L^1(P)$ by some $K<\infty$, whence \eqref{ELdG} is bounded by $K\norm{G_\infty}_\infty$ if the latter is finite. If $\norm{L_\tau\indi{\tau<\infty}}_\infty\leq K$ for every $\tau\in\T$ and $G$ is bounded by $G_\infty\in L^1(P)$, then \eqref{ELdG} is bounded by
\begin{equation*}
\int_0^\infty E\Bigl[K\indi{\tau^G(x)<\infty}\Bigr]\,dx=KE\biggl[\int_0^\infty\indi{\tau^G(x)<\infty}\,dx\biggr]\leq KE[G_\infty]<\infty.
\end{equation*}
In either case it follows that $\int_{[a,b)}\abs{L_t}\,dG_t<\infty$ a.s.
\end{proof}

\begin{lemma}\label{lem:SclassD}
Suppose the processes $L$, $F$, and $M$ are optional and of class {\rm (D)}, and the adapted process $G$ is right-continuous, nondecreasing and taking values in $[0,1]$ a.s. Then $S$ defined by
\begin{equation*}
S_t:=\int_{[0,t)}F_s\,dG_s+\Delta G_tM_t+\bigl(1-G_t\bigr)L_t,\quad t\in\R_+,
\end{equation*}
is optional and of class {\rm (D)} as well. Moreover, if additionally $M_\infty\in L^1(P)$ and $G_\infty\equiv 1$ are defined, and correspondingly $S_\infty$ by setting $t=\infty$, then also $S_\infty\in L^1(P)$.
\end{lemma}

\begin{proof}
The components of $S$ are obviously optional, in particular the integral being a left-continuous and $\Delta G$ the difference of a right-continuous and a left-continuous adapted process. To show that $S$ is of class {\rm (D)} under the hypothesis, and also that $S_\infty\in L^1(P)$, it is clearly enough to verify that $\{\int_{[0,\tau)}F\,dG\mid\tau\in\T\}$ is uniformly integrable. By the Theorem of de la Vall\'ee-Poussin, a family of random variables $\{X_\tau\mid\tau\in\T\}$ is uniformly integrable if and only if there exists a nondecreasing and convex function $g\colon\R_+\to\R_+$ such that $\lim_{t\to\infty}\frac{g(t)}{t}=\infty$ and $\sup_{\tau\in\T}E[g(\lvert X_\tau\rvert)]<\infty$. This holds by hypothesis for $\{F_\tau\indi{\tau<\infty}\mid\tau\in\T\}$; cf.\ fn.\ \ref{fn:classD}. Fix this $g$ associated to $F$. By a change of variable as in Lemma \ref{lem:LdG} then
\begin{align*}
&\sup_{\tau\in\T}E\biggl[g\biggl(\biggl\lvert\int_{[0,\tau)}F_s\,dG_s\biggr\rvert\biggr)\biggr]\leq\sup_{\tau\in\T}E\biggl[g\biggl(\int_0^1\abs{F_{\tau^G(x)}\indi{\tau^G(x)<\tau}}\,dx\biggr)\biggr]\\[4pt]
\leq{}&\sup_{\tau\in\T}E\biggl[\int_0^1g\Bigl(\abs{F_{\tau^G(x)}\indi{\tau^G(x)<\infty}}\Bigr)\,dx\biggr]\leq\int_0^1\sup_{\tau\in\T}E\Bigl[g\Bigl(\abs{F_{\tau}\indi{\tau<\infty}}\Bigr)\Bigr]\,dx<\infty.
\end{align*}
For the last two inequalities, we resp.\ used Jensen's inequality and Fubini's Theorem.
\end{proof}

\begin{lemma}\label{lem:tauC}
Consider two stopping times $\sigma\leq\tau$ and an event $C\in\F_\sigma$. Then
\begin{equation*}
\vartheta:=\sigma\indinb{C}+\tau\indinb{C^c}
\end{equation*}
is a stopping time. If the filtration is complete, it suffices that $\sigma\leq\tau$ a.s.
\end{lemma}

\begin{proof}
To verify that $\vartheta$ is a stopping time, we check whether $\{\vartheta\leq t\}\in\F_t$ for all $t\in\R_+$. First note that
\begin{equation*}
\{\vartheta\leq t\}=(\{\sigma\leq t\}\cap C)\cup(\{\tau\leq t\}\cap C^c).
\end{equation*}
The first intersection belongs to $\F_t$ by definition of $\F_\sigma$, and so also the complement $(\{\sigma\leq t\}\cap C)^c=\{\sigma>t\}\cup C^c\in\F_t$, implying $C^c\cap\{\sigma\leq t\}\in\F_t$. Finally, as $\{\tau\leq t\}\in\F_t$, we conclude
\begin{equation*}
C^c\cap\{\sigma\leq t\}\cap\{\tau\leq t\}=C^c\cap\{\tau\leq t\}\in\F_t.
\end{equation*}
If $\sigma\leq\tau$ only a.s., then the last equality holds up to the nullset $C^c\cap\{\sigma>t\}\cap\{\tau\leq t\}$, which is contained in $\F_t$ if the filtration is complete.
\end{proof}

\begin{lemma}\label{lem:DLcont}
The process $\tilde L^{\tau^\vartheta}:=\indi{t<\tau^\vartheta}L+\indi{t\geq\tau^\vartheta}\max(F_{\tau^\vartheta},M_{\tau^\vartheta})$ defined in Theorem \ref{thm:mixedeql} is upper-semi-continuous from the left in expectation on $[\vartheta,\infty]$, on which $D_{\tilde L}$ is hence left-continuous a.s.
\end{lemma}

\begin{proof}
Let $(\tau_n)_{n\in\N}$ be a sequence of stopping times that is a.s.\ increasing and taking values in $[\vartheta,\infty]$, and denote the limit by $\tau\in\T$. Define the measurable set
\begin{equation*}
A:=\bigcap_{n}\{\tau_n<\tau^\vartheta\}.
\end{equation*}
Then $\lim_{n\to\infty}\tilde L_{\tau_n}=\tilde L_\tau$ a.s.\ on $A^c$, implying $\lim_{n\to\infty} E[\indinb{A^c}\tilde L_{\tau_n}]=E[\indinb{A^c}\tilde L_\tau]$ as $\tilde L$ is of class {\rm (D)}. We obtain the analogue for $L\wedge F$ if we use $(\tau_n\wedge\tau^\vartheta)$ and $(\tau\wedge\tau^\vartheta)$, respectively. Combining the latter fact with upper-semi-continuity from the left in expectation given by Assumption \ref{asm:payoffs}\,\ref{LFusc} yields
\begin{align*}
& & \limsup_{n\in\N}{}&E\bigl[\indinb{A}(L\wedge F)_{\tau_n\wedge\tau^\vartheta}+\indinb{A^c}(L\wedge F)_{\tau_n\wedge\tau^\vartheta}\bigr]\leq E\bigl[\indinb{A}(L\wedge F)_{\tau\wedge\tau^\vartheta}+\indinb{A^c}(L\wedge F)_{\tau\wedge\tau^\vartheta}\bigr] & &\\
&\Rightarrow & \limsup_{n\in\N}{}&E\bigl[\indinb{A}(L\wedge F)_{\tau_n}\bigr]\leq E\bigl[\indinb{A}(L\wedge F)_{\tau}\bigr]. & &
\end{align*}
Note that $(L\wedge F)_{\tau_n}=\tilde L_{\tau_n}$ on $A$ and $(L\wedge F)_{\tau}\leq\tilde L_\tau$ on $A\subseteq\{\tau\leq\tau^\vartheta\}$, implying the first claim.

Finally, we show that if $\tilde L$ is upper-semi-continuous from the left in expectation on an interval $[\vartheta,\infty]$, then the paths of $D_{\tilde L}$ are left-continuous on that interval, a.s. Define the auxiliary process $\hat L$ by $\hat L_t:=\indi{t<\vartheta}\hat M_t+\indi{t\geq\vartheta}\tilde L_t$, where $\hat M$ is a right-continuous version of the martingale $(E[\tilde L_\vartheta\mid\F_t])_{t\geq 0}$, which is uniformly integrable and thus (left-) continuous in expectation thanks to optional sampling. $\hat L$ inherits upper-semi-continuity from the left in expectation because $E[\hat L_{\tau_n}]=E[\indi{\tau_n<\vartheta}\tilde L_\vartheta+\indi{\tau_n\geq\vartheta}\tilde L_{\tau_n}]=E[\tilde L_{\tau_n\vee\vartheta}]$ and similarly $E[\hat L_{\tau}]=E[\tilde L_{\tau\vee\vartheta}]$. As $\tilde L$ and $\hat L$ agree on $[\vartheta,\infty]$, their Snell envelopes $U_{\tilde L}$ and $U_{\hat L}$ agree at any stopping time in that interval and thus $U_{\tilde L}\indi{t\geq\vartheta}$ and $U_{\hat L}\indi{t\geq\vartheta}$ are indistinguishable by the uniqueness of optional projections. The same holds for the compensators $D_{\tilde L}$ and $D_{\hat L}$ on $[\vartheta,\infty]$ by uniqueness of the Doob-Meyer-decomposition. $D_{\hat L}$ is left-continuous a.s.\ because $\hat L$ is upper-semi-continuous from the left in expectation, see fn.\ \ref{fn:Lusc}.
\end{proof}

\begin{lemma}\label{lem:dG/1-G}
Let $A,B\colon\R_+\to\R\cup\{+\infty\}$ be two right-continuous, nondecreasing functions with $0\leq\Delta B\leq 1$. Then the differential equation
\begin{equation}\label{dG=(1-G)dA}
dG=(1-G)\,dA,\qquad G_{0-}=a\in\R
\end{equation}
has the solution
\begin{align*}
G=1-(1-a)e^{-\int\,dA^c}\textstyle\prod(1+\Delta A)^{-1}=1-(1-a)e^{-\int\,dA^c-\sum\ln(1+\Delta A)},
\end{align*}
where $A^c=A-\sum\Delta A\in[A_0,A]$ is the continuous part of $A$, and the differential equation
\begin{equation}\label{dG=(1-G-)dA}
dG=(1-G_-)\,dB,\qquad G_{0-}=b\in\R
\end{equation}
has the solution
\begin{align*}
G&=1-(1-b)e^{-\int\,dB^c}\textstyle\prod(1-\Delta B)=1-(1-b)e^{-\int\,dB^c+\sum\ln(1-\Delta B)}.
\end{align*}
Any solution to \eqref{dG=(1-G)dA} or \eqref{dG=(1-G-)dA} is monotone towards, but never crossing, one. $G$ solves both equations if and only if $\Delta A=\frac{\Delta B}{1-\Delta B}$, resp.\ $\Delta B=\frac{\Delta A}{1+\Delta A}$.  
\end{lemma}

\begin{proof}
Straightforward to check. Note that $dG_0=\frac{(1-a)\,dA_0}{1+\Delta A_0}$ for any solution of \eqref{dG=(1-G)dA}, implying both monotonicity on either side of one and that this value cannot be crossed. Indeed, $G_0\gtrless 1\Leftrightarrow\Delta G_0\gtrless 1-G_{0-}=1-a\Leftrightarrow 0\gtrless 1-a$.

Similarly, $dG_0=(1-b)\,dB_0$ for any solution of \eqref{dG=(1-G-)dA}, implying $G_0\gtrless 1\Leftrightarrow\Delta G_0\gtrless 1-G_{0-}=1-b\Leftrightarrow 0\gtrless 1-b$ \emph{and} $\Delta B_0<1$, using $\Delta B_0\in[0,1]$ for the last equivalence.
\end{proof}

\begin{lemma}\label{lem:E(YZ)to0}
Let $(Y_t,Z_t)_{t\in[0,1]}$ be a family of random variables. Assume that the family $(Y_t)$ is uniformly integrable and that $(Z_t)$ is bounded in $L^\infty(P)$, and $Z_t\to 0$ in probability as $t\to 1$. Then 
\begin{equation*}
\lim_{t\to 1}E[Y_tZ_t]=0.
\end{equation*}
\end{lemma}

\begin{proof}
$(\norm{Z_t}_\infty)$ is bounded by a constant $K$, hence $(Z_t)$ is uniformly integrable and converges to the constant zero also in $L^1(P)$ as $t\to 1$. As $(Y_t)$ is uniformly integrable, we can find for any $\varepsilon>0$ a suitable constant $K_\varepsilon\geq 0$ such that 
\begin{equation*}
E\bigl[\indi{\abs{Y_t}\geq K_\varepsilon}\abs{Y_tZ_t}\bigr]\leq\varepsilon K\qquad\text{for all }t\in[0,1].
\end{equation*}
In combination,
\begin{align*}
\limsup_{t\to 1}E\bigl[\abs{Y_tZ_t}\bigr]={}&\limsup_{t\to 1}\Bigl(E\bigl[\indi{\abs{Y_t}\geq K_\varepsilon}\abs{Y_tZ_t}\bigr]+E\bigl[\indi{\abs{Y_t}<K_\varepsilon}\abs{Y_tZ_t}\bigr]\Bigr)\\
\leq{}&\limsup_{t\to 1}E\bigl[\indi{\abs{Y_t}\geq K_\varepsilon}\abs{Y_tZ_t}\bigr]+\lim_{t\to 1}E\bigl[\indi{\abs{Y_t}<K_\varepsilon}\abs{Y_tZ_t}\bigr]\leq\varepsilon K
\end{align*}
and the claim follows.
\end{proof}

\section{Proofs}\label{app:proofs}

\subsection{Proofs for results in Sections \ref{sec:BR}--\ref{sec:eqlsym}}\label{app:miscproofs}

\begin{proof}[{\bf Proof of Lemma \ref{lem:BRpure}}]
Define the right-continuous inverse of $G_i^\vartheta$ by the stopping times $\tau_i^{G,\vartheta}(x):=\inf\{t\geq\vartheta\mid G_i^\vartheta(t)>x\}$, $x\in[0,1)$. Then Lemma \ref{lem:LdG} allows the change of variable
\begin{equation*}
\int_{[\vartheta,\infty)}S_i^\vartheta(t)\,dG_i^\vartheta(t)=\int_0^1 S_i^\vartheta\bigl(\tau_i^{G,\vartheta}(x)\bigr)\indi{\tau_i^{G,\vartheta}(x)<\infty}\,dx\quad\text{a.s.}
\end{equation*}
Moreover, $x>G_i^\vartheta(\infty-)\Rightarrow\tau_i^{G,\vartheta}(x)=\infty\Rightarrow x\geq G_i^\vartheta(\infty-)$, i.e., $\indi{x>G_i^\vartheta(\infty-)}\leq\indi{\tau_i^{G,\vartheta}(x)=\infty}\leq\indi{x\geq G_i^\vartheta(\infty-)}$ for all $x\in[0,1)$ a.s., implying
\begin{equation*}
\Delta G_i^\vartheta(\infty)S_i^\vartheta(\infty)=\biggl(\int_0^1 \indi{\tau_i^{G,\vartheta}(x)=\infty}\,dx\biggr)S_i^\vartheta(\infty)\quad\text{a.s.}
\end{equation*}
Therefore,
\begin{align}
V_i^\vartheta\bigl(G_i^\vartheta,G_j^\vartheta\bigr)&=E\biggl[\int_0^1 S_i^\vartheta\bigl(\tau_i^{G,\vartheta}(x)\bigr)\indi{\tau_i^{G,\vartheta}(x)<\infty}\,dx+\int_0^1 S_i^\vartheta(\infty)\indi{\tau_i^{G,\vartheta}(x)=\infty}\,dx\biggv\F_\vartheta\biggr]\nonumber\\
&=E\biggl[\int_0^1 S_i^\vartheta\bigl(\tau_i^{G,\vartheta}(x)\bigr)\,dx\biggv\F_\vartheta\biggr].\label{ESitauGi}
\end{align}
Let $A$ denote the event that \eqref{ESitauGi} exceeds $\esssup_{x\in[0,1)}E[S_i^\vartheta(\tau_i^{G,\vartheta}(x))\mid\F_\vartheta]=:Z_i^\vartheta$ and suppose by way of contradiction $P[A]>0$. Then $E[\indinb{A}\int_0^1 S_i^\vartheta(\tau_i^{G,\vartheta}(x))\,dx]>E[\indinb{A}Z_i^\vartheta]$ by $A\in\F_\vartheta$ and iterated expectations, so $\int_0^1E[\indinb{A}S_i^\vartheta(\tau_i^{G,\vartheta}(x))]\,dx>E[\indinb{A}Z_i^\vartheta]$ by Fubini's Theorem. There would thus exist $x\in(0,1)$ with $E[\indinb{A}Z_i^\vartheta]<E[\indinb{A}S_i^\vartheta(\tau_i^{G,\vartheta}(x))]=E[\indinb{A}E[S_i^\vartheta(\tau_i^{G,\vartheta}(x))\mid\F_\vartheta]]$, which contradicts $Z_i^\vartheta\geq E[S_i^\vartheta(\tau_i^{G,\vartheta}(x))\mid\F_\vartheta]$. So, $V_i^\vartheta(G_i^\vartheta,G_j^\vartheta)\leq Z_i^\vartheta$ a.s., implying \eqref{stopSi}. Given the latter and denoting its right-hand side by $U_{S_i}^\vartheta$, it is a.s.\ binding if and only if $E[V_i^\vartheta(G_i^\vartheta,G_j^\vartheta)]=E[U_{S_i}^\vartheta]$, i.e., by \eqref{ESitauGi} and Fubini, if and only if $\int_0^1E[S_i^\vartheta(\tau_i^{G,\vartheta}(x))]\,dx=E[U_{S_i}^\vartheta]$, resp.\ $S_i^\vartheta(\tau_i^{G,\vartheta}(x))=U_{S_i}^\vartheta$ a.s.\ for a.e.\ $x\in[0,1)$.

If $G_i^\vartheta$ is extended by any feasible $\alpha_i^\vartheta$ instead of the trivial $\alpha^\infty$ given by $\indi{t\geq\infty}$, then it is easy to check from Definitions \ref{def:payoffs_extended} and \ref{def:outcome} that
\begin{align*}
&V_i^\vartheta\bigl(G_i^\vartheta,\alpha_i^\vartheta,G_j^\vartheta,\alpha^\infty\bigr)-V_i^\vartheta\bigl(G_i^\vartheta,\alpha^\infty,G_j^\vartheta,\alpha^\infty\bigr) \\
={}&E\Bigl[\indi{\hat\tau_i^\vartheta<\infty}\Delta G_i^\vartheta\bigl(\hat\tau_i^\vartheta\bigr)\Delta G_j^\vartheta\bigl(\hat\tau_i^\vartheta\bigr)\bigl(1-\alpha_i^\vartheta\bigl(\hat\tau_i^\vartheta\bigr)\bigr)\Bigl(F_{\hat\tau_i^\vartheta}-M_{\hat\tau_i^\vartheta}\Bigr)\Bigv\F_\vartheta\Bigr]
\end{align*}
with $\hat\tau_i^\vartheta=\inf\{t\geq 0\mid\alpha_i^\vartheta(t)>0\}$, so $G_i^\vartheta(\hat\tau_i^\vartheta)=1$. Let $B=\{\indi{\hat\tau_i^\vartheta<\infty}\Delta G_j^\vartheta(\hat\tau_i^\vartheta)(F_{\hat\tau_i^\vartheta}-M_{\hat\tau_i^\vartheta})>0\}$, so $B\in\F_{\hat\tau_i^\vartheta}$, and define $G_n^\vartheta$ for any $n\in\N\setminus\{1,2\}$ by delaying $\Delta G_i^\vartheta(\hat\tau_i^\vartheta)$ to $\hat\tau_i^\vartheta+n^{-1}$ on $B$, i.e., by $G_n^\vartheta(t)=G_i^\vartheta(t)\indi{t<\hat\tau_i^\vartheta}+(G_i^\vartheta(\hat\tau_i^\vartheta-)\indinb{B}+\indinb{B^c})\indi{t-\hat\tau_i^\vartheta\in[0,n^{-1})}+\indi{t\geq\hat\tau_i^\vartheta+n^{-1}}$. Then $(G_n^\vartheta,\alpha^\infty)$ is a feasible (standard) mixed strategy for $\vartheta$. Let furthermore $C\in\F_\vartheta$. By right-continuity of $L$ and $G_j^\vartheta$, $S_i^\vartheta(t+)-S_i^\vartheta(t)=\Delta G_j^\vartheta(t)(F_t-M_t)$. As $S_i^\vartheta$ is of class {\rm (D)}, we obtain the limit in expectation
\begin{align*}
&\lim_{n\to\infty}E\Bigl[\indinb{C}\Bigl(V_i^\vartheta\bigl(G_n^\vartheta,\alpha^\infty,G_j^\vartheta,\alpha^\infty\bigr)-V_i^\vartheta\bigl(G_i^\vartheta,\alpha^\infty,G_j^\vartheta,\alpha^\infty\bigr)\Bigr)\Bigr] \\
={}&\lim_{n\to\infty}E\biggl[\indinb{C}\biggl(\int_{[0,\infty]}S_i^\vartheta(t)dG_n^\vartheta(t)-\int_{[0,\infty]}S_i^\vartheta(t)dG_i^\vartheta(t)\biggr)\biggr] \\
={}&E\Bigl[\indinb{C}\Bigl(S_i^\vartheta(\hat\tau_i^\vartheta+)-S_i^\vartheta(\hat\tau_i^\vartheta)\Bigr)\indinb{B}\Delta G_i^\vartheta(\hat\tau_i^\vartheta)\Bigr]=E\Bigl[\indinb{C}\Delta G_j^\vartheta\bigl(\hat\tau_i^\vartheta\bigr)\Bigl(F_{\hat\tau_i^\vartheta}-M_{\hat\tau_i^\vartheta}\Bigr)\indinb{B}\Delta G_i^\vartheta\bigl(\hat\tau_i^\vartheta\bigr)\Bigr] \\
\geq{}&E\Bigl[\indinb{C}\Bigl(V_i^\vartheta\bigl(G_i^\vartheta,\alpha_i^\vartheta,G_j^\vartheta,\alpha^\infty\bigr)-V_i^\vartheta\bigl(G_i^\vartheta,\alpha^\infty,G_j^\vartheta,\alpha^\infty\bigr)\Bigr)\Bigr].
\end{align*}
Together with \eqref{stopSi} for all pairs $(G_n^\vartheta,G_j^\vartheta)$, this shows that $C=\{V_i^\vartheta(G_i^\vartheta,\alpha_i^\vartheta,G_j^\vartheta,\alpha^\infty)>\esssup_{\tau\geq\vartheta}E[S_i^\vartheta(\tau)\mid\F_\vartheta]\}$ must have probability zero.
\end{proof}

\begin{proof}[{\bf Proof of Lemma \ref{lem:tauLdom}}]
By right-continuity of $L$ and $G_j^\vartheta$, $S_i^\vartheta(\tau+)-S_i^\vartheta(\tau)=\Delta G_j^\vartheta(\tau)(F_\tau-M_\tau)\geq 0$ for any $\tau\in\T$. Now consider the set $\{\tau_i<\tau_L^{*}(\vartheta)\}$, on which a.s.\
\begin{align*}
&E\Bigl[S_i^\vartheta\bigl(\tau_L^{*}(\vartheta)+\bigr)-S_i^\vartheta\bigl(\tau_i\bigr)\Bigv\F_{\tau_i}\Bigr]\\
&=E\biggl[\int_{[\tau_i,\tau_L^{*}(\vartheta)]}F_s\,dG_j^\vartheta(s)+\Bigl(1-G_j^\vartheta\bigl(\tau_L^{*}(\vartheta)\bigr)\Bigr)L_{\tau_L^{*}(\vartheta)}-\Delta G_j^\vartheta(\tau_i)M_{\tau_i}-\Bigl(1-G_j^\vartheta\bigl(\tau_i\bigr)\Bigr)L_{\tau_i}\biggv\F_{\tau_i}\biggr]\\
&\geq E\Bigl[F_{\tau_L^{*}(\vartheta)}\Bigl(G_j^\vartheta\bigl(\tau_L^{*}(\vartheta)\bigr)-G_j^\vartheta\bigl(\tau_i-\bigr)\Bigr)+\Bigl(1-G_j^\vartheta\bigl(\tau_L^{*}(\vartheta)\bigr)\Bigr)L_{\tau_L^{*}(\vartheta)}-\Bigl(1-G_j^\vartheta\bigl(\tau_i-\bigr)\Bigr)L_{\tau_i}\Bigv\F_{\tau_i}\Bigr]\\
&\geq E\Bigl[L_{\tau_L^{*}(\vartheta)}\Bigl(1-G_j^\vartheta\bigl(\tau_i-\bigr)\Bigr)-\Bigl(1-G_j^\vartheta\bigl(\tau_i-\bigr)\Bigr)L_{\tau_i}\Bigv\F_{\tau_i}\Bigr]\geq 0.
\end{align*}
The first inequality is obtained from a change of variable exploiting that $F$ is a supermartingale (demonstrated at the end of the proof), and from $L\geq M$. The second inequality is due to $F\geq L$ and the last one to the optimality of $\tau_L^{*}(\vartheta)$. Note that the latter will be strict if $P[\tau^i<\tau_L^{*}(\vartheta)\text{ and }G_j^\vartheta(\tau_i-)<1]>0$, by suboptimality of any $\vartheta\leq\tau_i<\tau_L^{*}(\vartheta)$. The second estimate in Lemma \ref{lem:tauLdom} follows from setting $\tau_i=\vartheta$ in the previous steps. The next claim is due to iterated expectations at $\tau_L^{*}(\vartheta)\geq\vartheta$ and $E[S_i^\vartheta(\tau_i+)\mid\F_{\tau_L^{*}(\vartheta)}]\leq\esssup_{\tau\geq\tau_L^{*}(\vartheta)}E[S_i^\vartheta(\tau)\mid\F_{\tau_L^{*}(\vartheta)}]$ for any stopping time $\tau_i\geq\tau_L^{*}(\vartheta)$. Indeed, let $A\in\F_{\tau_L^{*}(\vartheta)}$ be the event that the latter fails for given $\tau_i$ and consider the stopping times $\tau_i+n^{-1}$, $n\in\N$. As $S_i^\vartheta$ is of class {\rm (D)}, $E[\indinb{A}S_i^\vartheta(\tau_i+)]=\lim_{n\to\infty}E[\indinb{A}S_i^\vartheta(\tau_i+n^{-1})]\leq E[\indinb{A}\esssup_{\tau\geq\tau_L^{*}(\vartheta)}E[S_i^\vartheta(\tau)\mid\F_{\tau_L^{*}(\vartheta)}]]$, showing that $A$ has probability zero. Finally, the previous and following steps go through identically with $\tau_L^{*}(\vartheta)$ replaced by $\tau_L^{**}(\vartheta)$. 

The announced change of variable argument similar to the proof of Lemma \ref{lem:BRpure} is that
\begin{align*}
&E\Bigl[F_{\tau_L^{*}(\vartheta)}\Bigl(G_j^\vartheta\bigl(\tau_L^{*}(\vartheta)\bigr)-G_j^\vartheta\bigl(\tau_i-\bigr)\Bigr)\Bigv\F_{\tau_i}\Bigr]-E\biggl[\int_{[\tau_i,\tau_L^{*}(\vartheta)]}F_s\,dG_j^\vartheta(s)\biggv\F_{\tau_i}\biggr]\displaybreak[0]\\
&=E\biggl[\int_0^1\Bigl(F_{\tau_L^{*}(\vartheta)}-F_{\tau_j^{G,\vartheta}(x)}\Bigr)\indi{\tau_j^{G,\vartheta}(x)\in[\tau_i,\tau_L^{*}(\vartheta)]}\,dx\biggv\F_{\tau_i}\biggr]
\end{align*}
cannot exceed $\esssup_{\tau\geq\tau_i}E[(F_{\tau_L^{*}(\vartheta)}-F_{\tau})\indi{\tau\in[\tau_i,\tau_L^{*}(\vartheta)]}\mid\F_{\tau_i}]$. This is nonpositive by iterated expectations at $\tau$, with $E[(F_{\tau_L^{*}(\vartheta)}-F_{\tau})\mid\F_{\tau}]\leq 0$ on $\{\tau\leq\tau_L^{*}(\vartheta)\}$ as $F$ is a supermartingale.
\end{proof}

\begin{proof}[{\bf Proof of Proposition \ref{prop:pureeql}}]
Let $\vartheta\in\T$ and $C\in\F_{\tau_L^{*}(\vartheta)}$. Then $\tau_1^*$, $\tau_2^*$ as hypothesized are stopping times thanks to Lemma \ref{lem:tauC} by $\tau_L^{*}(\vartheta)\leq\tau_F(\vartheta)$ a.s. To verify the optimality of $\tau_1^*$, it suffices by Lemmas \ref{lem:BRpure} and \ref{lem:tauLdom} to consider stopping $S_1^\vartheta$ from $\tau_L^{*}(\vartheta)$ on. On $C$, $S^\vartheta_1(t)=L_{t\wedge\tau_F(\vartheta)}$ for all $t\geq\tau_L^{*}(\vartheta)$, such that stopping immediately at $\tau_L^{*}(\vartheta)$ is optimal by its optimality for $L$. On $C^c$, $S^\vartheta_1(t)=F_{\tau_L^{*}(\vartheta)}\geq M_{\tau_L^{*}(\vartheta)}$ for all $t>\tau_L^{*}(\vartheta)$, with equality on $\{\tau_F(\vartheta)=\tau_L^{*}(\vartheta)\}$ by hypothesis. Hence, $\tau_F(\vartheta)$ is optimal on $C^c$. The same argument applies to $\tau_2^*$, swapping $C$ and $C^c$.

We can use $\tau_F(\vartheta):=\inf\{t\geq\vartheta\mid F_t=M_t\}$ because then $\tau_F(\vartheta)\geq\tau_L^{*}(\vartheta)$ a.s. Indeed, as $F$ is a supermartingale dominating $L$, it also dominates the Snell envelope $U_L$. Therefore, at $\tau_F(\vartheta)$, $F=M$ (by right-continuity) implies that $F\geq U_L\geq L\geq M$ must bind throughout. Hence, $\tau_F(\vartheta)\geq\inf\{t\geq\vartheta\mid U_L(t)=L_t\}=\tau_L^{*}(\vartheta)$.
\end{proof}

\begin{proof}[{\bf Proof of Theorem \ref{thm:mixedeql}}]
Let $\vartheta,\tau^\vartheta\in\T$ be as hypothesized. As they remain fixed, we may suppress them as superscripts (except from $\tau^\vartheta$) in this proof to ease readability. $\tilde L$ is right-continuous a.s.\ and of class {\rm (D)}, so it has a Snell envelope $U_{\tilde L}$ with an integrable and predictable compensator $D_{\tilde L}$. On $[\vartheta,\infty]$, $\tilde L$ is upper-semi-continuous from the left in expectation, and then $D_{\tilde L}$ is a.s.\ continuous; see Lemma \ref{lem:DLcont} and cf.\ fn.s \ref{fn:Lusc}, \ref{fn:Lrc}.

Let now $G_i$, $G_j$ resp.\ be given by \eqref{Geql}, \eqref{Gjeql}. $G_i$ represents a standard mixed strategy: it is adapted as $\tau_i\in\T$, and it is a.s.\ right-continuous, nondecreasing, has $G_i(t)=0$ for $t\in[0,\vartheta)$, and $G_i(\infty)=1$. It can only jump at $\tau_i\leq\tau^\vartheta$: $\indi{F>L}(F-L)^{-1}$ can be understood as a Radon-Nikodym derivative, such that the integral defines a measure on $\R_+$ which is absolutely continuous w.r.t.\ the (finite) atomless measure $dD_{\tilde L}$.\footnote{%
The new measure is also \nbd{\sigma}finite as $\{F>L\}=\bigcup_{n\in\N}\{F-L\geq\frac{1}{n}\}$.
}
$G_j$ is even continuous on $[\vartheta,\tau^\vartheta)$.

To prove that the standard mixed strategy represented by $G_j$ is a best reply to $G_i$ (even among extended mixed strategies), it suffices by Lemma \ref{lem:BRpure} and its proof to show that for some random variable $\bar Z_j$ and every stopping time $\tau\geq\vartheta$ we have
\begin{equation}\label{optEStau}
E\Bigl[S_j(\tau)\Bigv\F_\vartheta\Bigr]\leq E\Bigl[\bar Z_j\Bigv\F_\vartheta\Bigr],
\end{equation}
with equality if $dG_j(\tau)>0$. Indeed, then $\esssup_{\tau\geq\vartheta}E[S_j(\tau)\mid\F_\vartheta]\leq E[\bar Z_j\mid\F_\vartheta]$, and by the equalities, it then holds in the argument following \eqref{ESitauGi} that $E[S_j(\tau_j^{G,\vartheta}(x))\mid\F_\vartheta]=E[\bar Z_j\mid\F_\vartheta]=:Z_j$ for all $x\in[0,1)$, so it also applies with all inequalities reversed (except $P[A]>0$), hence showing that $V_j(G_j,G_i)=Z_j= E[\bar Z_j\mid\F_\vartheta]$. 

To establish \eqref{optEStau}, consider an arbitrary stopping time $\tau\geq\vartheta$. Suppose first also $\tau\leq\tau_i$. On $[\vartheta,\tau_i)$, $G_i$ satisfies $dG_i=(1-G_i)(F-L)^{-1}\indi{F>L}\,dD_{\tilde L}$ by construction, and $\indi{F\leq L}\,dD_{\tilde L}=0$ by definition of $\tau_i$. Hence,
\begin{equation}\label{F-LdG}
\int_{[\vartheta,\tau)}(F-L)\,dG_i=\int_{[\vartheta,\tau)}(1-G_i)\indi{F>L}\,dD_{\tilde L}=\int_{[\vartheta,\tau)}(1-G_i)\,dD_{\tilde L}\qquad\text{a.s.},
\end{equation}
where $\int L\,dG_i$ is well defined by Lemma \ref{lem:LdG} and $L$ of class {\rm (D)}. 
Applying integration by parts to the right-hand side (adjusting for $[\vartheta,\tau)$ closed on the left, open on the right, and recalling that $D_{\tilde L}$ is continuous) yields 
\begin{align}\label{1-GdD}
\int_{[\vartheta,\tau)}(1-G_i)\,dD_{\tilde L}={}&\int_{[\vartheta,\tau)}D_{\tilde L}\,dG_i+\bigl(1-G_i(\tau-)\bigr)D_{\tilde L}(\tau)-\bigl(1-G_i(\vartheta-)\bigr)D_{\tilde L}(\vartheta).
\end{align}
Using \eqref{F-LdG}, \eqref{1-GdD}, and $G_i(\vartheta-)=0$, now
\begin{align*}
S_j(\tau)&=\int_{[\vartheta,\tau)}F\,dG_i+\Delta G_i(\tau)M_\tau+\bigl(1-G_i(\tau)\bigr)L_\tau \nonumber\\
&=\int_{[\vartheta,\tau)}\bigl(L+D_{\tilde L}\bigr)\,dG_i+\Delta G_i(\tau)\bigl(M_\tau+D_{\tilde L}(\tau)\bigr)+\bigl(1-G_i(\tau)\bigr)\bigl(L_\tau+D_{\tilde L}(\tau)\bigr)-D_{\tilde L}(\vartheta). 
\end{align*}
Next, as the martingale component $M_{\tilde L}$ of the Snell envelope is uniformly integrable, $\int M_{\tilde L}\,dG_i$ is well defined by Lemma \ref{lem:LdG}, and we can make the change of variable to show that
\begin{align*}
&E\biggl[\int_{[\vartheta,\tau)}M_{\tilde L}\,dG_i\biggv\F_\vartheta\biggr]=E\Bigl[M_{\tilde L}(\tau)\bigl(G_i(\tau-)-G_i(\vartheta-)\bigr)\Bigv\F_\vartheta\Bigr] \displaybreak[0]\\
={}&-E\Bigl[M_{\tilde L}(\tau)\bigl(1-G_i(\tau-)\bigr)\Bigv\F_\vartheta\Bigr]+M_{\tilde L}(\vartheta)\bigl(1-G_i(\vartheta-)\bigr)
\end{align*}
as $M_{\tilde L}$ is a martingale. Indeed, similarly as in the proof of Lemma \ref{lem:BRpure},
\begin{align*}
&E\biggl[\int_{[\vartheta,\tau)}M_{\tilde L}(t)\,dG_i(t)\biggv\F_\vartheta\biggr]-E\Bigl[M_{\tilde L}(\tau)\bigl(G_i(\tau-)-G_i(\vartheta-)\bigr)\Bigv\F_\vartheta\Bigr] \displaybreak[0]\\
={}&E\biggl[\int_0^1\Bigl(M_{\tilde L}\bigl(\tau^G_i(x)\bigr)-M_{\tilde L}(\tau)\Bigr)\indi{\tau^G_i(x)\in[\vartheta,\tau)}\,dx\biggv\F_\vartheta\biggr]
\end{align*}
cannot exceed $\esssup_{\tau'\geq\vartheta}E[(M_{\tilde L}(\tau')-M_{\tilde L}(\tau))\indi{\tau'\in[\vartheta,\tau)}\mid\F_\vartheta]$, which is zero by iterated expectations at $\tau'$, with $E[(M_{\tilde L}(\tau')-M_{\tilde L}(\tau))\mid\F_{\tau'}]=0$ on $\{\tau'<\tau\}$ due to the martingale property. Switching sign then yields the previously claimed identity.

Combining the last two results with $G_i(\vartheta-)=0$,
\begin{align}\label{ES_j}
E\bigl[S_j(\tau)\bigv\F_\vartheta\bigr]=E\biggl[\int_{[\vartheta,\tau)}&\bigl(L+D_{\tilde L}-M_{\tilde L}\bigr)\,dG_i+\Delta G_i(\tau)\bigl(M_\tau+D_{\tilde L}(\tau)-M_{\tilde L}(\tau)\bigr) \\
+{}&\bigl(1-G_i(\tau)\bigr)\bigl(L_\tau+D_{\tilde L}(\tau)-M_{\tilde L}(\tau)\bigr)\biggv\F_\vartheta\biggr]-D_{\tilde L}(\vartheta)+M_{\tilde L}(\vartheta). \nonumber
\end{align}

On $[\vartheta,\tau)$, $L+D_{\tilde L}-M_{\tilde L}=\tilde L-U_{\tilde L}$ by $\tau\leq\tau_i\leq\tau^\vartheta$, and thus the integral vanishes. Indeed, $dG_i$ is absolutely continuous w.r.t.\ $dD_{\tilde L}$ on $[\vartheta,\tau_i)$, and $\int(U_{\tilde L}-\tilde L)\,dD_{\tilde L}=0$; cf.\ \eqref{(U_L-L)dD_L=0}. 
Furthermore, $\Delta G_i(\tau)=0$ on $\{\tau<\tau_i\}$ and $U_{\tilde L}(\tau)=\tilde L_\tau$ on $\{\tau=\tau_i\}$. Specifically, on $\{\tau_i<\tau^\vartheta\}$, their definition implies $\indi{F=L}dD_{\tilde L}>0$ at $\tau_i$, so $U_{\tilde L}(\tau_i)=\tilde L_{\tau_i}=L_{\tau_i}=F_{\tau_i}$; cf.\ \eqref{U_L=L}. On $\{\tau_i=\tau^\vartheta\}$, $U_{\tilde L}(\tau_i)=\tilde L_{\tau_i}$ as $\tilde L_t$ is constant for $t\in[\tau^\vartheta,\infty]$. Finally, the last term in the expectation vanishes when $\tau\geq\tau^\vartheta$ and otherwise has again $L_\tau=\tilde L_\tau$.
\eqref{ES_j} thus becomes
\begin{align}
E\bigl[S_j(\tau)\bigv\F_\vartheta\bigr]&=E\Bigl[\indi{\tau=\tau_i}\Delta G_i(\tau_i)\bigl(M_{\tau_i}-\tilde L_{\tau_i}\bigr)+\bigl(1-G_i(\tau)\bigr)\bigl(\tilde L_\tau-U_{\tilde L}(\tau)\bigr)\Bigv\F_\vartheta\Bigr]+U_{\tilde L}(\vartheta) \nonumber\\
&\leq E\Bigl[\indi{\tau=\tau_i}\Delta G_i(\tau_i)\bigl(M_{\tau_i}-\tilde L_{\tau_i}\bigr)\Bigv\F_\vartheta\Bigr]+U_{\tilde L}(\vartheta) \label{ESjdG>0}
\end{align}
by $U_{\tilde L}\geq\tilde L$, with equality for $\tau=\tau_i$ by $U_{\tilde L}(\tau_i)=\tilde L_{\tau_i}$ (resp.\ for $dG_j(\tau)>0$, as then necessarily $\tau=\inf\{t\geq\tau\mid D_{\tilde L}(t)>D_{\tilde L}(\tau)\}$ on $\{\tau<\tau_i\}$; cf.\ \eqref{U_L=L} again). We now remove the restriction $\tau\leq\tau_i$ by an estimate. $S_j(t)$ is constant for $t\in(\tau_i,\infty]$, with $S_j(t)-S_j(\tau_i)=\Delta G_i(\tau_i)(F_{\tau_i}-M_{\tau_i})$. Hence, $S_j(\tau)\leq S_j(\tau_i)+\Delta G_i(\tau_i)(\max(F_{\tau_i},M_{\tau_i})-M_{\tau_i})$ on $\{\tau>\tau_i\}$. This estimate also holds on $\{\tau=\tau_i\}$, so for any stopping time $\tau\geq\vartheta$ we have, together with \eqref{ESjdG>0},
\begin{align}
&E\bigl[S_j(\tau)\bigv\F_\vartheta\bigr]=E\Bigl[S_j(\tau\wedge\tau_i)+\indi{\tau>\tau_i}\Delta G_i(\tau_i)\bigl(F_{\tau_i}-M_{\tau_i}\bigr)\Bigv\F_\vartheta\Bigr] \nonumber\\
&\leq E\Bigl[S_j(\tau\wedge\tau_i)+\indi{\tau\geq\tau_i}\Delta G_i(\tau_i)\bigl(\max(F_{\tau_i},M_{\tau_i})-M_{\tau_i}\bigr)\Bigv\F_\vartheta\Bigr] \nonumber\\
&\leq E\Bigl[\indi{\tau\wedge\tau_i=\tau_i}\Delta G_i(\tau_i)\bigl(M_{\tau_i}-\tilde L_{\tau_i}\bigr)+\indi{\tau\geq\tau_i}\Delta G_i(\tau_i)\bigl(\max(F_{\tau_i},M_{\tau_i})-M_{\tau_i}\bigr)\Bigv\F_\vartheta\Bigr]+U_{\tilde L}(\vartheta) \nonumber\\
&=E\Bigl[\indi{\tau\geq\tau_i}\Delta G_i(\tau_i)\bigl(\max(F_{\tau_i},M_{\tau_i})-\tilde L_{\tau_i}\bigr)\Bigv\F_\vartheta\Bigr]+U_{\tilde L}(\vartheta). \label{ESj_est}
\end{align}
Recall that $\tilde L_{\tau^\vartheta}=\max(F_{\tau^\vartheta},M_{\tau^\vartheta})$, and as just observed, $U_{\tilde L}(\tau_i)=\tilde L_{\tau_i}=L_{\tau_i}=F_{\tau_i}$ on $\{\tau_i<\tau^\vartheta\}$. Therefore, we can summarize \eqref{ESj_est} as
\begin{align}
E\bigl[S_j(\tau)\bigv\F_\vartheta\bigr]&\leq E\Bigl[\indi{\tau\geq\tau_i}\indi{\tau_i<\tau^\vartheta}\Delta G_i(\tau_i)\bigl(\max(F_{\tau_i},M_{\tau_i})-F_{\tau_i}\bigr)\Bigv\F_\vartheta\Bigr]+U_{\tilde L}(\vartheta) \nonumber\displaybreak[0]\\
&\leq E\Bigl[\indi{\tau_i<\tau^\vartheta}\Delta G_i(\tau_i)\bigl(\max(F_{\tau_i},M_{\tau_i})-F_{\tau_i}\bigr)\Bigv\F_\vartheta\Bigr]+U_{\tilde L}(\vartheta). \label{ESj_estuni}
\end{align}
This is a $\tau$-independent bound. To construct a stopping time $\tau_a\geq\vartheta$ that attains it, let $C=\{\Delta G_i(\tau_i)(F_{\tau_i}-M_{\tau_i})>0\}$, so $C\in\F_{\tau_i}$, and we may define $\tau_a=\tau_i$ on $C^c$ and $\tau_a=\infty$ on $C$; cf.\ Lemma \ref{lem:tauC}. Note that $\tau_i<\tau_a$ on $C$ by the convention $F_\infty=M_\infty$. Then
\begin{align*}
S_j(\tau_a)=S_j(\tau_i)+\indinb{C}\Delta G_i(\tau_i)\bigl(F_{\tau_i}-M_{\tau_i}\bigr)=S_j(\tau_i)+\Delta G_i(\tau_i)\bigl(\max(F_{\tau_i},M_{\tau_i})-M_{\tau_i}\bigr)
\end{align*}
Therefore, our estimate on $S_j(\tau)$ is binding for $\tau=\tau_a$, and thus also the first inequality in \eqref{ESj_est}. The second is binding because it represents \eqref{ESjdG>0} for $\tau=\tau_a\wedge\tau_i=\tau_i$. Hence the first inequality in \eqref{ESj_estuni} binds, and finally the second by $\tau_a\geq\tau_i$. This means that $G_j$ is a best reply to $G_i$ if and only if $V_j(G_j,G_i)=E[\int S_j\,dG_j\mid\F_\vartheta]=E[S_j(\tau_a)\mid\F_\vartheta]$. 

To reduce this to the claimed conditions, we can use \eqref{ESjdG>0}--\eqref{ESj_estuni} as follows. For all $\tau\geq\vartheta$: 
\begin{align*}
E\bigl[&S_j(\tau_a)\bigv\F_\vartheta\bigr]=E\bigl[S_j(\tau)+S_j(\tau_a)-S_j(\tau\wedge\tau_i)-\indi{\tau>\tau_i}\bigl(S_j(\tau)-S_j(\tau_i)\bigr)\bigv\F_\vartheta\bigr] \displaybreak[0]\\
={}E\Bigl[&S_j(\tau)+\indi{\tau_i<\tau^\vartheta}\Delta G_i(\tau_i)\bigl(\max(F_{\tau_i},M_{\tau_i})-F_{\tau_i}\bigr) -\indi{\tau\wedge\tau_i=\tau_i}\Delta G_i(\tau_i)\bigl(M_{\tau_i}-\tilde L_{\tau_i}\bigr) \\
&-\bigl(1-G_i(\tau\wedge\tau_i)\bigr)\bigl(\tilde L_{\tau\wedge\tau_i}-U_{\tilde L}(\tau\wedge\tau_i)\bigr) -\indi{\tau>\tau_i}\Delta G_i(\tau_i)\bigl(F_{\tau_i}-M_{\tau_i}\bigr)\Bigv\F_\vartheta\Bigr].
\end{align*}
The left-hand side is independent of $\tau$, so by the argument following \eqref{ESitauGi} resp.\ \eqref{optEStau} we can integrate inside the expectation on the right-hand side by $dG_j(t)$, replacing $\tau$ by $t$, to obtain
\begin{align*}
E\biggl[&\int_{[0,\infty]}S_j(t)\,dG_j(t)+\int_{[0,\infty]}\indi{\tau_i<\tau^\vartheta}\Delta G_i(\tau_i)\bigl(\max(F_{\tau_i},M_{\tau_i})-F_{\tau_i}\bigr)\,dG_j(t) \\
-{}&\int_{[0,\infty]}\indi{t\wedge\tau_i=\tau_i}\Delta G_i(\tau_i)\bigl(M_{\tau_i}-\tilde L_{\tau_i}\bigr)\,dG_j(t)-\int_{[0,\infty]}\bigl(1-G_i(t\wedge\tau_i)\bigr)\bigl(\tilde L_{t\wedge\tau_i}-U_{\tilde L}(t\wedge\tau_i)\bigr)\,dG_j(t) \\
-{}&\int_{[0,\infty]}\indi{t>\tau_i}\Delta G_i(\tau_i)\bigl(F_{\tau_i}-M_{\tau_i}\bigr)\,dG_j(t)\biggv\F_\vartheta\biggr].
\end{align*}
The fourth integral vanishes again as $dG_j$ is absolutely continuous w.r.t.\ $dD_{\tilde L}$ on $[0,\tau_i)$ and $U_{\tilde L}(\tau_i)=\tilde L_{\tau_i}$ (resp.\ as \eqref{ESjdG>0} binds whenever $dG_j(\tau)>0$). Therefore,
\begin{align*}
E\bigl[&S_j(\tau_a)\bigv\F_\vartheta\bigr]-E\biggl[\int_{[0,\infty]}S_j(t)\,dG_j(t)\biggv\F_\vartheta\biggr] \\
=E\Bigl[&\indi{\tau_i<\tau^\vartheta}\Delta G_i(\tau_i)\bigl(\max(F_{\tau_i},M_{\tau_i})-F_{\tau_i}\bigr)-\bigl(1-G_j(\tau_i-)\bigr)\Delta G_i(\tau_i)\bigl(M_{\tau_i}-\tilde L_{\tau_i}\bigr) \\
&-\bigl(1-G_j(\tau_i)\bigr)\Delta G_i(\tau_i)\bigl(F_{\tau_i}-M_{\tau_i}\bigr)\Bigv\F_\vartheta\Bigr] \\
=E\Bigl[&\indi{\tau_i<\tau^\vartheta}\Delta G_i(\tau_i)\bigl(\max(F_{\tau_i},M_{\tau_i})-F_{\tau_i}\bigr)+\Delta G_j(\tau_i)\Delta G_i(\tau^\vartheta)\bigl(\max(F_{\tau^\vartheta},M_{\tau^\vartheta})-M_{\tau^\vartheta}\bigr)\Bigv\F_\vartheta\Bigr],
\end{align*}
where the last step follows from $G_j(\tau^\vartheta)=1$, the definition of $\tilde L_{\tau^\vartheta}$, and that on $\{\tau_i<\tau^\vartheta\}$, $\Delta G_j(\tau_i)=0$ and furthermore $\Delta G_i(\tau_i)>0$ only if $\tilde L_{\tau_i}=L_{\tau_i}=F_{\tau_i}$. The last expectation vanishes if and only if both nonnegative terms inside do, i.e., if and only if a.s.\ $\Delta G_i(\tau_i)(M_{\tau_i}-F_{\tau_i})\leq 0$ on $\{\tau_i<\tau^\vartheta\}$ and $\Delta G_i(\tau^\vartheta)(M_{\tau^\vartheta}-F_{\tau^\vartheta})\geq 0$, noting that $\Delta G_i(\tau^\vartheta)>0$ implies $\tau_i=\tau^\vartheta$ and $\Delta G_j(\tau^\vartheta)=\Delta G_i(\tau^\vartheta)$. In this case, the value of the bound established in \eqref{ESj_estuni}, and thus of $V_j(G_j,G_i)=E[\int S_j\,dG_j\mid\F_\vartheta]=E[S_j(\tau_a)\mid\F_\vartheta]$, becomes simply $U_{\tilde L}(\vartheta)$.

The previous arguments are also used to show when $G_i$ is a best reply to $G_j$, but in two steps. $G_j$ also satisfies $dG_j=(1-G_j)(F-L)^{-1}\indi{F>L}\,dD_{\tilde L}$ on $[\vartheta,\tau_i)$, but possibly $G_j(\tau_i)<1$. However, \eqref{ESjdG>0} only used $U_{\tilde L}=\tilde L$ at $\tau_i$, so we obtain it for switched roles and still $\tau\leq\tau_i$. This restriction now has to be removed stepwise, as $S_i(t)$ need not be constant for $t\in(\tau_i,\infty]$ on $\{\tau_i<\tau^\vartheta\}$. Therefore, analogously to $C$ and $\tau_a$, let $D=\{\Delta G_j(\tau_i)(F_{\tau_i}-M_{\tau_i})>0\}\in\F_{\tau_i}$ and $\tau_b\in\T$ satisfy $\tau_b=\tau_i$ on $D^c$ and $\tau_b=\infty$ else, so $\tau_b>\tau_i$ only on $\{\tau_i=\tau^\vartheta\}$ as otherwise $\Delta G_j(\tau_i)=0$. Assuming $\tau\leq\tau_b$ a.s., then also $\tau>\tau_i$ only on $\{\tau_i=\tau^\vartheta\}$, on which $S_i(t)$ becomes constant, so the estimate used for $S_j(\tau)$ also applies for switched roles, leading to \eqref{ESj_estuni} for $j$ and every $\tau\leq\tau_b$. Now $\tau_b$ attains the bound, which trivially becomes $U_{\tilde L}(\vartheta)$ by $\indi{\tau_i<\tau^\vartheta}\Delta G_j(\tau_i)=0$. As $\tau_b\geq\tau_i$, $\int_{[0,\infty]}(\cdot)\,dG_i=\int_{[0,\tau_b]}(\cdot)\,dG_i$, and thus $V_i(G_i,G_j)=E[S_i(\tau_b)\mid\F_\vartheta]$ by switching roles if and only if $E[\Delta G_i(\tau^\vartheta)\Delta G_j(\tau^\vartheta)(\max(F_{\tau^\vartheta},M_{\tau^\vartheta})-M_{\tau^\vartheta})\mid\F_\vartheta]$ vanishes, i.e., if and only if $\Delta G_i(\tau^\vartheta)(M_{\tau^\vartheta}-F_{\tau^\vartheta})\geq 0$ a.s. Then also $V_i(G_i,G_j)=U_{\tilde L}(\vartheta)$. 

To show that $E[S_i(\tau_b)\mid\F_\vartheta]\geq E[S_i(\tau)\mid\F_\vartheta]$ for \emph{any} $\tau\geq\vartheta$, we now start with  $\tau\leq\tau^\vartheta$ and then obtain an inequality instead of the second inequality in \eqref{F-LdG}, as $\indi{F>L}(1-G_j)\,dD_{\tilde L}\leq (1-G_j)\,dD_{\tilde L}$ need not bind on $\{\tau>\tau_i\}$. However, carrying on this inequality, we can apply all subsequent steps for switched roles and $\tau^\vartheta$ in place of $\tau_i$, to arrive at the analogue of \eqref{ESj_estuni}, showing that in fact $E[S_i(\tau)\mid\F_\vartheta]\leq U_{\tilde L}(\vartheta)$ for every $\tau\geq\vartheta$.

The arguments used for $G_j$ against $G_i$ show that $G_i$ is a best reply to itself if and only if
\begin{align*}
&E\Bigl[\indi{\tau_i<\tau^\vartheta}\Delta G_i(\tau_i)\bigl(\max(F_{\tau_i},M_{\tau_i})-F_{\tau_i}\bigr)+\bigl(\Delta G_i(\tau_i)\bigr)^2\bigl(\tilde L_{\tau_i}-M_{\tau_i}\bigr)\Bigv\F_\vartheta\Bigr]=0 \\
\Leftrightarrow\ &E\Bigl[\indi{\tau_i<\tau^\vartheta}\Delta G_i(\tau_i)\bigl((1-\Delta G_i(\tau_i))(M_{\tau_i}-F_{\tau_i})^++\bigl(\Delta G_i(\tau_i)\bigr)^2(F_{\tau_i}-M_{\tau_i})^+\Bigv\F_\vartheta\Bigr]=0,
\end{align*}
using that $\tilde L_{\tau_i}=\max(F_{\tau^\vartheta},M_{\tau^\vartheta})$ on $\{\tau_i=\tau^\vartheta\}$ and $\tilde L_{\tau_i}=L_{\tau^\vartheta}=F_{\tau^\vartheta}$ on $\{\tau_i<\tau^\vartheta\}$ as argued before. This expectation vanishes if and only if $\Delta G_i(\tau_i)(M_{\tau_i}-F_{\tau_i})\geq 0$, with equality on $\{\tau_i<\tau^\vartheta\}$ whenever $\Delta G_i(\tau_i)<1$. Then we can finally represent $V_i(G_i,G_i)=E[S_j(\tau_a)\mid\F_\vartheta]$ as
\begin{align*}
E\Bigl[\indi{\tau_i<\tau^\vartheta}\indi{\Delta G_i(\tau_i)>0}\bigl(M_{\tau_i}-F_{\tau_i}\bigr)\Bigv\F_\vartheta\Bigr]+U_{\tilde L}(\vartheta),
\end{align*}
where $U_{\tilde L}(\vartheta)=E[\tilde L_{\tau_{\tilde L}^{**}(\vartheta)}\mid\F_\vartheta]$ with $\tau_{\tilde L}^{**}(\vartheta)=\inf\{t\geq\vartheta\mid D_{\tilde L}(t)>D_{\tilde L}(\vartheta)\}$; cf.\ \eqref{U_L=L}. Moreover, $\tau_i<\tau^\vartheta$ and $\Delta G_i(\tau_i)=1$ if and only if $\tau_i=\tau_{\tilde L}^{**}(\vartheta)<\tau^\vartheta$ and $L_{\tau_i}=F_{\tau_i}=\tilde L_{\tau_i}$, so we can rewrite
\begin{equation*}
V_i(G_i,G_i)=E\Bigl[\tilde L_{\tau_{\tilde L}^{**}(\vartheta)}+\indinb{\{\tau_{\tilde L}^{**}(\vartheta)<\tau^\vartheta\}\cap\{L_{\tau_{\tilde L}^{**}(\vartheta)}=F_{\tau_{\tilde L}^{**}(\vartheta)}\}}\bigl(M_{\tau_{\tilde L}^{**}(\vartheta)}-\tilde L_{\tau_{\tilde L}^{**}(\vartheta)}\bigr)\Bigv\F_\vartheta\Bigr]. \qedhere
\end{equation*}
\end{proof}

\begin{remark}\label{rem:eqlLusc}
Theorem \ref{thm:mixedeql} remains true if $L$ is only upper-semi-continuous from the right (and the left), but $L\equiv M$. Then $D_{\tilde L}^{\tau^\vartheta}$ will be left-continuous (see fn.\ \ref{fn:Lusc}) and there exists a feasible mixed strategy $G_i^\vartheta$ given by 
\begin{equation*}
G_i^\vartheta(t):=1-\exp\biggl\{-\int_\vartheta^t\frac{\indi{F_s>L_s}\,d(D_{\tilde L}^{\tau^\vartheta})^c(s)}{F_s-L_s}-\sum_{[\vartheta,t]}\ln\biggl(\frac{\indi{F_s>L_s}\Delta D_{\tilde L}^{\tau^\vartheta}(s)}{F_s-L_s}+1\biggr)\biggr\}
\end{equation*}
for $t\in[\vartheta,\tau_i^\vartheta)$, where $(D_{\tilde L}^{\tau^\vartheta})^c$ is the continuous part of $D_{\tilde L}^{\tau^\vartheta}$ and $\Delta D_{\tilde L}^{\tau^\vartheta}(s)=D_{\tilde L}^{\tau^\vartheta}(s+)-D_{\tilde L}^{\tau^\vartheta}(s)$, and which then satisfies
\begin{equation*}
dG_i^\vartheta(t)=\bigl(1-G_i^\vartheta(t)\bigr)\frac{\indi{F_t>L_t}\,dD_{\tilde L}^{\tau^\vartheta}(t+)}{F_t-L_t}.
\end{equation*}
Then the only modifications in the proof are that we put $dD_{\tilde L}(\,\cdot\ +)$ on the right-hand side of \eqref{F-LdG} and the left-hand side (only!) of \eqref{1-GdD}. We do not have right-continuity of $U_{\tilde L}-\tilde L$, but $\int_{[\vartheta,\tau)}(U_{\tilde L}-\tilde L)\,dG_i=0$ still holds in \eqref{ES_j}: $dG_i(\cdot)$ is absolutely continuous w.r.t.\ $dD_{\tilde L}(\cdot\,+)$ on $[\vartheta,\tau_i)$, for which we can apply a change of variable as in Lemma \ref{lem:LdG} with $\tau^{D_{\tilde L}}(x):=\inf\{t\geq 0\mid D_{\tilde L}(t+)>x\}$. 
Then $\tau^{D_{\tilde L}}(x)=t\Leftrightarrow\forall s>t\colon D_{\tilde L}(t-)\leq x<D_{\tilde L}(s)$, implying $U_{\tilde L}=\tilde L$ a.s.\ at $\tau^{D_{\tilde L}}(x)$; cf.\ \eqref{tauopt}. Hence, $E[\int_{[0,\infty)}\indi{U_{\tilde L}>\tilde L}\,dD_{\tilde L}(+)]=\int_0^\infty E[\indi{U_{\tilde L}(\tau^{D_{\tilde L}}(x))>\tilde L(\tau^{D_{\tilde L}}(x))}]\,dx=0$. Finally, when $\tau<\tau^\vartheta$ and $\Delta G_i(\tau)>0$ in \eqref{ES_j}, then simply $U_{\tilde L}(\tau)=\tilde L_{\tau}=L_{\tau}\equiv M_{\tau}$ now, so it is still enough to consider $\Delta G_i(\tau_i)$ in all of the following, where also still $\Delta G_i(\tau_i)>0$ on $\{\tau_i<\tau^\vartheta\}$ only if $U_{\tilde L}(\tau_i)=\tilde L_{\tau_i}=L_{\tau_i}=F_{\tau_i}$.
\end{remark}

\begin{proof}[{\bf Proof of Theorem \ref{thm:SPE}}]
We only need to establish time-consistency. If the hypothesis holds, then $\{(\vartheta\vee\vartheta')\leq(\tau^\vartheta\wedge\tau^{\vartheta'})\}$ differs from $\{(\vartheta\vee\vartheta')\leq\tau^\vartheta=\tau^{\vartheta'}\}:=A\in\F_{(\vartheta\vee\vartheta')}$ at most by a nullset. Only this event is relevant for time-consistency, because $(\vartheta\vee\vartheta')>(\tau^\vartheta\wedge\tau^{\vartheta'})$ a.s. on $A^c$ and there is no restriction when $G_i^\vartheta\vee G^{\vartheta'}_i=1$. With $\tilde L^{\tau^\vartheta}=\tilde L^{\tau^{\vartheta'}}$ a.s.\ on $A$, also
\begin{equation*}
\esssup_{\tau'\in\T\colon\tau'\geq\tau}E\bigl[\tilde L^{\tau^\vartheta}(\tau')\bigv\F_\tau\bigr]=\esssup_{\tau'\in\T\colon\tau'\geq\tau}E\bigl[\tilde L^{\tau^{\vartheta'}}(\tau')\bigv\F_\tau\bigr]\quad\text{a.s.}
\end{equation*}
on $\{\tau\geq(\vartheta\vee{\vartheta'})\}\cap A$ for any $\tau\in\T$, implying $U_{\tilde L}^{\tau^\vartheta}\indi{t\geq(\vartheta\vee{\vartheta'})}=U_{\tilde L}^{\tau^{\vartheta'}}\indi{t\geq(\vartheta\vee{\vartheta'})}$ a.s.\ on $A$ (i.e., the latter two processes are indistinguishable) by the uniqueness of optional projections. Correspondingly, $D_{\tilde L}^{\tau^\vartheta}=D_{\tilde L}^{\tau^{\vartheta'}}$ on $[\vartheta\vee{\vartheta'},\infty]$ a.s.\ by the uniqueness of the Doob-Meyer decomposition. 

Time-consistency in the equivalent form $(1-G_i^\vartheta(t))=(1-G_i^\vartheta({\vartheta'}-))(1-G^{\vartheta'}_i(t))$ reads for the given $G_i^\vartheta$ as
\begin{align*}
&\indi{t<\tau_i^\vartheta}\exp\biggl(-\int_\vartheta^t\frac{\indi{F>L}\,dD^{\tau^\vartheta}_{\tilde L}}{F-L}\biggr) \displaybreak[0]\\
&=\indi{{\vartheta'}\leq\tau_i^\vartheta}\exp\biggl(-\int_\vartheta^{\vartheta'}\frac{\indi{F>L}\,dD^{\tau^\vartheta}_{\tilde L}}{F-L}\biggr)\indi{t<\tau_i^{\vartheta'}}\exp\biggl(-\int_{\vartheta'}^t\frac{\indi{F>L}\,dD^{\tau^{\vartheta'}}_{\tilde L}}{F-L}\biggr),
\end{align*}
which on $A$ and for $t\geq(\vartheta\vee{\vartheta'})$ reduces to the true statement
\begin{align*}
&\indi{t<\tau_i^\vartheta}\exp\biggl(-\int_\vartheta^t\frac{\indi{F>L}\,dD^{\tau^\vartheta}_{\tilde L}}{F-L}\biggr) \\
&=\indi{t<\tau_i^\vartheta}\exp\biggl(-\int_\vartheta^{\vartheta'}\frac{\indi{F>L}\,dD^{\tau^\vartheta}_{\tilde L}}{F-L}\biggr)\exp\biggl(-\int_{\vartheta'}^t\frac{\indi{F>L}\,dD^{\tau^\vartheta}_{\tilde L}}{F-L}\biggr)
\end{align*}
thanks to what we have shown before. The argument for $j$ is analogous.
\end{proof}

\begin{proof}[{\bf Proof of Theorem \ref{thm:symeql}}]
Fix $i,j\in\{1,2\}$, $i\neq j$, and let the extended mixed strategies be as described. For any $\vartheta\in\T$ then $\tau^\vartheta\leq\inf\{t\geq\vartheta\mid  L_t>F_t\}=\inf\{t\geq\vartheta\mid\alpha_i^\vartheta(t)>0\}$, so $\alpha_i^\vartheta(t)>0\Rightarrow G_i^\vartheta(t)=1$ a.s.\ and the same for $j$. The other feasibility conditions for $G_i^\vartheta$, $G_j^\vartheta$ follow from Theorem \ref{thm:mixedeql}, and those for $\alpha_i^\vartheta$, $\alpha_j^\vartheta$ are shown in \cite{RiedelSteg17} (in the proof of Proposition 3.1, not using their assumption that $\vartheta=\inf\{t\geq\vartheta\mid L_t>F_t\}$ and $F\geq M$). Time-consistency of the families $(G_i^\vartheta;\vartheta\in\T)$, $(G_j^\vartheta;\vartheta\in\T)$ follows from Theorem \ref{thm:SPE}, and both $(\alpha_i^\vartheta;\vartheta\in\T)$, $(\alpha_j^\vartheta;\vartheta\in\T)$ are time-consistent because $\alpha_i^\vartheta$ from Proposition \ref{prop:eqlL>F} does not depend on $\vartheta$ (except for setting $\alpha_i^\vartheta(t)=0$ for $t\in[0,\vartheta)$). Now fix arbitrary $\vartheta\in\T$. 

As $\tau^\vartheta\leq\inf\{t\geq\vartheta\mid\alpha_i^\vartheta(t)>0\}$, in fact $\alpha_i^\vartheta\equiv\alpha_i^{\tau^\vartheta}$ and so for $j$, and then it is easy to check from Definitions \ref{def:payoffs_extended} and \ref{def:outcome} that by time-consistency of $G_i^\vartheta$, $G_j^\vartheta$ with $G_i^{\tau^\vartheta}$, $G_j^{\tau^\vartheta}$, resp.,
\begin{align}\label{ViTC}
V_i^\vartheta\bigl(G_i^\vartheta,\alpha_i^\vartheta,G_j^\vartheta,\alpha_j^\vartheta\bigr)={}&E\biggl[\int_{[0,\tau^\vartheta)}\bigl(1-G_j^\vartheta\bigr)L\,dG_i^\vartheta+\int_{[0,\tau^\vartheta)}\bigl(1-G_i^\vartheta\bigr)F\,dG_j^\vartheta+\sum_{[0,\tau^\vartheta)}M\Delta G_i^\vartheta\Delta G_j^\vartheta\nonumber\\
+{}&\bigl(1-G_i^\vartheta(\tau^\vartheta-)\bigr)\bigl(1-G_j^\vartheta(\tau^\vartheta-)\bigr)V_i^{\tau^\vartheta}\bigl(G_i^{\tau^\vartheta},\alpha_i^{\tau^\vartheta},G_j^{\tau^\vartheta},\alpha_j^{\tau^\vartheta}\bigr)\biggv\F_\vartheta\biggr].
\end{align}
Here, $V_i^{\tau^\vartheta}(G_i^{\tau^\vartheta},\alpha_i^{\tau^\vartheta},G_j^{\tau^\vartheta},\alpha_j^{\tau^\vartheta})=\max(F_{\tau^\vartheta},M_{\tau^\vartheta})$ and $(G_i^{\tau^\vartheta},\alpha_i^{\tau^\vartheta})$ is at $\tau^\vartheta$ a best reply to $(G_j^{\tau^\vartheta},\alpha_j^{\tau^\vartheta})$. Indeed, using the stopping time $\hat\tau^\vartheta=\inf\{t\geq\vartheta\mid L_t>F_t\}\geq\tau^\vartheta$, this follows on $\{\tau^\vartheta=\hat\tau^\vartheta\}\in\F_{\tau^\vartheta}$ from Proposition \ref{prop:eqlL>F} applied at $\hat\tau^\vartheta$, whereas it is easily verified on $\{\tau^\vartheta<\hat\tau^\vartheta\}$ by Definitions \ref{def:payoffs_extended}, \ref{def:outcome} for $G_i^{\tau^\vartheta}(\tau^\vartheta)=G_j^{\tau^\vartheta}(\tau^\vartheta)=1$, as then $\tau^\vartheta<\inf\{t\geq\tau^\vartheta\mid\alpha_i^{\tau^\vartheta}(t)+\alpha_j^{\tau^\vartheta}(t)>0\}$ and $M_{\tau^\vartheta}\geq F_{\tau^\vartheta}$ (due to right-continuity).

Suppose player $i$ deviates to any feasible $(G_a^\vartheta,\alpha_a^\vartheta)$. First assume $\alpha_a^\vartheta(t)\equiv 0$ for $t<\tau^\vartheta$. Then the expected payoff can be written like \eqref{ViTC}, with $(G_a^{\tau^\vartheta},\alpha_a^{\tau^\vartheta})$ constructed to be time-consistent, by $G^{\tau^\vartheta}_a(t)=\indi{t\geq\tau^\vartheta}(\indi{G^{\vartheta}_a(\tau^\vartheta-)<1}(G_a^\vartheta(t)-G^{\vartheta}_a(\tau^\vartheta-))/(1-G^{\vartheta}_a(\tau^\vartheta-))+\indi{G^{\vartheta}_a(\tau^\vartheta-)=1})$ and $\alpha_a^{\tau^\vartheta}\equiv\alpha_a^\vartheta$. As this is feasible at $\tau^\vartheta$, $V_i^{\tau^\vartheta}(G_a^{\tau^\vartheta},\alpha_a^{\tau^\vartheta},G_j^{\tau^\vartheta},\alpha_j^{\tau^\vartheta})\leq V_i^{\tau^\vartheta}(G_i^{\tau^\vartheta},\alpha_i^{\tau^\vartheta},G_j^{\tau^\vartheta},\alpha_j^{\tau^\vartheta})$, so replacing the former by the latter in the analogue of \eqref{ViTC} yields at least $V_i^\vartheta(G_a^\vartheta,\alpha_a^\vartheta,G_j^\vartheta,\alpha_j^\vartheta)$. 

Now consider the hypothesis $M_{\tau^\vartheta}=V_i^{\tau^\vartheta}(G_i^{\tau^\vartheta},\alpha_i^{\tau^\vartheta},G_j^{\tau^\vartheta},\alpha_j^{\tau^\vartheta})$ ($=\max(F_{\tau^\vartheta},M_{\tau^\vartheta})$ for the true $M_{\tau^\vartheta}$ as observed before). Then \eqref{ViTC} becomes also player $i$'s expected payoff from standard mixed strategies represented by $G_i^\vartheta$, $G_j^\vartheta$, as $G_i^\vartheta(\tau^\vartheta)=G_j^\vartheta(\tau^\vartheta)=1$. Moreover, the bound for $V_i^\vartheta(G_a^\vartheta,\alpha_a^\vartheta,G_j^\vartheta,\alpha_j^\vartheta)$ constructed analogously to \eqref{ViTC} becomes the expected payoff from the standard mixed strategy given by $G_a^\vartheta$ adjusted to $G_a^\vartheta(\tau^\vartheta)=1$ (which is still feasible) when played against $G_j^\vartheta$. As the hypothesis induces $M_{\tau^\vartheta}\geq F_{\tau^\vartheta}$, and as $F_{\tau_i^\vartheta}\geq M_{\tau_i^\vartheta}$ on $\{\tau_i^\vartheta<\tau^\vartheta\}$ by construction, Theorem \ref{thm:mixedeql} now implies $V_i^\vartheta(G_a^\vartheta,\alpha_a^\vartheta,G_j^\vartheta,\alpha_j^\vartheta)\leq V_i^\vartheta(G_i^\vartheta,\alpha_i^\vartheta,G_j^\vartheta,\alpha_j^\vartheta)$.

Furthermore, Lemma \ref{lem:BRpure} implies that under the hypothesis on $M_{\tau^\vartheta}$, the expected payoff from \emph{any} extended mixed strategy at $\vartheta$ played against the standard mixed strategy represented by $G_j^\vartheta$ can never exceed $V_i^\vartheta(G_i^\vartheta,\alpha_i^\vartheta,G_j^\vartheta,\alpha_j^\vartheta)$. Therefore, to deal with arbitrary feasible $\alpha_a^\vartheta$, it still suffices to bound the expected payoff by one from a feasible strategy played against $G_j^\vartheta$. 

If not necessarily $\tau^\vartheta\leq\hat\tau_a^\vartheta:=\inf\{t\geq\vartheta\mid\alpha_a^\vartheta(t)>0\}$, then the previous argument can be adjusted as follows. As $\alpha_j^\vartheta(t)\equiv 0$ for $t<\tau^\vartheta$, the analogue of \eqref{ViTC} for $(G_a^\vartheta,\alpha_a^\vartheta)$ has integrals and sum restricted to $[0,\tau^\vartheta\wedge\hat\tau_a^\vartheta)$, the expectation additionally includes $\indi{\hat\tau_a^\vartheta<\tau^\vartheta}(\lambda_{L,i}^\vartheta L_{\hat\tau_a^\vartheta}+\lambda_{L,j}^\vartheta F_{\hat\tau_a^\vartheta}+\lambda_{M}^\vartheta M_{\hat\tau_a^\vartheta})$, and the continuation payoff applies on $\{\tau^\vartheta\leq\hat\tau_a^\vartheta\}$. The bound on $V_i^\vartheta(G_a^\vartheta,\alpha_a^\vartheta,G_j^\vartheta,\alpha_j^\vartheta)$ is produced by the same estimate on the continuation value. If $(G_a^\vartheta,\alpha_a^\vartheta)$ is again adjusted to satisfy $G_a^\vartheta(\tau^\vartheta)=1$ and $\alpha_a^\vartheta(t)\equiv 0$ for $t\in[\tau^\vartheta,\infty)$ (which is feasible), then its expected payoff when played against $G_j^\vartheta$ under the same hypothesis on $M_{\tau^\vartheta}$ becomes indeed the present bound. To see this now, the only difference is to note that nothing changes on $\{\hat\tau_a^\vartheta<\tau^\vartheta\}$, because then the outcome probabilities from $(\alpha_a^\vartheta,\alpha_j^\vartheta)$ at $\hat\tau_a^\vartheta$ do not depend on values for $t\geq\tau^\vartheta$, and $M_{\tau^\vartheta}$ has no effect.

Verifying optimality for player $j$ is completely analogous, as the specific roles of $i$ and $j$ have not been used. Similarly, if $F_\tau=M_\tau$ or $\tau=\inf\{t>\tau\mid L_t>F_t\}$ whenever $L_\tau=F_\tau$, then the condition $\Delta G_i^\vartheta(M-F)=0$ at $\tau_i^\vartheta$ a.s.\ on $\{\tau_i^\vartheta<\tau^\vartheta\}$ from Theorem \ref{thm:mixedeql} holds, and the previous arguments prove that $(G_i^\vartheta,\alpha_i^\vartheta)$ is a best reply to itself.
\end{proof}

\subsection{Proof of Proposition \ref{prop:U_min(L,F)}}\label{app:U_min(L,F)}

We begin with three Lemmas \ref{lem:dGi=dGj}--\ref{lem:DG>0}, which establish some important necessary conditions for strategies and payoffs in payoff-symmetric equilibria. They are crucial for the subsequent proof of Proposition \ref{prop:U_min(L,F)}.

\begin{lemma}\label{lem:dGi=dGj}
In any payoff-symmetric subgame-perfect equilibrium and for any $\vartheta\in\T$,
\begin{equation*}
\int_{[0,t)}\indi{L_s\neq F_s}\bigl(1-G_2^\vartheta(s-)\bigr)\,dG_1^\vartheta(s)=\int_{[0,t)}\indi{L_s\neq F_s}\bigl(1-G_1^\vartheta(s-)\bigr)\,dG_2^\vartheta(s)
\end{equation*}
for all $t\in\R_+$ a.s.\ and hence
\begin{equation}\label{(L-F)(Gi-Gj)=0}
\indi{L_t\neq F_t}\frac{dG_1^\vartheta(t)}{\bigl(1-G_1^\vartheta(t-)\bigr)}=\indi{L_t\neq F_t}\frac{dG_2^\vartheta(t)}{\bigl(1-G_2^\vartheta(t-)\bigr)},
\end{equation}
which is to be interpreted as ``${}=0$'' if $(1-G_1^\vartheta(t-))(1-G_2^\vartheta(t-))=0$. Both representations also hold with $(1-G_1^\vartheta(\cdot))$, $(1-G_2^\vartheta(\cdot))$ in place of the left-hand limits.
\end{lemma}

\begin{proof}
First consider arbitrary $\tau\in\T$ with $\vartheta\leq\tau\leq\hat\tau^\vartheta=\inf\{t\geq\vartheta\mid\alpha_1^\vartheta(t)+\alpha_2^\vartheta(t)>0\}$ a.s. Analogously to \eqref{ViTC}, time-consistency and iterated expectations imply for $i,j\in\{1,2\}$, $i\neq j$, that
\begin{align}\label{(1-G_j-)}
V_i^\vartheta\bigl(G_i^\vartheta,\alpha_i^\vartheta,G_j^\vartheta,\alpha_j^\vartheta\bigr)&=E\biggl[\int_{[0,\tau)}\bigl(1-G_j^\vartheta\bigr)L\,dG_i^\vartheta+\int_{[0,\tau)}\bigl(1-G_i^\vartheta\bigr)F\,dG_j^\vartheta+\sum_{[0,\tau)}M\Delta G_i^\vartheta\Delta G_j^\vartheta \nonumber\\
&\qquad+\bigl(1-G_i^\vartheta(\tau-)\bigr)\bigl(1-G_j^\vartheta(\tau-)\bigr)V^\tau_i\bigl(G^\tau_i,\alpha^\tau_i,G^\tau_j,\alpha^\tau_j\bigr)\biggv\F_\vartheta\biggr] \displaybreak[0]\\
&=E\biggl[\int_{[0,\tau)}\bigl(1-G_j^\vartheta(s-)\bigr)L_s\,dG_i^\vartheta(s)+\int_{[0,\tau)}\bigl(1-G_i^\vartheta(s-)\bigr)F_s\,dG_j^\vartheta(s) \nonumber\\
&\qquad+\sum_{[0,\tau)}(M_s-L_s-F_s)\Delta G_i^\vartheta(s)\Delta G_j^\vartheta(s) \nonumber\\
&\qquad+\bigl(1-G_i^\vartheta(\tau-)\bigr)\bigl(1-G_j^\vartheta(\tau-)\bigr)V^\tau_i\bigl(G^\tau_i,\alpha^\tau_i,G^\tau_j,\alpha^\tau_j\bigr)\biggv\F_\vartheta\biggr]. \nonumber
\end{align}
In a payoff-symmetric equilibrium, $V_1^\vartheta(\cdot)-V_2^\vartheta(\cdot)=V_1^\tau(\cdot)-V_2^\tau(\cdot)=0$. Hence,
\begin{equation*}
E\biggl[\int_{[0,\tau)}\bigl(1-G_2^\vartheta(s-)\bigr)(L_s-F_s)\,dG_1^\vartheta-\int_{[0,\tau)}\bigl(1-G_1^\vartheta(s-)\bigr)(L_s-F_s)\,dG_2^\vartheta\biggv\F_\vartheta\biggr]=0
\end{equation*}
for any $\tau\in[\vartheta,\hat\tau^\vartheta]$. The integrals represent two signed optional random measures,\footnote{%
$\int_{[0,t]}(1-G_2^\vartheta(s-))(L_s-F_s)\,dG_1^\vartheta(s)$ and $\int_{[0,t]}(1-G_1^\vartheta(s-))(L_s-F_s)\,dG_2^\vartheta(s)$ are adapted, right-continuous, and of finite variation. Their minimal decomposition is using $(L-F)^{+}$ and $(L-F)^{-}$, respectively.
} 
which agree on all optional sets in $[0,\hat\tau^\vartheta)$ (trivially on $[0,\vartheta)$ by $G_i^\vartheta(\vartheta-)=0$; the optional \nbd{\sigma}field is generated by the stochastic intervals $[0,\tau)$, $\tau\in\T$). $L$ and $F$ are optional processes, so we may cancel $(L-F)\neq 0$ to observe two left-continuous (thus optional) processes $\int_{[0,t)}\indi{L_s\neq F_s}(1-G_2^\vartheta(s-))\,dG_1^\vartheta(s)$ and $\int_{[0,t)}\indi{L_s\neq F_s}(1-G_1^\vartheta(s-))\,dG_2^\vartheta(s)$ that agree in expectation at any stopping time $\tau\leq\hat\tau^\vartheta$. They are thus indistinguishable up to $\hat\tau^\vartheta$ by the uniqueness of optional projections. 

As $G_1^\vartheta(s)\vee G_2^\vartheta(s)=1$ for all $s\geq\hat\tau^\vartheta$, the measures $\int_{[0,t)}\indi{L_s\neq F_s}(1-G_j^\vartheta(s-))\,dG_i^\vartheta(s)$ can only depend on $i$, $j$ on $[\hat\tau^\vartheta]$, specifically when $L_{\hat\tau^\vartheta}\neq F_{\hat\tau^\vartheta}$ and $G_i^\vartheta(\hat\tau^\vartheta)<G_j^\vartheta(\hat\tau^\vartheta)=1$ and hence $\hat\tau^\vartheta=\inf\{t\geq\vartheta\mid\alpha_j^\vartheta(t)>0\}<\inf\{t\geq\vartheta\mid\alpha_i^\vartheta(t)>0\}$. By time-consistency, the same properties hold for the strategies in the subgame starting at $\hat\tau^\vartheta$, and it is then easy to check from Definition \ref{def:outcome} that
\begin{align}
V_i^{\hat\tau^\vartheta}(\cdot)&=G_i^{\hat\tau^\vartheta}(\hat\tau^\vartheta)\bigl(1-\alpha_j^{\hat\tau^\vartheta}({\hat\tau^\vartheta})\bigr)L_{\hat\tau^\vartheta}+\bigl(1-G_i^{\hat\tau^\vartheta}({\hat\tau^\vartheta})\bigr)F_{\hat\tau^\vartheta}+G_i^{\hat\tau^\vartheta}({\hat\tau^\vartheta})\alpha_j^{\hat\tau^\vartheta}({\hat\tau^\vartheta})M_{\hat\tau^\vartheta} \nonumber\\
&=G_i^{\hat\tau^\vartheta}(\hat\tau^\vartheta)\bigl(\bigl(1-\alpha_j^{\hat\tau^\vartheta}({\hat\tau^\vartheta})\bigr)L_{\hat\tau^\vartheta}-F_{\hat\tau^\vartheta}+\alpha_j^{\hat\tau^\vartheta}({\hat\tau^\vartheta})M_{\hat\tau^\vartheta}\bigr)+F_{\hat\tau^\vartheta}, \label{Vialphaj}\\
V_j^{\hat\tau^\vartheta}(\cdot)&=G_i^{\hat\tau^\vartheta}(\hat\tau^\vartheta)\bigl(1-\alpha_j^{\hat\tau^\vartheta}({\hat\tau^\vartheta})\bigr)F_{\hat\tau^\vartheta}+\bigl(1-G_i^{\hat\tau^\vartheta}({\hat\tau^\vartheta})\bigr)L_{\hat\tau^\vartheta}+G_i^{\hat\tau^\vartheta}({\hat\tau^\vartheta})\alpha_j^{\hat\tau^\vartheta}({\hat\tau^\vartheta})M_{\hat\tau^\vartheta} \nonumber\\
&=\alpha_j^{\hat\tau^\vartheta}({\hat\tau^\vartheta})G_i^{\hat\tau^\vartheta}({\hat\tau^\vartheta})\bigl(M_{\hat\tau^\vartheta}-F_{\hat\tau^\vartheta}\bigr)+G_i^{\hat\tau^\vartheta}(\hat\tau^\vartheta)F_{\hat\tau^\vartheta}+\bigl(1-G_i^{\hat\tau^\vartheta}({\hat\tau^\vartheta})\bigr)L_{\hat\tau^\vartheta}. \label{Vjalphaj}
\end{align}
Given $L_{\hat\tau^\vartheta}\neq F_{\hat\tau^\vartheta}$, payoffs are symmetric if and only if $G_i^{\hat\tau^\vartheta}(\hat\tau^\vartheta)(1-\alpha_j^{\hat\tau^\vartheta}({\hat\tau^\vartheta}))=1-G_i^{\hat\tau^\vartheta}({\hat\tau^\vartheta})$. Then $G_i^{\hat\tau^\vartheta}(\hat\tau^\vartheta)>0$, and given also $G_i^{\hat\tau^\vartheta}(\hat\tau^\vartheta)<1$, by linearity in \eqref{Vialphaj} thus $(1-\alpha_j^{\hat\tau^\vartheta}({\hat\tau^\vartheta}))L_{\hat\tau^\vartheta}-F_{\hat\tau^\vartheta}+\alpha_j^{\hat\tau^\vartheta}({\hat\tau^\vartheta})M_{\hat\tau^\vartheta}=0$. With $L_{\hat\tau^\vartheta}\neq F_{\hat\tau^\vartheta}$, this implies $\alpha_j^{\hat\tau^\vartheta}({\hat\tau^\vartheta})>0$, and that $M_{\hat\tau^\vartheta}=F_{\hat\tau^\vartheta}$ only if $\alpha_j^{\hat\tau^\vartheta}({\hat\tau^\vartheta})=1$, but this is excluded by payoff symmetry and $G_i^{\hat\tau^\vartheta}(\hat\tau^\vartheta)<1$. However, $\alpha_j^{\hat\tau^\vartheta}({\hat\tau^\vartheta})\in(0,1)$ and $G_i^{\hat\tau^\vartheta}(\hat\tau^\vartheta)>0$ requires $M_{\hat\tau^\vartheta}=F_{\hat\tau^\vartheta}$ by linearity in \eqref{Vjalphaj}, which shows that we cannot have $L_{\hat\tau^\vartheta}\neq F_{\hat\tau^\vartheta}$ and $G_i^{\hat\tau^\vartheta}(\hat\tau^\vartheta)<1$ with positive probability. This implies that our previous measures agree on all of $[0,\infty)$ a.s.

The representation \eqref{(L-F)(Gi-Gj)=0} is obtained by integrating $[(1-G_1^\vartheta(s-))(1-G_2^\vartheta(s-))]^{-1}$ on $\{G_1^\vartheta(s-)\vee G_2^\vartheta(s-)<1\}$ w.r.t.\ each measure. Finally, $\indi{L\neq F}(1-G_2^\vartheta)\,dG_1^\vartheta=\indi{L\neq F}(1-G_1^\vartheta)\,dG_2^\vartheta$ is obtained analogously, without even rewriting \eqref{(1-G_j-)}.
\end{proof}

\begin{lemma}\label{lem:Gi=Gj}
In any payoff-symmetric subgame-perfect equilibrium and for any $\vartheta\in\T$, it holds that $G_1^\vartheta(t)=G_2^\vartheta(t)$ for all $t\in[\vartheta,\,\inf\{t\geq\vartheta\mid\int_{[0,t]}\indi{L=F}\,(dG_1^\vartheta+dG_2^\vartheta)>0\})$ a.s.
\end{lemma}

\begin{proof}
By Lemma \ref{lem:dGi=dGj}, $\indi{L\neq F}(1-G_j^\vartheta)\,dG_i^\vartheta=\indi{L\neq F}(1-G_i^\vartheta)\,dG_j^\vartheta$, $i,j\in\{1,2\}$, $i\neq j$. On the interval $[\vartheta,\,\inf\{t\geq\vartheta\mid\int_{[0,t]}\indi{L=F}(dG_1^\vartheta+dG_2^\vartheta)>0\})$, we can ignore $\indi{L\neq F}$, so if it is nonempty, then $G_i^\vartheta(\vartheta)=G_j^\vartheta(\vartheta)$ and~-- as long as $G_j^\vartheta(t)<1$~-- $dG_i^\vartheta(t)=\phi_t\,dG_j^\vartheta(t)$ for
\begin{equation}\label{phi}
\phi_t=\frac{1-G_i^\vartheta(t)}{1-G_j^\vartheta(t)},\qquad\phi_0=1.
\end{equation}
By rearranging \eqref{phi}, for all $t\in[\vartheta,\,\inf\{t\geq\vartheta\mid\int_{[0,t]}\indi{L=F}(dG_1^\vartheta+dG_2^\vartheta)>0\text{ or }G_j^\vartheta(t)=1\})$,
\begin{alignat*}{2}
&& G_i^\vartheta(t)=1-\bigl(1-G_j^\vartheta(t)\bigr)\phi_t&=G_i^\vartheta(0)+\int_{(0,t]}\phi_s\,dG_j^\vartheta(s)\\
&\Leftrightarrow\qquad & \phi_t-1&=G_j^\vartheta(t)\phi_t-G_j^\vartheta(0)-\int_{(0,t]}\phi_s\,dG_j^\vartheta(s)\\
&\Leftrightarrow & \phi_t-\phi_0&=\int_{(0,t]}G_j^\vartheta(s-)\,d\phi_s.
\end{alignat*}
The last line is obtained by integration by parts, as $(\phi_t)$ is of finite variation and right-continuous for $t<\inf\{t\geq\vartheta\mid G_j^\vartheta(t)=1\}$. It implies that $\phi_t$ must indeed be constant before $G_j^\vartheta$ attains one. If $G_j^\vartheta$ jumps to one on the given interval when $L\neq F$, then $(1-G_i^\vartheta)\Delta G_j^\vartheta=(1-G_j^\vartheta)\Delta G_i^\vartheta=0$, i.e., $G_i^\vartheta$ must attain one, too, which completes the proof.
\end{proof}

\begin{lemma}\label{lem:DG>0}
In any payoff-symmetric subgame-perfect equilibrium and for any $\vartheta\in\T$, it holds that, on $\{\Delta G_1^\vartheta(\vartheta)\vee\Delta G_2^\vartheta(\vartheta)>0\}$,
\begin{equation*}
F_\vartheta\leq\max(L_\vartheta,M_\vartheta)\quad\text{and}\quad V_1^\vartheta(\cdot)=V_2^\vartheta(\cdot)\leq\Delta G_j^\vartheta(\vartheta)\max(F_\vartheta,M_\vartheta)+\bigl(1-\Delta G_j^\vartheta(\vartheta)\bigr)L_\vartheta
\end{equation*}
for $j=1,2$, a.s.
\end{lemma}

\begin{proof}
Let $i,j\in\{1,2\}$, $i\neq j$. Note that $\Delta G_i^\vartheta(\vartheta)=G_i^\vartheta(\vartheta)$ and recall $G_i^\vartheta(\hat\tau^\vartheta)=1$ for $\hat\tau^\vartheta=\inf\{t\geq\vartheta\mid\alpha_i^\vartheta(t)+\alpha_j^\vartheta(t)>0\}$. The equilibrium payoff must be at least that from the feasible strategy $(G_a^\vartheta,\alpha_i^\vartheta)$ satisfying $G_a^\vartheta(t)\equiv 1$ for $t\geq\vartheta$, which is $V_i^\vartheta(G_a^\vartheta,\alpha_i^\vartheta,G_j^\vartheta,\alpha_j^\vartheta)=G_j^\vartheta(\vartheta)M_\vartheta+(1-G_j^\vartheta(\vartheta))L_\vartheta$ on $\{\vartheta<\hat\tau^\vartheta\}$. Another feasible strategy is $(G_b^\vartheta,\alpha_i^\vartheta)$ given by $G_b^\vartheta(t)=\indi{G_i^\vartheta(\vartheta)<1}(G_i^\vartheta(t)-G_i^\vartheta(\vartheta))(1-G_i^\vartheta(\vartheta))^{-1}+\indi{G_i^\vartheta(\vartheta)=1}$ for $t\geq\vartheta$, which still satisfies time-consistency with $(G_i^{\hat\tau^\vartheta},\alpha_i^{\hat\tau^\vartheta})$. By \eqref{(1-G_j-)} then
\begin{equation}\label{V_iG_i>0}
V_i^\vartheta\bigl(G_i^\vartheta,\alpha_i^\vartheta,G_j^\vartheta,\alpha_j^\vartheta\bigr)=G_i^\vartheta(\vartheta)V_i^\vartheta\bigl(G_a^\vartheta,\alpha_i^\vartheta,G_j^\vartheta,\alpha_j^\vartheta\bigr)+\bigl(1-G_i^\vartheta(\vartheta)\bigr)V_i^\vartheta\bigl(G_b^\vartheta,\alpha_i^\vartheta,G_j^\vartheta,\alpha_j^\vartheta\bigr),
\end{equation}
so $G_i^\vartheta(\vartheta)>0$ implies $V_i^\vartheta(G_i^\vartheta,\alpha_i^\vartheta,G_j^\vartheta,\alpha_j^\vartheta)=V_i^\vartheta(G_a^\vartheta,\alpha_i^\vartheta,G_j^\vartheta,\alpha_j^\vartheta)\geq V_i^\vartheta(G_b^\vartheta,\alpha_i^\vartheta,G_j^\vartheta,\alpha_j^\vartheta)$, i.e., in particular equilibrium payoffs equal to $G_j^\vartheta(\vartheta)M_\vartheta+(1-G_j^\vartheta(\vartheta))L_\vartheta$ on $\{\vartheta<\hat\tau^\vartheta\}$. Consider furthermore the feasible strategies $(G_n^\vartheta,\alpha_i^\vartheta)$ for $n\in\N\setminus\{1,2\}$ satisfying $G_n^\vartheta(t)=\indi{t\geq(\vartheta+n^{-1})\wedge\hat\tau^\vartheta}$. Then $\lim_{n\to\infty}E[\indinb{A}V_i^\vartheta(G_n^\vartheta,\alpha_i^\vartheta,G_j^\vartheta,\alpha_j^\vartheta)]=E[\indinb{A}(G_j^\vartheta(\vartheta)F_\vartheta+(1-G_j^\vartheta(\vartheta))L_\vartheta)]$ for any $A\in\F_\vartheta$ with $A\subseteq\{\vartheta<\hat\tau^\vartheta\}$. Indeed, we have pointwise convergence inside the expectation from Definition \ref{def:payoffs_extended} by right-continuity of $L$ and $G_j^\vartheta$, and we may pass to the limit in expectation as $L$ and $\int F\,dG_j^\vartheta$ are of class {\rm (D)} (the latter by Lemma \ref{lem:SclassD}). As the limit cannot exceed $E[\indinb{A}V_i^\vartheta(G_i^\vartheta,\alpha_i^\vartheta,G_j^\vartheta,\alpha_j^\vartheta)]$, choosing $A=\{\vartheta<\hat\tau^\vartheta\}\cap\{G_i^\vartheta(\vartheta)G_j^\vartheta(\vartheta)(F_\vartheta-M_\vartheta)>0\}$ shows that this must have probability zero. Therefore, $M_\vartheta\geq F_\vartheta$ on $\{\vartheta<\hat\tau^\vartheta\}\cap\{G_1^\vartheta(\vartheta)\wedge G_2^\vartheta(\vartheta)>0\}$, so that both claims hold. On $\{\vartheta<\hat\tau^\vartheta\}\cap\{G_1^\vartheta(\vartheta)\vee G_2^\vartheta(\vartheta)>G_1^\vartheta(\vartheta)\wedge G_2^\vartheta(\vartheta)=0\}$, $L_\vartheta=F_\vartheta$ due to Lemma \ref{lem:dGi=dGj}, which implies the first claimed inequality, and also the second as now $V_i^\vartheta(\cdot)=L_\vartheta$ when $G_i^\vartheta(\vartheta)>0$ and otherwise $V_j^\vartheta(\cdot)=L_\vartheta$.

It remains to consider $\{\vartheta=\hat\tau^\vartheta\}$. Suppose specifically $\vartheta=\inf\{t\geq\vartheta\mid\alpha_j^\vartheta(t)>0\}$, so $G_j^\vartheta(\vartheta)=1$. Then it is easy to check from Definition \ref{def:outcome} that the value of $V_i^\vartheta(G_i^\vartheta,\alpha_i^\vartheta,G_j^\vartheta,\alpha_j^\vartheta)=\lambda_{L,i}^\vartheta L_\vartheta+\lambda_{L,j}^\vartheta F_\vartheta+\lambda_{M}^\vartheta M_\vartheta$ is a convex combination of $(1-\alpha_j^\vartheta(\vartheta))L_\vartheta+\alpha^\vartheta_j(\vartheta)M_\vartheta$ and $F_\vartheta$, and that these are resp.\ player $i$'s payoffs from the feasible strategies $(G^\vartheta,\alpha^\infty)$ and $(G^\infty,\alpha^\infty)$, where $G^\vartheta(t)=\indi{t\geq\vartheta}$ and $\alpha^\infty(t)=G^\infty(t)=\indi{t\geq\infty}$, when played against the given $(G_j^\vartheta,\alpha_j^\vartheta)$.\footnote{%
In particular in the last case in  Definition \ref{def:outcome}, when $\vartheta=\inf\{t\geq\vartheta\mid\alpha_i^\vartheta(t)>0\}=\inf\{t\geq\vartheta\mid\alpha_j^\vartheta(t)>0\}$, but $\alpha_i^\vartheta(\vartheta)=\alpha_j^\vartheta(\vartheta)=0$, then simply $V_i^\vartheta(\cdot)=\lambda_{L,i}L_\vartheta+(1-\lambda_{L,i})F_\vartheta$.
}
Hence, on $\{F_\vartheta>\max(M_\vartheta,L_\vartheta)\}$, we would need to have $V_i^\vartheta(\cdot)=F_\vartheta$ and thus $\lambda_{L,j}^\vartheta=1$, implying $V_j^\vartheta(\cdot)=L_\vartheta$, contradicting payoff symmetry. Therefore, the first claim must hold. As to the second, given $G_j^\vartheta(\vartheta)=1$, suppose $V_i^\vartheta(\cdot)>\max(F_\vartheta,M_\vartheta)$, so necessarily $\lambda_{L,i}^\vartheta>0$ and $L_\vartheta>F_\vartheta$. Specifically, we would now have $V_i^\vartheta(\cdot)=(1-\alpha_j^\vartheta(\vartheta))L_\vartheta+\alpha^\vartheta_j(\vartheta)M_\vartheta>F_\vartheta$, which is in fact only possible if $\lambda_{L,j}^\vartheta=0$. This, however, would again contradict payoff symmetry.

Still supposing $\vartheta=\inf\{t\geq\vartheta\mid\alpha_j^\vartheta(t)>0\}$, it remains to verify the second claim for $i$ in place of $j$ and $G_i^\vartheta(\vartheta)<1$, whence $\vartheta<\inf\{t\geq\vartheta\mid\alpha_i^\vartheta(t)>0\}$. 
Then $L_\vartheta=F_\vartheta$ by Lemma \ref{lem:dGi=dGj} and the claim holds if $V_i^\vartheta(\cdot)=F_\vartheta$. Otherwise, again $V_i^\vartheta(\cdot)=(1-\alpha_j^\vartheta(\vartheta))L_\vartheta+\alpha^\vartheta_j(\vartheta)M_\vartheta>F_\vartheta$ and thus $\lambda_{L,j}^\vartheta=0$, which would by Definition \ref{def:outcome} clearly contradict $G_i^\vartheta(\vartheta)<1$.
\end{proof}

\begin{proof}[{\bf Proof of Proposition \ref{prop:U_min(L,F)}}]
Suppose $((G_1^\vartheta,\alpha_1^\vartheta);\vartheta\in\T)$, $((G_2^\vartheta,\alpha_2^\vartheta);\vartheta\in\T)$ are a payoff-symmetric subgame-perfect equilibrium, and fix any $\vartheta\in\T$ and $i,j\in\{1,2\}$, $i\neq j$. The proof is build on two observations. First, for any stopping time $\tau\geq\vartheta$, the strategy $(G_i^\tau,\alpha_i^\tau)$ is also feasible in the subgame starting at $\vartheta$, so
\begin{equation}\label{optdelay1}
V_i^\vartheta\bigl(G_i^\vartheta,\alpha_i^\vartheta,G_j^\vartheta,\alpha_j^\vartheta\bigr)\geq V_i^\vartheta\bigl(G_i^\tau,\alpha_i^\tau,G_j^\vartheta,\alpha_j^\vartheta\bigr).
\end{equation}
Second, if $\tau\leq\hat\tau^\vartheta$, then
\begin{align}\label{optdelay2}
&V_i^\vartheta\bigl(G_i^\tau,\alpha_i^\tau,G_j^\vartheta,\alpha_j^\vartheta\bigr)=E\biggl[\int_{[0,\tau)}F\,dG_j^\vartheta+\bigl(1-G_j^\vartheta(\tau-)\bigr)V^\tau_i\bigl(G^\tau_i,\alpha^\tau_i,G^\tau_j,\alpha^\tau_j\bigr)\biggv\F_\vartheta\biggr]
\end{align}
by time-consistency; cf.\ \eqref{(1-G_j-)}. The proof strategy is to identify stopping times $\tau\leq\inf\{t\geq\vartheta\mid G_i^\vartheta(t)\vee G_j^\vartheta(t)\geq 1\}$ such that \eqref{optdelay1} is binding, \eqref{optdelay2} also holds with $L\wedge F$ in place of $F$, and, when passing to a limit, also $L_\tau\wedge F_\tau$ in place of $V_i^\tau(\cdot)$. Then both claims of Proposition \ref{prop:U_min(L,F)} follow from Lemma \ref{lem:BRpure}. Indeed, then the right-hand side of \eqref{optdelay2} equals player $j$'s (!) payoff from the standard mixed strategy represented by $G_j^\vartheta$ when played against $G_a^\vartheta$ given by $G_a^\vartheta(t)=\indi{t\geq\tau}$, and when $L$, $F$ and $M$ are replaced by $L\wedge F$, because then $S_j^\vartheta(t)$ becomes $L_{t\wedge\tau}\wedge F_{t\wedge\tau}$ for all $t\geq\vartheta$.

To establish the desired representation of $V_i^\vartheta(\cdot)$, we begin by showing that \eqref{optdelay1} binds if $G_i^\vartheta(\tau-)\leq 1-\varepsilon$ a.s.\ for some $\varepsilon\in(0,1)$. Suppose the latter holds and set $\lambda=(1-\varepsilon)^{-1}\in(1,\infty)$. Then define $G_a^\vartheta$ by $G_a^\vartheta(t)=\lambda G_i^\vartheta(t)$ for $t\in[0,\tau)$ and $G_a^\vartheta(t)=\lambda G_i^\vartheta(\tau-)+(1-\lambda G_i^\vartheta(\tau-))(G_i^\vartheta(t)-G_i^\vartheta(\tau-))$. This implies $G_i^\vartheta(t)=1\Rightarrow G_a^\vartheta(t)=1$ for all $t\in[0,\infty]$, so $(G_a^\vartheta,\alpha_i^\vartheta)$ is feasible at $\vartheta$, and it satisfies time-consistency with $(G_i^\tau,\alpha_i^\tau)$ by construction. Therefore, we can apply \eqref{(1-G_j-)} also for $G_a^\vartheta$ in place of $G_i^\vartheta$, to see that $V_i^\vartheta(G_a^\vartheta,\alpha_i^\vartheta,G_j^\vartheta,\alpha_j^\vartheta)=\lambda V_i^\vartheta(G_i^\vartheta,\alpha_i^\vartheta,G_j^\vartheta,\alpha_j^\vartheta)-(\lambda-1)V_i^\vartheta(G_i^\tau,\alpha_i^\tau,G_j^\vartheta,\alpha_j^\vartheta)$, and thus, if \eqref{optdelay1} were strict with positive probability, this would contradict optimality of $(G_i^\vartheta,\alpha_i^\vartheta)$.

Hence, with $\tau_i^{G,\vartheta}(x)=\inf\{t\geq\vartheta\mid G_i^\vartheta(t)>x\}$ for any $x\in[0,1)$ and analogously for $j$, \eqref{optdelay1} binds for each of the stopping times
\begin{equation*}
\bar\tau_n:=\tau_i^{G,\vartheta}\bigl(1-n^{-1}\bigr)\wedge\tau_j^{G,\vartheta}\bigl(1-n^{-1}\bigr)\wedge\bar\tau,\quad n\in\N,
\end{equation*}
where $\bar\tau:=\inf\{t\geq\vartheta\mid\int_{[0,t]}\indi{F\geq L}\,(dG_i^\vartheta+dG_j^\vartheta)>0\}$. Moreover, $\bar\tau_n\leq\inf\{t\geq\vartheta\mid G_i^\vartheta(t)\vee G_j^\vartheta(t)\geq 1\}\leq\hat\tau^\vartheta$, such that \eqref{optdelay2} applies, and also with $L\wedge F$ in place of $F$ by definition of $\bar\tau$. The latter still hold for the monotone limit $\bar\tau_\infty:=\lim_{n\to\infty}\bar\tau_n\leq\bar\tau$, but then \eqref{optdelay1} is not necessarily binding anymore. Therefore, we need to show that $E[V_i^\vartheta(G_i^\vartheta,\alpha_i^\vartheta,G_j^\vartheta,\alpha_j^\vartheta)]=E[V_i^\vartheta(G_i^{\bar\tau_\infty},\alpha_i^{\bar\tau_\infty},G_j^\vartheta,\alpha_j^\vartheta)]$, resp., as \eqref{optdelay1} binds for each $\bar\tau_n$, that we can pass to the limit in expectation in \eqref{optdelay2}, with a suitable continuation payoff.

The integral converges as $n\to\infty$, also in expectation because it is of class {\rm (D)} by Lemma \ref{lem:SclassD}. As to the continuation payoffs, we first argue that $\{V_i^{\bar\tau_n}(\cdot);n\in\N\}$ is uniformly integrable. By Definition \ref{def:payoffs_extended}, and as $\lambda_{L,i}^\vartheta,\lambda_{L,j}^\vartheta,\lambda_{M}^\vartheta\in[0,(1-G_i^\vartheta(\hat\tau^\vartheta-))(1-G_j^\vartheta(\hat\tau^\vartheta-)]\subseteq[0,1]$,
\begin{align*}
\abs{V_i^\vartheta\bigl(G_i^\vartheta,\alpha_i^\vartheta,G_j^\vartheta,\alpha_j^\vartheta\bigr)}\leq E\biggl[&\int_{[0,\infty)}\abs{L}\,dG_i^\vartheta+\int_{[0,\infty)}\abs{F}\,dG_j^\vartheta+\int_{[0,\infty)}\abs{M}\,dG_i^\vartheta \\
&+\abs{L_{\hat\tau^\vartheta}}+\abs{F_{\hat\tau^\vartheta}}+\abs{M_{\hat\tau^\vartheta}}\biggv\F_\vartheta\biggr].
\end{align*}
Applying Lemma \ref{lem:BRpure} with $G^\infty$ given by $G^\infty(t)=\indi{t\geq\infty}$ in place of $G_j^\vartheta$ (whence $S_i^\vartheta(t)=L(t)$ for all $t\in\R_+$) and $\abs{L}$ in place of $L$ shows that $E[\int_{[0,\infty)}\abs{L}\,dG_i^\vartheta\mid\F_\vartheta]\leq\esssup_{\tau\geq\vartheta}E[\abs{L_\tau}\mid\F_\vartheta]$. The latter is the value of the Snell envelope of $\abs{L}$, denoted $U_{\abs{L}}$, which is itself bounded by the uniformly integrable martingale $M_{\abs{L}}$ from its Doob-Meyer decomposition; cf.\ Section \ref{subsec:optstop}. Treating the other integrals and remaining terms analogously, hence
\begin{align*}
\abs{V_i^\vartheta\bigl(G_i^\vartheta,\alpha_i^\vartheta,G_j^\vartheta,\alpha_j^\vartheta\bigr)}\leq 2\Bigl(M_{\abs{L}}(\vartheta)+M_{\abs{F}}(\vartheta)+M_{\abs{M}}(\vartheta)\Bigr).
\end{align*}
As this holds for \emph{any} $\vartheta\in\T$ and feasible strategies for the respective subgame, the family $\{2(M_{\abs{L}}(\bar\tau_n)+M_{\abs{F}}(\bar\tau_n)+M_{\abs{M}}(\bar\tau_n));n\in\N\}$ is a uniformly integrable bound for $\{\lvert V_i^{\bar\tau_n}(\cdot)\rvert;n\in\N\}$.

Now, concerning convergence, first note that when $G_j^\vartheta(\bar\tau_\infty-)=1$, then $dG_j^\vartheta$ only charges $\{F<L\}$, so also $G_i^\vartheta(\bar\tau_\infty-)=1$ by Lemma \ref{lem:Gi=Gj} and vice-versa. Therefore, on $\{G_i^\vartheta(\bar\tau_\infty-)=1\}=\{G_j^\vartheta(\bar\tau_\infty-)=1\}$, $(1-G_j^\vartheta(\bar\tau_n-))V_i^{\bar\tau_n}(\cdot)\to 0$ in expectation by Lemma \ref{lem:E(YZ)to0}. The remaining set is $\{G_i^\vartheta(\bar\tau_\infty-)\vee G_j^\vartheta(\bar\tau_\infty-)<1\}=\{\exists n\in\N\colon\bar\tau_n=\bar\tau\}$, so that $V_i^{\bar\tau_n}(\cdot)\to V_i^{\bar\tau}(\cdot)=V_i^{\bar\tau_\infty}(\cdot)$ a.s.\ on this set when $n\to\infty$, and we may pass to the limit in expectation by uniform integrability.

It remains to show that on $\{G_i^\vartheta(\bar\tau_\infty-)\vee G_j^\vartheta(\bar\tau_\infty-)<1\}$, the limit continuation payoff $V_i^{\bar\tau_\infty}(\cdot)=V_i^{\bar\tau}(\cdot)$ equals $L_{\bar\tau}\wedge F_{\bar\tau}$. We are going to apply the optimality condition from Lemma \ref{lem:BRpure} at $\bar\tau$, but in a variant such that equality holds for \emph{every} $x\in[0,1)$ (in fact independently of payoff symmetry, $F\wedge L\geq M$, and the starting date). Therefore, consider any $\tau\in\T$, the standard mixed strategies represented by $G_i^\tau$, $G_j^\tau$, and define $\bar S_i^\tau(t):=\max(S_i^\tau(t),S_i^\tau(t+))$ for all $t\in\R_+$, implying $\bar S_i^\tau(t)\geq\bar S_i^\tau(t+)=S_i^\tau(t+)$; set furthermore $\bar S_i^\tau(\infty):=S_i^\tau(\infty)$. Denoting the right-hand side of \eqref{stopSi} for $\tau$ in place of $\vartheta$ by $Z_i^\tau$, it holds that $E[S_i^\tau(\tau')\mid\F_\tau]\leq Z_i^\tau$ for every stopping time $\tau'\geq\tau$, with equality for a.e.\ $x\in[0,1)$ if $\tau'=\tau_i^{G,\tau}(x)$, as $V_i^\tau(\cdot)=Z_i^\tau$ in equilibrium. Keeping $Z_i^\tau$ fixed, we now show that the same conditions hold for $\bar S_i^\tau$, even with equality for every $x\in[0,1)$. Hence, consider an arbitrary stopping time $\tau'\geq\tau$ and let $A\in\F_{\tau'}$ be the event $\{S_i^\tau(\tau'+)>S_i^\tau(\tau')\}=\{\Delta G_j^\tau(\tau')(F_{\tau'}-M_{\tau'})>0\}$ by right-continuity of $L$ and $G_j^\tau$. With the stopping times $\tau'+\indinb{A}n^{-1}$ for any $n\in\N$ (cf.\ Lemma \ref{lem:tauC}) then $S_i^\tau(\tau'+\indinb{A}n^{-1})\to\bar S_i^\tau(\tau')$ a.s.\ as $n\to\infty$. Moreover, for any $B\in\F_\tau$, $E[\indinb{B}S_i^\tau(\tau'+\indinb{A}n^{-1})]\leq E[\indinb{B}Z_i^\tau]$ by iterated expectations, where we may pass to the limit as $S_i^\tau$ is of class {\rm (D)}. Choosing $B=\{E[\bar S_i^\tau(\tau')\mid\F_\tau]>Z_i^\tau\}$ shows that this event must have probability zero. In particular, for every $x\in[0,1)$ now $E[S_i^\tau(\tau_i^{G,\tau}(x))\mid\F_\tau]\leq E[\bar S_i^\tau(\tau_i^{G,\tau}(x))\mid\F_\tau]\leq Z_i^\tau$, with equality throughout for a.e.\ $x$.\footnote{%
This implies that $S_i^\tau(\tau_i^{G,\tau}(x))=\bar S_i^\tau(\tau_i^{G,\tau}(x))$ a.s.\ for a.e.\ $x\in[0,1)$, so that the proof of Lemma \ref{lem:BRpure} shows that also $E[\int S_i^\tau\,dG_i^\tau\mid\F_\tau]=E[\int\bar S_i^\tau\,dG_i^\tau\mid\F_\tau]$.
} 
Now fix any $x\in[0,1)$, and let for every $n\in\N$ then $x_n\in(x+(1-x)(n+1)^{-1},x+(1-x)n^{-1}]$ be such that $E[\bar S_i^\tau(\tau_i^{G,\tau}(x_n))\mid\F_\tau]=Z_i^\tau$. By monotonicity and right-continuity of $\tau_i^{G,\tau}(\cdot)$, $\tau_i^{G,\tau}(x_n)\searrow\tau_i^{G,\tau}(x)$ a.s., and as $\bar S_i^\tau(\cdot)$ is upper-semi-continuous from the right and of class {\rm (D)}, thus $E[\bar S_i^\tau(\tau_i^{G,\tau}(x))]\geq\lim_{n\to\infty}E[\bar S_i^\tau(\tau_i^{G,\tau}(x_n))]=E[Z_i^\tau]$. This shows that $E[\bar S_i^\tau(\tau_i^{G,\tau}(x))\mid\F_\tau]\leq Z_i^\tau$ must be a.s.\ binding, too. Note that, as $S_i^\tau(t+)=S_i^\tau(t)+\Delta G_j^\tau(t)(F_t-M_t)$, the assumption $F\geq M$ induces that simply $\bar S_i^\tau(t)=S_i^\tau(t+)=\int_{[0,t]}F\,dG_j^\tau+(1-G_j^\tau(t))L_t$ for all $t\in\R_+$.

Now, to characterize $V_i^{\bar\tau}(\cdot)$, note that in the subgame starting at $\bar\tau$, time-consistency implies $\bar\tau=\tau_i^{G,\bar\tau}(0)\wedge\tau_j^{G,\bar\tau}(0)$, so a.s.\ $\bar S_i^{\bar\tau}(\bar\tau)=Z_i^{\bar\tau}$ or $\bar S_j^{\bar\tau}(\bar\tau)=Z_i^{\bar\tau}$ as argued before, and thus $V_i^{\bar\tau}(\cdot)=G_j^{\bar\tau}(\bar\tau)F_{\bar\tau}+(1-G_j^{\bar\tau}(\bar\tau))L_{\bar\tau}$ or $V_j^{\bar\tau}(\cdot)=G_i^{\bar\tau}(\bar\tau)F_{\bar\tau}+(1-G_i^{\bar\tau}(\bar\tau))L_{\bar\tau}$. By right-continuity, $F_{\bar\tau}\geq L_{\bar\tau}$, and this must be binding by Lemma \ref{lem:DG>0} and the assumption $L\geq M$ whenever $G_i^{\bar\tau}(\bar\tau)>0$ or $G_j^{\bar\tau}(\bar\tau)>0$, so we must a.s.\ have $V_i^{\bar\tau}(\cdot)=V_j^{\bar\tau}(\cdot)=L_{\bar\tau}\leq F_{\bar\tau}$.

This completes the proof. Note that by Lemma \ref{lem:DG>0} and the assumption $L\geq M$, we can also state the result in the form
\begin{equation*}
V_i^\vartheta\bigl(G_i^\vartheta,\alpha_i^\vartheta,G_j^\vartheta,\alpha_j^\vartheta\bigr)\leq\Delta G_j^\vartheta(\vartheta)F_\vartheta+\bigl(1-\Delta G_j^\vartheta(\vartheta)\bigr)U_{L\wedge F}(\vartheta)\leq U_{L\wedge F}(\vartheta).\qedhere
\end{equation*}
\end{proof}

\subsection{Proof of Theorem \ref{thm:maxeql}}\label{app:maxeql}

Proposition \ref{prop:U_min(L,F)} and proof show that in any payoff-symmetric subgame-perfect equilibrium, the continuation values at any $\vartheta\in\T$ are at most
\begin{equation*}
\Delta G_j^\vartheta(\vartheta)F_\vartheta+\bigl(1-\Delta G_j^\vartheta(\vartheta)\bigr)\esssup_{\tau\in\T\colon\tau\geq\vartheta}E\bigl[L_\tau\wedge F_\tau\bigv\F_\vartheta\bigr]\leq U_{L\wedge F}(\vartheta),
\end{equation*}
in fact considering only stopping times $\tau\leq\inf\{t\geq\vartheta\mid G_i^\vartheta(t)\vee G_j^\vartheta(t)=1\}$. However, when $L_\vartheta>U_{L\wedge F}(\vartheta)$ and $G_j^\vartheta(\vartheta)<1$, then $\vartheta<\inf\{t\geq\vartheta\mid\alpha_j^\vartheta(t)>0\}$ and player $i$'s value $V_i^\vartheta(\cdot)$ would be at least $\Delta G_j^\vartheta(\vartheta)F_\vartheta+(1-\Delta G_j^\vartheta(\vartheta))L_\vartheta$ by considering stopping times $\vartheta+n^{-1}$ and $n\to\infty$ analogously to the proof of Lemma \ref{lem:DG>0}; a contradiction. Therefore, fixing $\vartheta$ and letting $\tau_n=\inf\{t\geq\vartheta\mid L_t-U_{L\wedge F}(t)\geq n^{-1}\}$ for each $n\in\N$, necessarily $G_j^{\tau_n}(\tau_n)=1$ as $L_{\tau_n}>U_{L\wedge F}(\tau_n)$ on $\{\tau_n<\infty\}$ by right-continuity under Assumption \ref{asm:payoffs}\,\ref{LFusc} (cf.\ Section \ref{subsec:optstop}). Time-consistency implies also $G_j^\vartheta(\tau_n)=1$, and hence, as $\tau_n\searrow\tau_0(\vartheta)$, $G_j^\vartheta(\tau_0(\vartheta))=1$ for some $j\in\{1,2\}$. The continuation values are then actually bounded by $U_{(L\wedge F)^{\tau_0(\vartheta)}}\leq U_{L\wedge F}$. This argument of course iterates, inducing decreasing sequences of stopping times $(\tau_n(\vartheta))_n$ and Snell envelopes $(U_{(L\wedge F)^{\tau_n(\vartheta)}})_n$, as increasingly strict constraints are introduced. The decreasing sequence of stopping times is bounded below by $\vartheta$, so the limit $\tilde\tau(\vartheta)=\lim_{n\to\infty}\tau_n(\vartheta)$ exists and is a stopping time as well. 

It is convenient to have optimal stopping times that resp.\ attain $U_{(L\wedge F)^{\tau_n(\vartheta)}}$. They exist because each process $(L\wedge F)^{\tau_n(\vartheta)}$ is clearly right-continuous and of class {\rm (D)}, but also upper-semi-continuous in expectation under Assumption \ref{asm:payoffs}\,\ref{LFusc}. These optimal stopping times simplify the argument to prove that $\tilde\tau(\vartheta)\leq\inf\{t\geq\vartheta\mid L_t>U_{(L\wedge F)^{\tilde\tau(\vartheta)}}(t)\}$. Consider an arbitrary stopping time $\sigma\in[\vartheta,\tilde\tau(\vartheta)]$, and for each $n\in\N$, let $\sigma_n\leq\tau_n(\vartheta)$ be a stopping time attaining $U_{(L\wedge F)^{\tau_n(\vartheta)}}(\sigma)=E[(L\wedge F)_{\sigma_n}\mid\F_\sigma]$, where we may assume $\sigma_n\geq\sigma_{n+1}$ because $E[(L\wedge F)_{\sigma_{n+1}}\mid\F_{\sigma_n\wedge\sigma_{n+1}}]\leq(L\wedge F)_{\sigma_n\wedge\sigma_{n+1}}$ by optimality of $\sigma_n$ and $\sigma_{n+1}\leq\tau_{n+1}(\vartheta)\leq\tau_n(\vartheta)$. Therefore, the monotone sequence $(\sigma_n)_n$ a.s.\ converges to a limit stopping time $\sigma_\infty\leq\tilde\tau(\vartheta)$. Hence, as $L_\sigma\leq E[(L\wedge F)_{\sigma_n}\mid\sigma]$ a.s.\ on $\{\sigma<\tilde\tau(\vartheta)\}$ for every $n\in\N$, and as $(L\wedge F)$ is of class {\rm (D)}, $E[\indinb{A}L_\sigma]\leq E[\indinb{A}(L\wedge F)_{\sigma_\infty}]$ for every $A\in\F_\sigma$ that is a subset of $\{\sigma<\tilde\tau(\vartheta)\}$. Choosing $A=\{L_\sigma>E[(L\wedge F)_{\sigma_\infty}\mid\F_\sigma]\}\cap\{\sigma<\tilde\tau(\vartheta)\}$ then shows that this has probability zero. In particular, thus $L_\sigma\leq E[(L\wedge F)_{\sigma_\infty}\mid\F_\sigma]\leq U_{(L\wedge F)^{\tilde\tau(\vartheta)}}(\sigma)$ on $\{\sigma<\tilde\tau(\vartheta)\}$. Let now $\tau_n=\inf\{t\geq\vartheta\mid L(t)-U_{(L\wedge F)^{\tilde\tau(\vartheta)}}(t)\geq n^{-1}\}\wedge\tilde\tau(\vartheta)$ for each $n\in\N$. By right-continuity then $L_{\tau_n}>U_{(L\wedge F)^{\tilde\tau(\vartheta)}}(\tau_n)$ on $\{\tau_n<\tilde\tau(\vartheta)\}$, so by the previous argument we must have $\tau_n=\tilde\tau(\vartheta)$ a.s.\ for every $n\in\N$. This shows that also $\inf\{t\geq\vartheta\mid L(t)>U_{(L\wedge F)^{\tilde\tau(\vartheta)}}(t)\}\wedge\tilde\tau(\vartheta)=\tilde\tau(\vartheta)$ a.s.\ as desired, resp.\ $L_t\leq U_{(L\wedge F)^{\tilde\tau(\vartheta)}}(t)$ for all $t\in[\vartheta,\tilde\tau(\vartheta))$ a.s.

$\tilde\tau(\vartheta)$ is maximal in this respect by construction: Whenever we have another $\hat\tau\geq\vartheta$ such that $L\leq U_{(L\wedge F)^{\hat\tau}}$ on $[\vartheta,\hat\tau)$, then $\hat\tau\leq\tau_0(\vartheta)$ a.s., because $U_{(L\wedge F)^{\hat\tau}}\leq U_{L\wedge F}$. So $U_{(L\wedge F)^{\hat\tau}}\leq U_{(L\wedge F)^{\tau_0(\vartheta)}}$, and by iteration $\hat\tau\leq\tau_n(\vartheta)$ for all $n\in\N$.

Let us remark some further regularity properties. First,
\begin{equation}\label{prmptL>F}
\tilde\tau(\vartheta)=\inf\{t\geq\tilde\tau(\vartheta)\mid L_t>F_t\}\qquad\text{a.s.}
\end{equation}
and thus $L_{\tilde\tau(\vartheta)}\geq F_{\tilde\tau(\vartheta)}$ by right-continuity, because we would otherwise have $L=L\wedge F\leq U_{(L\wedge F)^{\tau_n(\vartheta)}}$ for all $n\in\N_0$ between $\tilde\tau(\vartheta)$ and the stopping time on the right-hand side. Second, for any $\vartheta'\in\T$,
\begin{equation}\label{tildetauTC}
\tilde\tau(\vartheta')=\tilde\tau(\vartheta)\quad\text{on}\quad\{\vartheta'\in[\vartheta,\tilde\tau(\vartheta)]\}\qquad\text{a.s.} 
\end{equation}
by construction. Therefore, $\tilde\tau(\vartheta)$ is monotone in $\vartheta$ and 
\begin{equation}\label{prmptdense}
\tilde\tau(\vartheta)=\lim_{n\to\infty}\tilde\tau(\vartheta_n)\qquad\text{a.s.}
\end{equation}
for any sequence of stopping times $\vartheta_n\searrow\vartheta$. Indeed, denoting the monotone limit on the right-hand side by $\tau_\infty\geq\tilde\tau(\vartheta)$, then $\vartheta_n\vee\tau_\infty\in[\vartheta_n,\tilde\tau(\vartheta_n)]\Rightarrow\tilde\tau(\tau_\infty)\leq\tilde\tau(\vartheta_n\vee\tau_\infty)=\tilde\tau(\vartheta_n)$ for all $n\in\N$, so $\tau_\infty=\tilde\tau(\tau_\infty)$ a.s. Then also $\tau_\infty\leq\tilde\tau(\vartheta_n)$ must bind on $\{\vartheta_n\leq\tau_\infty\}$ for every $n\in\N$, implying $\tilde\tau(\vartheta_n\wedge\tau_\infty)=\tau_\infty$ and hence $L_t\leq U_{(L\wedge F)^{\tau_\infty}}(t)$ for all $t\in[\vartheta_n\wedge\tau_\infty,\tau_\infty)$ a.s. The latter holds then a.s.\ for all $n\in\N$, i.e., the inequality holds a.s.\ on $(\vartheta,\tau_\infty)$ and by right-continuity also at $\vartheta$ on $\{\vartheta<\tau_\infty\}$. The maximality of $\tilde\tau(\vartheta)$ regarding this property now implies that $\tau_\infty\geq\tilde\tau(\vartheta)$ cannot be strict with positive probability.

Next, define the process $\tilde L$ by
\begin{equation}\label{tildeL_F>M}
\tilde L_t:=\begin{cases}
L_t & t<\tilde\tau(\vartheta)\\
F_{\tilde\tau(\vartheta)} & t\geq\tilde\tau(\vartheta)
\end{cases}
\end{equation}
with Snell envelope $U_{\tilde L}$ and its compensator $D_{\tilde L}$. From $L_{\tilde\tau(\vartheta)}\geq F_{\tilde\tau(\vartheta)}$ and $L\leq U_{(L\wedge F)^{\tilde\tau(\vartheta)}}$ on $[\vartheta,\tilde\tau(\vartheta))$, we see that $(L\wedge F)^{\tilde\tau(\vartheta)}\leq \tilde L\leq U_{(L\wedge F)^{\tilde\tau(\vartheta)}}\leq U_{\tilde L}$ on $[\vartheta,\infty]$. As the Snell envelope is the smallest supermartingale dominating the payoff process, the last inequality must actually bind and thus also $D_{\tilde L}=D_{(L\wedge F)^{\tilde\tau(\vartheta)}}$ on $[\vartheta,\infty]$. By this observation we obtain 
\begin{equation}\label{D_tildeLreg}
\int_{[\vartheta,\infty]}\indi{L>F}\,dD_{\tilde L}=\sum_{[\vartheta,\infty]}\Delta D_{\tilde L}=0\qquad\text{a.s.}
\end{equation}
Indeed, the path regularity of $(L\wedge F)^{\tilde\tau(\vartheta)}$ verified before implies continuity of $D_{(L\wedge F)^{\tilde\tau(\vartheta)}}$, and \eqref{U_L=L} applied to $(L\wedge F)^{\tilde\tau(\vartheta)}$ in place of $L$ with $U_{(L\wedge F)^{\tilde\tau(\vartheta)}}\geq L$ on $[\vartheta,\tilde\tau(\vartheta))$ and $D_{(L\wedge F)^{\tilde\tau(\vartheta)}}=D_{\tilde L}$ yields 
\begin{align*}
0&=\int_{[0,\infty]}\Bigl(U_{(L\wedge F)^{\tilde\tau(\vartheta)}}-(L\wedge F)^{\tilde\tau(\vartheta)}\Bigr)\,dD_{\tilde L} \\
&\geq\int_{[\vartheta,\tilde\tau(\vartheta))}\bigl(L-(L\wedge F)\bigr)\,dD_{\tilde L}=\int_{[\vartheta,\tilde\tau(\vartheta))}\indi{L>F}\bigl(L-F\bigr)\,dD_{\tilde L}\geq 0,
\end{align*}
which must hold with equality throughout. Furthermore, $D_{\tilde L}=D_{(L\wedge F)^{\tilde\tau(\vartheta)}}$ is right-continuous and flat from $\tilde\tau(\vartheta)$ on, so it does not charge $[\tilde\tau(\vartheta),\infty]$.

Now we can define equilibrium strategies $((G_1^\vartheta,\tilde\alpha_1^\vartheta);\vartheta\in\T)$ and $((G_2^\vartheta,\tilde\alpha_2^\vartheta);\vartheta\in\T)$ as follows. Pick $i,j\in\{1,2\}$, $i\neq j$, and then define $G_i^\vartheta$, $G_j^\vartheta$ for every $\vartheta\in\T$ as in Theorem \ref{thm:mixedeql} with $\tau^\vartheta=\tilde\tau(\vartheta)$. $\tilde\alpha_i^\vartheta$, $\tilde\alpha_j^\vartheta$ will be defined later, as we need a careful construction. The family $(G_i^\vartheta;\vartheta\in\T)$ satisfies the time-consistency condition from Theorem \ref{thm:SPE} by \eqref{tildetauTC}; the same holds for $j$. Moreover, given $\tau^\vartheta=\tilde\tau(\vartheta)$ and the assumption $F\geq M$, $\tilde L^{\tau^\vartheta}$ from Theorem \ref{thm:mixedeql} is the same as $\tilde L$ from \eqref{tildeL_F>M} for any fixed $\vartheta$. Therefore, the regularity conditions for $D_{\tilde L}^{\tau^\vartheta}=D_{\tilde L}$ ensured by the hypothesis $\tau^\vartheta\leq\inf\{t\geq\mid L_t>F_t\}$ in Theorem \ref{thm:mixedeql} (which had no other role) are now implied by \eqref{D_tildeLreg}: continuity and that $L_{\tau_i^\vartheta}=F_{\tau_i^\vartheta}$ on $\{\tau_i^\vartheta<\tau^\vartheta\}$; the latter now follows from $\int\indi{F\leq L}\,dD_{\tilde L}=\int\indi{F=L}\,dD_{\tilde L}$ (and right-continuity of $L-F$). Given that now furthermore $F_{\tau_i^\vartheta}\geq M_{\tau_i^\vartheta}$ on $\{\tau_i^\vartheta<\tau^\vartheta\}$ holds by assumption, this means that the proof of Theorem \ref{thm:symeql} can be applied if we achieve continuation equilibrium payoffs $V_i^{\tilde\tau(\vartheta)}(\cdot)=V_j^{\tilde\tau(\vartheta)}(\cdot)=F_{\tilde\tau(\vartheta)}$ ($\geq M_{\tilde\tau(\vartheta)}$ by assumption) for any $\vartheta\in\T$ from time-consistent, feasible families $(\tilde\alpha_i^\vartheta;\vartheta\in\T)$, $(\tilde\alpha_j^\vartheta;\vartheta\in\T)$. This will then induce equilibrium payoffs $U_{\tilde L}(\vartheta)=U_{(L\wedge F)^{\tilde\tau(\vartheta)}}(\vartheta)$ at any $\vartheta\in\T$.

We need to suppress $\alpha_i^\vartheta$, $\alpha_j^\vartheta$ as given by Proposition \ref{prop:eqlL>F} on the random intervals $[\vartheta,\tilde\tau(\vartheta))$ in a measurable way across all $\vartheta\in\T$, while ensuring that they still generate suitable equilibrium payoffs at each $\tilde\tau(\vartheta)$. Therefore, begin with the (optional) set $A=\bigcup_{q\in\Q_+}[q,\tilde\tau(q))$. Then approximate any given $\vartheta\in\T$ from above by the decreasing sequence $(\vartheta_n)_{n\in\N}$, where $\vartheta_n=2^{-n}\ceil{2^n\vartheta+1}\in\T$. Any $\vartheta_n$ takes values only in $\{k2^{-n};k\in\bar\N\}$ and thus $[\vartheta_n,\tilde\tau(\vartheta_n))=\sum_{k\in\N}[k2^{-n},\tilde\tau(k2^{-n}))\indi{\vartheta_n=k2^{-n}}$ a.s., which means that $[\vartheta_n,\tilde\tau(\vartheta_n))$ belongs to $A$ up to a $P$-nullset. Furthermore, $\vartheta_n\leq\vartheta+2^{1-n}$ implies $\tilde\tau(\vartheta)=\tilde\tau(\vartheta_k)$ for all $k\geq n$ on $\{\tilde\tau(\vartheta)\geq\vartheta+2^{1-n}\}$ a.s., and therefore on this set $(\vartheta,\tilde\tau(\vartheta))=\bigcup_{k\geq n}[\vartheta_k,\tilde\tau(\vartheta_k))$ a.s., which means that $(\vartheta,\tilde\tau(\vartheta))$ belongs to $A$ on $\{\tilde\tau(\vartheta)\geq\vartheta+2^{1-n}\}$ up to a $P$-nullset. Aggregating, $(\vartheta,\tilde\tau(\vartheta))=(\vartheta,\tilde\tau(\vartheta))(\indi{\tilde\tau(\vartheta)\geq\vartheta+2^{0}}+\sum_{n}\indi{\tilde\tau(\vartheta)\in[\vartheta+2^{-n},\vartheta+2^{1-n})})$ a.s.\ and hence belongs to $A$ up to a $P$-nullset. 

Now define $\tilde\alpha_i^\vartheta=\tilde\alpha_j^\vartheta$ by
\begin{equation*}
\tilde\alpha_i^\vartheta(t):=\Bigl(\limsup_{u\searrow t}\indi{u\in A^c}\Bigr)\alpha_i^\vartheta(t)
\end{equation*}
for any $t\in\R_+$ and $\vartheta\in\T$, with $\alpha_i^\vartheta$ as in Proposition \ref{prop:eqlL>F}. $\tilde\alpha_i^\vartheta$ inherits progressive measurability from its two factors, where that of the $\limsup(\cdot)$ holds by Theorem IV.33 (c) in \cite{DellacherieMeyer78}. Like $\alpha_i^\vartheta$, $\tilde\alpha_i^\vartheta$ satisfies $\tilde\alpha_i^\vartheta(t)=\indi{t\geq\vartheta}\tilde\alpha_i^{\vartheta_0}(t)$ for all $t\in\R_+$, where $\vartheta_0\equiv 0$, which also ensures time-consistency across $\vartheta$. Additionally, now $\tilde\alpha_i^\vartheta(t)=0$ for all $t\in[\vartheta,\tilde\tau(\vartheta))$ a.s., so $\tilde\alpha_i^\vartheta(t)>0\Rightarrow G_i^\vartheta(t)=G_j^\vartheta(t)=1$. For feasibility, it thus only remains to verify right-continuity when $\tilde\alpha_i^\vartheta(t)<1$. First note that when $\limsup_{u\searrow t}\indi{u\in A^c}=0$, then also its right-hand limit vanishes by upper-semi-continuity, such that also $\tilde\alpha_i^\vartheta(t)=\tilde\alpha_i^\vartheta(t+)$.  When $\limsup_{u\searrow t}\indi{u\in A^c}=1$ and $\tilde\alpha_i^\vartheta(t)<1$, then also $\alpha_i^\vartheta(t)<1$, so that $\alpha_i^\vartheta(\cdot)$ is right-continuous in $t$, and it thus remains to ensure right-continuity of $t\mapsto\limsup_{u\searrow t}\indi{u\in A^c}$ when $\tilde\alpha_i^\vartheta(t)\in(0,1)$.

Therefore, consider any fixed $\vartheta\in\T$ and the associated set $B=\{\limsup_{u\searrow\vartheta}\indi{u\in A^c}>\liminf_{u\searrow\vartheta}\indi{u\in A^c}\}$. On $B$, $\vartheta=\tilde\tau(\vartheta)$ a.s.\ because $(\vartheta,\tilde\tau(\vartheta))$ belongs to $A$ up to a $P$-nullset. Consider also a sequence $\vartheta_n\searrow\vartheta$ (e.g., the previous one). By \eqref{prmptdense} then $\tilde\tau(\vartheta_n)\searrow\vartheta$ a.s.\ on $B$. Moreover, for each $n\in\N$, as $\vartheta_n>\vartheta$ and $\liminf_{u\searrow\vartheta}\indi{u\in A^c}=0$ on $B$, there must exist some $q\in(\vartheta,\vartheta_n)$ with $q<\tilde\tau(q)$ for each $\omega\in B$, resp.\ $\sup_{q\in\Q_+}\indi{q>\vartheta}((\tilde\tau(q)\wedge\vartheta_n)-q)>0$ on $B$. The aim is now to select such $q(\omega)$ measurably, such that they form a stopping time $\tau_n$. Therefore, let 
\begin{equation*}
l_n=\min\{l\in\N\mid P[\indinb{B}\sup_{q\in\Q_+}\indi{q>\vartheta}((\tilde\tau(q)\wedge\vartheta_n)-q)>2^{-l}]\geq(1-n^{-1})P[B]\}.
\end{equation*} 
This means that the probability that $B$ is realized and that some (grid) point $2^{k-l_n}$ with $k\in\N$ belongs to an interval $[q,\tilde\tau(q))\cap(\vartheta,\vartheta_n)$, i.e., to $A\cap(\vartheta,\vartheta_n)$, is at least $(1-n^{-1})P[B]$. As monotonicity implies $\tilde\tau(q)\leq\tilde\tau(2^{k-l_n})$ for all rational $q\leq 2^{k-l_n}$ a.s.\ for any given $k$, we have $P[\bigcup_{k\in\N}\{2^{k-l_n}=\tilde\tau(2^{k-l_n})\}\cap\{2^{k-l_n}\in A\}]=0$ and thus still $P[B\cap(\bigcup_{k\in\N}\{2^{k-l_n}\in A\cap(\vartheta,\vartheta_n)\}\cap\{2^{k-l_n}<\tilde\tau(2^{k-l_n})\})]\geq(1-n^{-1})P[B]$. Hence, if we define stopping times $\tau_{n,k}$ with value $2^{k-l_n}$ if this is in $(\vartheta,\tilde\tau(2^{k-l_n})\wedge\vartheta_n)$ and $\vartheta_n$ else for every $k\in\N$, then these only take values in $(\{2^{k-l_n}\mid k\in\N\}\cap(\vartheta,\vartheta_n))\cup\{\vartheta_n\}$ and $P[B\cap(\bigcup_{k\in\N}\{\tau_{n,k}=2^{k-l_n}\in(\vartheta,\tilde\tau(2^{k-l_n})\wedge\vartheta_n)\})]\geq(1-n^{-1})P[B]$, so we can also define $\tau_n=\min_{k\in\N}\tau_{k,n}\in(\vartheta,\vartheta_n)$ satisfying $\tilde\tau(\tau_n)=\indi{\tau_n=\vartheta_n}\tilde\tau(\vartheta_n)+\sum_{k\in\N}\indi{\tau_n=2^{k-l_n}}\tilde\tau(2^{k-l_n})$ a.s.\ by \eqref{tildetauTC} and thus $P[B\cap\{\tau_n<\tilde\tau(\tau_n)\}]\geq(1-n^{-1})P[B]$.

Letting $C_n=\{\tau_n<\tilde\tau(\tau_n)\}$ for each $n\in\N$, now $L_{\tau_n}\leq U_{(L\wedge F)^{\tilde\tau(\tau_n)}}(\tau_n)$ a.s.\ on $C_n$ and $P[B\cap C_n^c]\leq n^{-1}P[B]$. Moreover, iterated expectations with $B\cap B'\cap C_n^c\in\F_{\tau_n}$ for any further set $B'\in\F_\vartheta$ and using that $U_{(L\wedge F)^{\tilde\tau(\tau_n)}}(\cdot)$ is a supermartingale yield $E[\indinb{B\cap B'\cap C_n^c}(L_{\tau_n}-U_{(L\wedge F)^{\tilde\tau(\tau_n)}}(\tau_n))]\leq E[\indinb{B\cap B'\cap C_n^c}(L_{\tau_n}-U_{(L\wedge F)^{\tilde\tau(\tau_n)}}(\tilde\tau(\tau_n)))]=E[\indinb{B\cap B'\cap C_n^c}(L_{\tau_n}-(L\wedge F)_{\tilde\tau(\tau_n)})]$, which vanishes by Lemma \ref{lem:E(YZ)to0} as $n\to\infty$ (letting $t=(n-1)n^{-1}$), because $L$ and $F$ are of class {\rm (D)} and $\indinb{B\cap B'\cap C_n^c}\to 0$ in probability. Therefore, $\limsup_{n\to\infty}E[\indinb{B\cap B'}(L_{\tau_n}-U_{(L\wedge F)^{\tilde\tau(\tau_n)}}(\tau_n))]\leq\limsup_{n\to\infty}E[\indinb{B\cap B'\cap C_n}(L_{\tau_n}-U_{(L\wedge F)^{\tilde\tau(\tau_n)}}(\tau_n))]\leq 0$. On the other hand, letting similarly as before $\sigma_n\in[\tau_n,\tilde\tau(\tau_n)]$ resp.\ attain $U_{(L\wedge F)^{\tilde\tau(\tau_n)}}(\tau_n)$ for each $n\in\N$, $E[\indinb{B\cap B'}(L_{\tau_n}-U_{(L\wedge F)^{\tilde\tau(\tau_n)}}(\tau_n))]=E[\indinb{B\cap B'}(L_{\tau_n}-(L\wedge F)_{\sigma_n})]\to E[\indinb{B\cap B'}(L_\vartheta-F_\vartheta)]$ as $n\to\infty$ because $\tau_n\leq\sigma_n\leq\tilde\tau(\tau_n)\leq\tilde\tau(\vartheta_n)\to\vartheta$, $F_\vartheta\leq L_\vartheta$ by $\vartheta=\tilde\tau(\vartheta)$ on $B$, and $L$ and $F$ are right-continuous and of class {\rm (D)}. Choosing $B'=\{L_\vartheta>F_\vartheta\}$ now shows that then $B\cap B'$ has probability zero, i.e., a.s.\ $\limsup_{u\searrow\vartheta}\indi{u\in A^c}=\liminf_{u\searrow\vartheta}\indi{u\in A^c}$ or $L_\vartheta\leq F_\vartheta$.

Now apply this for the desired path property to $\vartheta_n'=\inf\{t\geq 0\mid(\limsup_{u\searrow t}\indi{u\in A^c}-\liminf_{u\searrow t}\indi{u\in A^c})(L_t-F_t)\geq n^{-1}\}$, $n\in\N$. $\vartheta_n'$ is indeed a stopping time by the progressive measurability of $\limsup(\cdot)$ remarked before and similarly of $\liminf(\cdot)$. The latters' difference is upper-semi-continuous and $L-F$ even continuous from the right, so $(\limsup_{u\searrow\vartheta_n'}\indi{u\in A^c}-\liminf_{u\searrow\vartheta_n'}\indi{u\in A^c})(L_{\vartheta_n'}-F_{\vartheta_n'})\geq n^{-1}$ a.s.\ on $\{\vartheta_n'<\infty\}$, showing that in fact $\vartheta_n'=\infty$ a.s. As this still holds simultaneously for all $n\in\N$, indeed a.s.\ $\limsup_{u\searrow t}\indi{u\in A^c}>\liminf_{u\searrow t}\indi{u\in A^c}\Rightarrow L_t\leq F_t$ for all $t\in\R_+$. $L_t\leq F_t$ implies furthermore $\alpha_i^{\vartheta}(t)\in\{0,1\}$ and thus what we intended to show.

Finally, to verify that we obtain equilibrium payoffs as in Proposition \ref{prop:eqlL>F} at each $\vartheta'=\tilde\tau(\vartheta)$, we first show that then $\vartheta'=\inf\{t\geq\vartheta'\mid\tilde\alpha_i^{\vartheta'}(t)>0\}=:\tilde\tau^{\vartheta'}$ a.s. We have $\tilde\tau(\vartheta')=\vartheta'$ and $\tilde\tau(\tilde\tau^{\vartheta'})=\tilde\tau^{\vartheta'}$ a.s., the former by \eqref{tildetauTC} and the latter as $(\tau,\tilde\tau(\tau))$ belongs to $A$ up to a $P$-nullset for any $\tau\in\T$ and thus $\tilde\alpha_i^{\vartheta'}(t)=0$ for all $t\in(\tau,\tilde\tau(\tau))$ a.s. Now suppose that $\tilde\tau^{\vartheta'}>\vartheta'$ with positive probability. Then also $\tilde\tau^{\vartheta'}>\inf\{t\geq\vartheta'\mid L_t>U_{(L\wedge F)^{\tilde\tau^{\vartheta'}}}(t)\}$ with positive probability, because otherwise $\tilde\tau(\vartheta')\geq\tilde\tau^{\vartheta'}>\vartheta'$ with positive probability. Hence, letting $\vartheta_n'=\inf\{t\geq\vartheta'\mid L_t-U_{(L\wedge F)^{\tilde\tau^{\vartheta'}}}(t)\geq n^{-1}\}\wedge\tilde\tau^{\vartheta'}$ for any $n\in\N$, $\lim_{n\to\infty}P[\{\vartheta_n'<\tilde\tau^{\vartheta'}\}]>0$. On $\{\vartheta_n'<\tilde\tau^{\vartheta'}\}$, $L_{\vartheta_n'}>U_{(L\wedge F)^{\tilde\tau^{\vartheta'}}}(\vartheta_n')$ by right-continuity. Then necessarily $L_{\vartheta_n'}>F_{\vartheta_n'}$ and thus $\alpha_i^{\vartheta'}(\vartheta_n')>0$, so $\tilde\alpha_i^{\vartheta'}(\vartheta_n')=0$ implies $\limsup_{u\searrow\vartheta_n'}\indi{u\in A^c}=0$ and furthermore $\vartheta_n'<\tilde\tau(\vartheta_n')$ by \eqref{prmptdense}. This, however, requires $L_{\vartheta_n'}\leq U_{(L\wedge F)^{\tilde\tau(\vartheta_n')}}(\vartheta_n')$, contradicting that the latter is at most $U_{(L\wedge F)^{\tilde\tau^{\vartheta'}}}(\vartheta_n')$ by $\tilde\tau(\vartheta_n')\leq\tilde\tau(\tilde\tau^{\vartheta'})=\tilde\tau^{\vartheta'}$. Hence, $P[\{\vartheta_n'<\tilde\tau^{\vartheta'}\}]=0$ for all $n\in\N$, resp.\ indeed $\tilde\tau^{\vartheta'}=\vartheta'$ a.s. 

Now furthermore $\limsup_{u\searrow\vartheta'}\indi{u\in A^c}=1$ by \eqref{prmptdense} and thus $\tilde\alpha_i^{\vartheta'}(\vartheta')=\alpha_i^{\vartheta'}(\vartheta')$. From Definition \ref{def:outcome}, it is obvious that for \emph{any} choice of $i,j\in\{1,2\}$, $i\neq j$, and any strategy $(G_a^{\vartheta'},\alpha_a^{\vartheta'})$ for player $i$, the outcome probabilities only depend on the values of $\alpha_a^{\vartheta'}$, $\tilde\alpha_j^{\vartheta'}$ at $\vartheta'$, except when both are zero and also $\vartheta'=\inf\{t\geq\vartheta'\mid\alpha_a^{\vartheta'}(t)>0\}$; then also the values on an arbitrarily short interval $(\vartheta',\vartheta'+\varepsilon)$ with $\varepsilon>0$ need to be known. In the former case, the outcome probabilities from any such $\alpha_a^{\vartheta'}$ are thus the same for $\tilde\alpha_j^{\vartheta'}$ and $\alpha_j^{\vartheta'}$. In the latter case, $\lambda_M^{\vartheta'}=0$ for any such $\alpha_a^{\vartheta'}$ and either of $\tilde\alpha_j^{\vartheta'}$ or $\alpha_j^{\vartheta'}$, and $\tilde\alpha_j^{\vartheta'}(\vartheta')=\alpha_j^{\vartheta'}(\vartheta')=0$ implies $L_{\vartheta'}=F_{\vartheta'}$, so player $i$'s expected payoff is $F_{\vartheta'}$, whether against $\tilde\alpha_j^{\vartheta'}$ or $\alpha_j^{\vartheta'}$. As moreover $G_a^{\vartheta'}=G_i^{\vartheta'}$ with $\alpha_a^{\vartheta'}=\tilde\alpha_i^{\vartheta'}$ or $\alpha_i^{\vartheta'}$ resp.\ yields the same expected payoff in each case (in fact even the same outcome probabilities also in the second case due to symmetry), $(G_i^{\vartheta'},\tilde\alpha_i^{\vartheta'})$ and $(G_j^{\vartheta'},\tilde\alpha_j^{\vartheta'})$ inherit the equilibrium property from $(G_i^{\vartheta'},\alpha_i^{\vartheta'})$ and $(G_j^{\vartheta'},\alpha_j^{\vartheta'})$.
\qed

\section{Outcome probabilities}\label{app:outcome}

The following definition is a simplification of Definition 2.9 in \cite{RiedelSteg17}, resulting from right-continuity of any $\alpha_i^\vartheta(\cdot)$ also where it takes the value zero. Right-continuity implies that when $\alpha_i^\vartheta(t)>0$, then infinitely many such ``atoms'' of approximately the same size follow in any arbitrarily short time interval. Therefore, define the functions $\mu_L$ and $\mu_M$ from $[0,1]^2\setminus (0,0)$ to $[0,1]$ by
\begin{align*}
\mu_L(x,y):=\frac{x(1-y)}{x+y-xy}\qquad\text{and}\qquad\mu_M(x,y):=\frac{xy}{x+y-xy},
\end{align*}
which resp.\ yield the probability that player $i$ stops first or that both players stop simultaneously in an infinitely repeated game where player $i$'s constant stage stopping probability is $x$ and player $j$'s $y$; $1-\mu_L(x,y)-\mu_M(x,y)=\mu_L(y,x)$ is then the probability of player $j$ stopping first. An interval of atoms can also meet an isolated (conditional) atom $\Delta G_j^\vartheta(t)/(1-G_j^\vartheta(t-))$ played just once. Special care is needed when atoms become arbitrarily small, because $\mu_L$ cannot be made continuous at the origin; see \cite{RiedelSteg17} for more details. Recall that $\alpha_i^\vartheta(t)>0\Rightarrow G_i^\vartheta(t)=1$ and that $(1-G_i^\vartheta(\hat\tau^\vartheta-))(1-G_j^\vartheta(\hat\tau^\vartheta-))$ is the probability that nobody stops before the extensions are used. For notational convenience, it is understood that if $(1-G_i^\vartheta(\hat\tau^\vartheta-))=0$, then $(1-G_i^\vartheta(\hat\tau^\vartheta-))/(1-G_i^\vartheta(\hat\tau^\vartheta-)):=0$.

\begin{definition}\label{def:outcome}
Given $\vartheta\in\T$ and a pair of extended mixed strategies $\bigl(G_1^\vartheta,\alpha_1^\vartheta\bigr)$ and $\bigl(G_2^\vartheta,\alpha_2^\vartheta\bigr)$, the \emph{outcome probabilities} $\lambda^\vartheta_{L,1}$, $\lambda^\vartheta_{L,2}$ and $\lambda^\vartheta_M$ at $\hat\tau^\vartheta:=\inf\{t\geq\vartheta\mid\alpha_1^\vartheta(t)+\alpha_2^\vartheta(t)>0\}$ are defined as follows. Let $i,j\in\{1,2\}$, $i\neq j$. 

\noindent
If $\hat\tau^\vartheta<\hat\tau_j^\vartheta:=\inf\{t\geq\vartheta\mid\alpha_j^\vartheta(t)>0\}$, then
\begin{align*}
\lambda^\vartheta_{L,i}:={}&\bigl(1-G_i^\vartheta(\hat\tau^\vartheta-)\bigr)\bigl(1-G_j^\vartheta(\hat\tau^\vartheta-)\bigr)\Biggl[1-\frac{\Delta G^\vartheta_j(\hat\tau^\vartheta)}{1-G^\vartheta_j(\hat\tau^\vartheta-)}\Biggr]=\Delta G_i^\vartheta(\hat\tau^\vartheta)\bigl(1-G_j^\vartheta(\hat\tau^\vartheta)\bigr),\displaybreak[0]\\
\lambda^\vartheta_M:={}&\bigl(1-G_i^\vartheta(\hat\tau^\vartheta-)\bigr)\bigl(1-G_j^\vartheta(\hat\tau^\vartheta-)\bigr)\alpha_i^\vartheta(\hat\tau^\vartheta)\frac{\Delta G^\vartheta_j(\hat\tau^\vartheta)}{1-G^\vartheta_j(\hat\tau^\vartheta-)}=\Delta G_i^\vartheta(\hat\tau^\vartheta)\Delta G_j^\vartheta(\hat\tau^\vartheta)\alpha_i^\vartheta(\hat\tau^\vartheta).
\end{align*}
If $\hat\tau^\vartheta<\hat\tau_i^\vartheta:=\inf\{t\geq\vartheta\mid\alpha_i^\vartheta(t)>0\}$, then
\begin{align*}
\lambda^\vartheta_{L,i}:={}&\bigl(1-G_i^\vartheta(\hat\tau^\vartheta-)\bigr)\bigl(1-G_j^\vartheta(\hat\tau^\vartheta-)\bigr)\frac{\Delta G^\vartheta_i(\hat\tau^\vartheta)}{1-G^\vartheta_i(\hat\tau^\vartheta-)}\bigl(1-\alpha_j^\vartheta(\hat\tau^\vartheta)\bigr)=\Delta G_i^\vartheta(\hat\tau^\vartheta)\Delta G_j^\vartheta(\hat\tau^\vartheta)\bigl(1-\alpha_j^\vartheta(\hat\tau^\vartheta)\bigr),\displaybreak[0]\\[8pt]
\lambda^\vartheta_M:={}&\bigl(1-G_i^\vartheta(\hat\tau^\vartheta-)\bigr)\bigl(1-G_j^\vartheta(\hat\tau^\vartheta-)\bigr)\frac{\Delta G^\vartheta_i(\hat\tau^\vartheta)}{1-G^\vartheta_i(\hat\tau^\vartheta-)}\alpha_j^\vartheta(\hat\tau^\vartheta)=\Delta G_i^\vartheta(\hat\tau^\vartheta)\Delta G_j^\vartheta(\hat\tau^\vartheta)\alpha_j^\vartheta(\hat\tau^\vartheta).
\end{align*}
If $\hat\tau^\vartheta=\hat\tau_1^\vartheta=\hat\tau_2^\vartheta$ and $\alpha_1^\vartheta(\hat\tau^\vartheta)+\alpha_2^\vartheta(\hat\tau^\vartheta)>0$, then
\begin{align*}
\lambda^\vartheta_{L,i}:={}&\bigl(1-G_i^\vartheta(\hat\tau^\vartheta-)\bigr)\bigl(1-G_j^\vartheta(\hat\tau^\vartheta-)\bigr)\mu_L(\alpha_i^\vartheta(\hat\tau^\vartheta),\alpha_j^\vartheta(\hat\tau^\vartheta)),\\[6pt]
\lambda^\vartheta_M:={}&\bigl(1-G_i^\vartheta(\hat\tau^\vartheta-)\bigr)\bigl(1-G_j^\vartheta(\hat\tau^\vartheta-)\bigr)\mu_M(\alpha_1^\vartheta(\hat\tau^\vartheta),\alpha_2^\vartheta(\hat\tau^\vartheta)).
\end{align*}
If $\hat\tau^\vartheta=\hat\tau_1^\vartheta=\hat\tau_2^\vartheta$ and $\alpha_1^\vartheta(\hat\tau^\vartheta)+\alpha_2^\vartheta(\hat\tau^\vartheta)=0$, then
\begin{align*}
\lambda^\vartheta_{L,i}:={}&\bigl(1-G_i^\vartheta(\hat\tau^\vartheta-)\bigr)\bigl(1-G_j^\vartheta(\hat\tau^\vartheta-)\bigr)\,\frac{1}{2}\!\begin{aligned}[t]
\biggl\{&\liminf_{\underset{\alpha_i^\vartheta(t)+\alpha_j^\vartheta(t)>0}{t\searrow\hat\tau^\vartheta}}\mu_L(\alpha_i^\vartheta(t),\alpha_j^\vartheta(t))\\
+&\limsup_{\underset{\alpha_i^\vartheta(t)+\alpha_j^\vartheta(t)>0}{t\searrow\hat\tau^\vartheta}}\mu_L(\alpha_i^\vartheta(t),\alpha_j^\vartheta(t))\biggr\},
\end{aligned}\\
\lambda^\vartheta_M:={}&0.
\end{align*}
\end{definition}

\begin{remark}
\noindent
\begin{enumerate}
\item
$\lambda^\vartheta_{L,j}$ is of course also the probability that player $i$ becomes follower. It always holds that $\lambda^\vartheta_M+\lambda^\vartheta_{L,i}+\lambda^\vartheta_{L,j}=(1-G_i^\vartheta(\hat\tau^\vartheta-))(1-G_j^\vartheta(\hat\tau^\vartheta-))$. Dividing by $(1-G_i^\vartheta(\hat\tau^\vartheta-))(1-G_j^\vartheta(\hat\tau^\vartheta-))$ if feasible yields the corresponding conditional probabilities. 

\item
If $\alpha_i^\vartheta(\hat\tau^\vartheta)=1$, then player $i$ stops for sure and no limit argument is needed. Otherwise, both $\alpha_1^\vartheta(\cdot)$, $\alpha_2^\vartheta(\cdot)$ are right-continuous and the corresponding limit of $\mu_M$ exists. $\mu_L$, however, has no continuous extension at the origin, whence we use the symmetric combination of $\liminf$ and $\limsup$, ensuring consistency whenever the limit does exist. If the limit in a potential equilibrium does not exist, then both players will be indifferent about the roles; see Lemma A.5 in \cite{RiedelSteg17}.
\end{enumerate}
\end{remark}
} 


\begin{thebibliography}{}

\bibitem[\protect\citeauthoryear{Bismut and Skalli}{Bismut and
  Skalli}{1977}]{BismutSkalli77}
Bismut, J.-M. and B.~Skalli (1977).
\newblock Temps d'arr{\^e}t optimal, th{\'e}orie g{\'e}n{\'e}rale des processus
  et processus de {M}arkov.
\newblock {\em Z.\ Wahrscheinlichkeitstheorie verw.\ Gebiete\/}~{\em 39},
  301--313.

\bibitem[\protect\citeauthoryear{Bulow and Klemperer}{Bulow and
  Klemperer}{1999}]{BulowKlemperer99}
Bulow, J. and P.~Klemperer (1999).
\newblock The generalized war of attrition.
\newblock {\em Am.\ Econ.\ Rev.\/}~{\em 89\/}(1), 175--189.

\bibitem[\protect\citeauthoryear{D{\'e}camps and Mariotti}{D{\'e}camps and
  Mariotti}{2004}]{DecampsMariotti04}
D{\'e}camps, J.-P. and T.~Mariotti (2004).
\newblock Investment timing and learning externalities.
\newblock {\em J.\ Econ.\ Theory\/}~{\em 118}, 80--102.

\bibitem[\protect\citeauthoryear{Dellacherie and Meyer}{Dellacherie and
  Meyer}{1978}]{DellacherieMeyer78}
Dellacherie, C. and P.-A. Meyer (1978).
\newblock {\em Probabilities and Potential}, Volume~29 of {\em Mathematics
  Studies}.
\newblock Amsterdam New York Oxford: North-Holland.

\bibitem[\protect\citeauthoryear{Dutta, Lach, and Rustichini}{Dutta
  et~al.}{1995}]{Duttaetal95}
Dutta, P.~K., S.~Lach, and A.~Rustichini (1995).
\newblock Better late than early: Vertical differentiation in the adoption of a
  new technology.
\newblock {\em J.\ Econ.\ Manage.\ Strategy\/}~{\em 4\/}(4), 563--589.

\bibitem[\protect\citeauthoryear{Dutta and Rustichini}{Dutta and
  Rustichini}{1993}]{DuttaRustichini93}
Dutta, P.~K. and A.~Rustichini (1993).
\newblock A theory of stopping time games with applications to product
  innovations and asset sales.
\newblock {\em Econ.\ Theory\/}~{\em 3}, 743--763.

\bibitem[\protect\citeauthoryear{Fudenberg and Tirole}{Fudenberg and
  Tirole}{1985}]{FudenbergTirole85}
Fudenberg, D. and J.~Tirole (1985).
\newblock Preemption and rent equalization in the adoption of new technology.
\newblock {\em Rev.\ Econ.\ Stud.\/}~{\em 52\/}(3), 383--401.

\bibitem[\protect\citeauthoryear{Fudenberg and Tirole}{Fudenberg and
  Tirole}{1986}]{FudenbergTirole86}
Fudenberg, D. and J.~Tirole (1986).
\newblock A theory of exit in duopoly.
\newblock {\em Econometrica\/}~{\em 54\/}(4), 943--960.

\bibitem[\protect\citeauthoryear{Ghemawat and Nalebuff}{Ghemawat and
  Nalebuff}{1985}]{GhemawatNalebuff85}
Ghemawat, P. and B.~Nalebuff (1985).
\newblock Exit.
\newblock {\em Rand J.\ Econ.\/}~{\em 16\/}(2), 184--194.

\bibitem[\protect\citeauthoryear{Grenadier}{Grenadier}{1996}]{Grenadier96}
Grenadier, S.~R. (1996).
\newblock The strategic exercise of options: {D}evelopment cascades and
  overbuilding in real estate markets.
\newblock {\em J.\ Finance\/}~{\em 51\/}(5), 1653--1679.

\bibitem[\protect\citeauthoryear{Hamad{\`e}ne and Hassani}{Hamad{\`e}ne and
  Hassani}{2014}]{HamadeneHassani14}
Hamad{\`e}ne, S. and M.~Hassani (2014).
\newblock The multi-player nonzero-sum {D}ynkin game in continuous time.
\newblock {\em SIAM J.\ Control Optim.\/}~{\em 52\/}(2), 821--835.

\bibitem[\protect\citeauthoryear{Hamad{\`e}ne and Zhang}{Hamad{\`e}ne and
  Zhang}{2010}]{HamadeneZhang10}
Hamad{\`e}ne, S. and J.~Zhang (2010).
\newblock The continuous-time nonzero-sum {D}ynkin game problem and application
  in game options.
\newblock {\em SIAM J.\ Control Optim.\/}~{\em 48\/}(5), 3659--3669.

\bibitem[\protect\citeauthoryear{Hendricks, Weiss, and Wilson}{Hendricks
  et~al.}{1988}]{Hendricksetal88}
Hendricks, K., A.~Weiss, and C.~Wilson (1988).
\newblock The war of attrition in continuous time with complete information.
\newblock {\em Int.\ Econ.\ Rev.\/}~{\em 29\/}(4), 663--680.

\bibitem[\protect\citeauthoryear{Hendricks and Wilson}{Hendricks and
  Wilson}{1992}]{HendricksWilson92}
Hendricks, K. and C.~Wilson (1992).
\newblock Equilibrium in preemption games with complete information.
\newblock In M.~Majumdar (Ed.), {\em Equilibrium and Dynamics: Essays in Honour
  of {D}avid {G}ale}, pp.\  123--147. Basingstoke, Hampshire: Macmillan.

\bibitem[\protect\citeauthoryear{Hoppe and Lehmann-Grube}{Hoppe and
  Lehmann-Grube}{2005}]{HoppeLehmann-Grube05}
Hoppe, H.~C. and U.~Lehmann-Grube (2005).
\newblock Innovation timing games: a general framework with applications.
\newblock {\em J.\ Econ.\ Theory\/}~{\em 121}, 30--50.

\bibitem[\protect\citeauthoryear{Kallenberg}{Kallenberg}{2002}]{Kallenberg02}
Kallenberg, O. (2002).
\newblock {\em Foundations of Modern Probability\/} (2nd ed.).
\newblock New York, Berlin, Heidelberg: Springer.

\bibitem[\protect\citeauthoryear{Kwon, Xu, Agrawal, and Muthulingam}{Kwon
  et~al.}{2016}]{Kwonetal16}
Kwon, H.~D., W.~Xu, A.~Agrawal, and S.~Muthulingam (2016).
\newblock Impact of bayesian learning and externalities on strategic
  investment.
\newblock {\em Management Sci.\/}~{\em 62\/}(2), 550--570.

\bibitem[\protect\citeauthoryear{Lambrecht and Perraudin}{Lambrecht and
  Perraudin}{2003}]{LambrechtPerraudin03}
Lambrecht, B. and W.~Perraudin (2003).
\newblock Real options and preemption under incomplete information.
\newblock {\em J.\ Econ.\ Dyn.\ Control\/}~{\em 27}, 619--643.

\bibitem[\protect\citeauthoryear{Laraki and Solan}{Laraki and
  Solan}{2013}]{LarakiSolan13}
Laraki, R. and E.~Solan (2013).
\newblock Equilibrium in two-player non-zero-sum {D}ynkin games in continuous
  time.
\newblock {\em Stochastics\/}~{\em 85\/}(6), 997--1014.

\bibitem[\protect\citeauthoryear{Laraki, Solan, and Vieille}{Laraki
  et~al.}{2005}]{Larakietal05}
Laraki, R., E.~Solan, and N.~Vieille (2005).
\newblock Continuous-time games of timing.
\newblock {\em J.\ Econ.\ Theory\/}~{\em 120}, 206--238.

\bibitem[\protect\citeauthoryear{Mason and Weeds}{Mason and
  Weeds}{2010}]{MasonWeeds10}
Mason, R. and H.~Weeds (2010).
\newblock Investment, uncertainty and pre-emption.
\newblock {\em Int.\ J.\ Ind.\ Organ.\/}~{\em 28\/}(3), 278--287.

\bibitem[\protect\citeauthoryear{Mertens}{Mertens}{1972}]{Mertens72}
Mertens, J.-F. (1972).
\newblock Th{\'e}orie des processus stochastiques g{\'e}n{\'e}raux:
  Applications aux surmartingales.
\newblock {\em Z.\ Wahrscheinlichkeitstheorie verw.\ Gebiete\/}~{\em 22},
  45--68.

\bibitem[\protect\citeauthoryear{Murto}{Murto}{2004}]{Murto04}
Murto, P. (2004).
\newblock Exit in duopoly under uncertainty.
\newblock {\em Rand J.\ Econ.\/}~{\em 35\/}(1), 111--127.

\bibitem[\protect\citeauthoryear{Pawlina and Kort}{Pawlina and
  Kort}{2006}]{PawlinaKort06}
Pawlina, G. and P.~M. Kort (2006).
\newblock Real options in an asymmetric duopoly: Who benefits from your
  competitive disadvantage?
\newblock {\em J.\ Econ.\ Manage.\ Strategy\/}~{\em 15\/}(1), 1--35.

\bibitem[\protect\citeauthoryear{Riedel and Steg}{Riedel and
  Steg}{2017}]{RiedelSteg17}
Riedel, F. and J.-H. Steg (2017).
\newblock Subgame-perfect equilibria in stochastic timing games.
\newblock {\em J.\ Math.\ Econ.\/}~{\em 72}, 36--50.

\bibitem[\protect\citeauthoryear{Smets}{Smets}{1991}]{Smets91}
Smets, F. (1991).
\newblock Exporting versus {FDI}: The effect of uncertainty, irreversibilities
  and strategic interactions.
\newblock Working paper, Yale University, New Haven, CT.

\bibitem[\protect\citeauthoryear{Steg}{Steg}{2018}]{Steg18}
Steg, J.-H. (2018).
\newblock Preemptive investment under uncertainty.
\newblock {\em Games Econ.\ Behav.\/}~{\em 110\/}(C), 90--119.

\bibitem[\protect\citeauthoryear{Steg}{Steg}{2017}]{Steg17}
Steg, J.-H. (\noop{3001}forthcoming, 2017).
\newblock On preemption in discrete and continuous time.
\newblock {\em Dyn.\ Games Appl.\/}.

\bibitem[\protect\citeauthoryear{Steg and Thijssen}{Steg and
  Thijssen}{2015}]{StegThijssen15}
Steg, J.-H. and J.~J.~J. Thijssen (2015).
\newblock Quick or persistent? {S}trategic investment demanding versatility.
\newblock Working Paper 541, Center for Mathematical Economics, Bielefeld
  University.

\bibitem[\protect\citeauthoryear{Thijssen}{Thijssen}{2010}]{Thijssen10}
Thijssen, J. J.~J. (2010).
\newblock Preemption in a real option game with a first mover advantage and
  player-specific uncertainty.
\newblock {\em J.\ Econ.\ Theory\/}~{\em 145}, 2448--2462.

\bibitem[\protect\citeauthoryear{Touzi and Vieille}{Touzi and
  Vieille}{2002}]{TouziVieille02}
Touzi, N. and N.~Vieille (2002).
\newblock Continuous-time {D}ynkin games with mixed strategies.
\newblock {\em SIAM J. Control Optim.\/}~{\em 41}, 1073--1088.

\bibitem[\protect\citeauthoryear{Weeds}{Weeds}{2002}]{Weeds02}
Weeds, H. (2002).
\newblock Strategic delay in a real options model of {R}\&{D} competition.
\newblock {\em Rev.\ Econ.\ Stud.\/}~{\em 69}, 729--747.

\end{thebibliography}
\newcommand{\noop}[1]{}

\end{document}